\documentclass[reqno,11pt]{amsart}
\usepackage{a4wide}
\usepackage{color,eucal,enumerate,mathrsfs,amsthm}
\usepackage[normalem]{ulem}
\usepackage{amsmath,amssymb,amsthm,epsfig,bbm,makecell}
\numberwithin{equation}{section}

\usepackage{amsfonts}
\usepackage{dsfont}
\usepackage{mathtools}  
\usepackage{leftidx}
\usepackage{graphicx}
\usepackage{esint}
\usepackage{multirow}

\usepackage[colorlinks,linkcolor={blue},citecolor={blue},urlcolor={black},pdfborder={0 0 0}]{hyperref}

\usepackage[latin1]{inputenc}
\usepackage[english]{babel}

\newtheorem{Definition}{Definition}[section]
\newtheorem{Proposition}[Definition]{Proposition}
\newtheorem{Lemma}[Definition]{Lemma}
\newtheorem{Theorem}[Definition]{Theorem}
\newtheorem{Corollary}[Definition]{Corollary}
\theoremstyle{definition}
\newtheorem{Example}[Definition]{Example}

\newtheorem{Remark}[Definition]{Remark}

\allowdisplaybreaks

\newcommand{\D}{\mathbb{D}}
\newcommand{\E}{\mathbb{E}}

\newcommand{\M}{\mathbb{M}}
\newcommand{\N}{\mathbb{N}}

\newcommand{\R}{\mathbb{R}}
\newcommand{\Z}{\mathbb{Z}}

\newcommand{\fD}{\mathfrak{D}}

\newcommand{\mm}{{\mbox{\boldmath$m$}}}

\newcommand{\sfd}{{\sf d}}

\newcommand{\sfg}{{\sf g}}

\newcommand{\sfL}{{\sf L}}

\newcommand{\sfM}{{\sf M}}

\newcommand{\cE}{\mathcal{E}}
\newcommand{\cK}{\mathcal{K}}

\newcommand{\cF}{\mathcal{F}}

                          \newcommand{\Kliminf}{K\kern-3pt-\kern-2pt\mathop{\rm lim\,inf}\limits}                  \renewcommand{\d}{{\mathrm d}}
\newcommand{\dt}{{\d t}}

\newcommand{\restr}[1]{\lower3pt\hbox{$|_{#1}$}}

\newcommand{\nchi}{{\raise.3ex\hbox{$\chi$}}}

\renewcommand{\mm}{\mathfrak m}

\renewenvironment{proof}{\removelastskip\par\medskip   \noindent{\em Proof.} \rm}{\penalty-20\null\hfill$\square$\par\medbreak}

\newtheorem{theorem}[Definition]{Theorem}

\newtheorem{lemma}[Definition]{Lemma}
\newtheorem{proposition}[Definition]{Proposition}

\newtheorem{definition}[Definition]{Definition}

\newtheorem{example}[Definition]{Example}

\newtheorem{remark}[Definition]{Remark}

\newcommand{\X}{{\rm X}}

\newcommand{\BE}{{\sf BE}}
\newcommand{\GE}{{\sf GE}}

\newcommand{\HS}{{\lower.3ex\hbox{\scriptsize{\sf HS}}}}

\newcommand{\Ric}{{\rm Ric}}

\newcommand{\Y}{{\rm Y}}

\newcommand{\K}{{\mathcal K}}
\newcommand{\KK}{\boldsymbol{\mathcal K}}

\begin{document}

\title{Tamed spaces -- Dirichlet spaces with distribution-valued Ricci bounds}

\author{Matthias Erbar}

\address{Fakult\"at f\"ur Mathematik, Universit\"at Bielefeld,\newline
  Postfach 100131, 33501 Bielefeld, Germany.}
\email{erbar@math.uni-bielefeld.de}

\author{Chiara Rigoni}
\address{Institut f\"ur Angewandte Mathematik, \newline
Endenicher Allee 60, Universit\"at Bonn, \newline
Germany}
\email{rigoni@iam.uni-bonn.de}

\author{Karl-Theodor Sturm}
\address{Institut f\"ur Angewandte Mathematik, \newline
Endenicher Allee 60, Universit\"at Bonn, \newline
Germany}
\email{sturm@uni-bonn.de}

\author{Luca Tamanini}
\address{CEREMADE (UMR CNRS 7534), Universit\'e Paris Dauphine PSL, Place du Mar\'echal de Lattre de Tassigny, 75775 Paris Cedex 16, France and INRIA-Paris, MOKAPLAN, 2 Rue Simone Iff, 75012, Paris, France.}
\email{tamanini@ceremade.dauphine.fr}

\thanks{This work has been funded by the Deutsche Forschungsgemeinschaft (DFG,
German Research Foundation) through the Hausdorff Center for
Mathematics under Germany's Excellence Strategy -- EXC-2047/1 --
390685813 and through CRC 1060 - projekt number 211504053 as well as
by the European Union through the ERC-AdG 694405 RicciBounds.}

\begin{abstract}
  We develop the theory of tamed spaces which are Dirichlet spaces
  with distribution-valued lower bounds on the Ricci curvature and
  investigate these from an Eulerian point of view. To this end we
  analyze in detail singular perturbations of Dirichlet form by a
  broad class of distributions. The distributional Ricci bound is then
  formulated in terms of an integrated version of the Bochner
  inequality using the perturbed energy form and generalizing the
  well-known Bakry-\'Emery curvature-dimension condition. Among other
  things we show the equivalence of distributional Ricci bounds to
  gradient estimates for the heat semigroup in terms of the
  Feynman-Kac semigroup induced by the taming distribution as well as
  consequences in terms of functional inequalities. We give many
  examples of tamed spaces including in particular Riemannian
  manifolds with either interior singularities or singular boundary
  behavior.
\end{abstract}

\maketitle
\tableofcontents
\bigskip

\section{Introduction}

 {\bf A.} {\bf Synthetic lower Ricci bounds} have proven to be a powerful concept for analyzing the geometry of singular spaces, solutions to PDEs in irregular or infinite-dimensional settings, and the evolution of Markov processes. The most prominent versions of such synthetic Ricci bounds
are the Eulerian formulation in the setting of  Dirichlet spaces by Bakry--\'Emery and  the Lagrangian formulation in the setting of  metric measure spaces by Lott--Villani  and Sturm.

Bakry and \'Emery, in their seminal paper \cite{BakryEmery85}, characterized synthetic lower Ricci bounds $K\in\R$ for a given strongly local Dirichlet space $(\X,\cE,\mm)$ in terms of the generalized Bochner inequality
\begin{equation}\label{BE-old} \Gamma_2(f)\ge K\cdot \Gamma(f).\end{equation}
Here $\Gamma$ denotes the carr\'e du champ associated with $\cE$ and $\Gamma_2$ the iterated carr\'e du champ.
For the canonical Dirichlet space with $\X=\sfM$,  $\mm=\text{vol}_\sfg$, and $\cE(f)=\frac12\int_M |\nabla f|^2\,\d\mm$ on a Riemannian manifold $(\sfM,\sfg)$  this reads as 
\[\frac12\Delta|\nabla f|^2-\langle\nabla f,\nabla\Delta f\rangle\ge K\cdot |\nabla f|^2,\]
which in turn is well known -- due to  Bochner's equality -- to be  equivalent to
\[\Ric_\sfg\ge K\cdot \sfg\;.\]

A synthetic notion of lower Ricci curvature bounds in the setting of
metric measure spaces based on optimal transport has been developed by
Lott and Villani and the third author in \cite{Lott-Villani09,Sturm06I, Sturm06II} leading to a
huge wave of research activities shaping a far reaching theory of
metric measure spaces with lower Ricci bounds. In particular,
Ambrosio, Gigli and Savar\'e in a series of seminal papers
\cite{AmbrosioGigliSavare11, AmbrosioGigliSavare11-2, AmbrosioGigliSavare-compact} developed a powerful first order calculus on such spaces
leading to natural notions of (modulus of the) gradient, energy
functional (called Cheeger energy), and heat flow. For so-called 
infinitesimally Hilbertian spaces the Cheeger energy is
quadratic and defines a Dirichlet form and (under minimal assumptions)
the Eulerian and Langrangian approaches to synthetic Ricci bounds have
been shown to be equivalent \cite{AmbrosioGigliSavare12, Erbar-Kuwada-Sturm13, AmbrosioMondinoSavare13-2}, providing in
particular a Bochner inequality for metric measure spaces. A huge
number of contributions by numerous authors have established many sharp
analytic and geometric results for such metric measure spaces
including e.g.~estimates for volume growth and diameter,
gradient estimates, transport estimates, Harnack inequalities,
logarithmic Sobolev inequalities, isoperimetric inequalities,
splitting theorems, maximal diameter theorems, and further rigidity
results, see e.g.~\cite{AES15, KopferSturm19, CavMon15, Gigli13, Ketterer2015a, Ketterer2015b, Erbar-Sturm17} and references
therein. Moreover, deep results on the local structure of metric
measure spaces with synthetic Ricci bounds have been obtained
recently \cite{mondino2019}, \cite{BrueSemola18b} and an impressive second
order calculus has been developed \cite{Gigli18}.
\medskip

{\bf B.}
The aim of the present work is to develop a generalization of the
concept of synthetic lower Ricci bounds that goes far beyond the
framework of uniform bounds.  Indeed, many important properties and
quantitive estimates which typically are regarded as consequences of
uniform lower Ricci bounds also hold true in much more general
settings.

Our notion of  {\bf ``tamed spaces''} will refer to Dirichlet spaces $(\X,\cE,\mm)$  which admit a distri\-bu\-tio\-n-valued lower Ricci bound, formulated as a canonical generalization of \eqref{BE-old}.
Roughly speaking, we are going to replace the constant $K$ in \eqref{BE-old} by a distribution $\kappa$ and to consider the inequality in distributional sense, that is, as
\begin{equation*}
\int_\X \varphi\, \Gamma_2(f)\,\d\mm\ge \big\langle \kappa, \varphi\,\Gamma(f)\big\rangle\qquad
\end{equation*}
for all sufficiently regular $f$ and $\varphi\ge0$.
(For the precise -- and slightly more restrictive -- formulation, see Definition \ref{def-tamed} below as well as  \eqref{2-Bochner}.)

The distributions $\kappa$ to be considered will lie in the class $\cF^{-1}_{\rm qloc}$. Here $\cF^{-1}$ denotes the dual space of the form domain $\cF=\fD(\cE)$ and $\cF^{-1}_{\rm qloc}$ denotes the class of $\kappa$'s for which there exists an exhaustion of $\X$ by quasi-open subsets $G_n\nearrow \X$ such that $\kappa$ coincides on each $G_n$ with some element in $\cF^{-1}_{G_n}$. (The option to exhaust $\X$ by quasi-open sets instead of exhausting it merely by open sets will lead to a significant enlargement of our scope. This will be important e.g. in Example (ii) below.)

\medskip

Already in the case of Riemannian manifolds, our new setting contains plenty of important examples which are not covered by any of the  concepts of ``spaces with uniform lower Ricci bounds''.
\begin{enumerate}[(i)]
\item  \emph{``Singularity of Ricci at $\infty$''}: Smooth Riemannian manifolds with Ricci curvature bounded from below in terms of a continuous -- but unbounded -- function  which  globally lies in the Kato class, see e.g.~recent results for such manifolds \cite{Rose-Stollmann17}, \cite{Braun-Guneysu20}.
\item \emph{``Local singularities of Ricci''}: Riemannian manifolds with (synthetic) Ricci curvature bounded from below in terms of a locally unbounded function  which lies in $L^p$ for some $p>n/2$.

Such ``singular'' manifolds for instance are obtained from smooth manifolds by ground state transformations (see e.g.~\cite{Guneysu-vonRenesse19}), conformal transformations, or time changes  with singular weight functions.

\item  \emph{``Singular Ricci induced by the boundary''}: Riemannian manifolds with boundary for which the  second fundamental form  is bounded from below in terms of a (possibly unbounded) function  which  lies in $L^p$ w.r.t.~the boundary measure for some $p>n-1$.
Such manifolds with boundaries in particular appear as  closed subsets of manifolds without boundaries.
\item  \emph{``Singular Ricci at the rim''}: Doubling of a Riemannian surface with boundary 
leads to a (nonsmooth) Riemannian surface which admits a uniform (synthetic) lower Ricci bound if and only if the initial surface has convex boundary.
\end{enumerate}
Indeed, however, out setting allows for much more examples.
\begin{enumerate}[(a)]
\item In each of the examples (i), (ii), and (iii), the bounds can be far more singular than Kato class functions. Our setting for instance allows for highly oscillating  bounds which are nowhere locally integrable. More generally, in (ii) it allows for measure-valued bounds and even for distribution-valued bounds. In particular, examples will be provided where these distributions can not be represented as signed measures. 

\item Instead of dealing with Riemannian manifolds, in each of the examples (i), (ii), and (iii),  we can deal with general metric measure spaces or (even slightly more general) with strongly local, quasi-regular Dirichlet forms.

\item Extending example (iii) to the setting of Dirichlet forms allows us to 
take into account
 curvature effects of the boundary for a detailed analysis of Neumann Laplacians and heat flows with reflecting boundary conditions.
 
 Even more, an analogous curvature concept (including the curvature effects of the boundary) will also be applied to the analysis of  Dirichlet Laplacians and heat flows with vanishing boundary conditions.
\end{enumerate}
\medskip

{\bf C.} 
We will formulate our synthetic lower Ricci bounds in the setting of {\bf Dirichlet spaces}. These spaces always will be assumed to be quasi-regular and strongly local and to admit a carr\' e du champ. Among the most prominent examples are the canonical Dirichlet spaces induced by infinitesimally Hilbertian {\bf metric measure spaces}. Indeed, 
defining the \emph{Cheeger energy} as
$$\cE(f)=\frac12\int_\X |\nabla f|^2\,\d\mm$$
in terms of the \emph{minimal weak upper gradient} $|\nabla f|$,
each such $(\X,\sfd,\mm)$ induces a Dirichlet space as above. 
To simplify our presentation, here in this Introduction we will not distinguish between semigroups acting on equivalence classes and semigroups defined pointwise or quasi-everywhere.

Given a Dirichlet space $(\X,\cE,\mm)$ and a distribution $\kappa\in\cF^{-1}_{\rm qloc}$, the crucial quantities to formulate our synthetic lower Ricci bound will be the {\bf taming energy}  $\cE^\kappa$ -- a singular zero-order perturbation of $\cE$ -- and the {\bf taming semigroup} $(P^\kappa_t)_{t\ge0}$. The latter allows for a straightforward definition via the Feynman-Kac formula as
  \begin{eqnarray}\label{FK-f}
  P^\kappa_{t}f(x):=\E_x\Big[e^{-A^{\kappa}_t}\, f(B_t)\Big]
   \end{eqnarray}
in terms of the stochastic process $\big({\mathbb P}_x, B_t\big)_{x\in \X, t\ge0}$  properly associated with $(\X,\cE,\mm)$ and in terms of the local continuous additive functional $(A^\kappa_t)_{t\ge0}$ associated with $\kappa$ (existence and uniquenss of which we will prove at Lemma \ref{caf}).
We say that the distribution $\kappa$ is {\bf moderate} if 
  \begin{equation*}
  \sup_{t\in[0,1]}\,   \sup_{x\in \X} 
  \E_x\Big[e^{-A^{\kappa}_t}\Big]<\infty.
  \end{equation*}
In this case, 
$(P^{\kappa}_t)_{t\ge0}$ defines a strongly continuous, exponentially bounded semigroup on $L^2(\X,\mm)$ and thus it generates a lower bounded, closed quadratic form $\big(\cE^{\kappa}, \fD(\cE^{\kappa})\big)$. The latter indeed can be identified (see Theorem \ref{relaxThm})  with the relaxation of the quadratic form 
 \begin{eqnarray*}\dot\cE^{\kappa}(f):=\cE(f)+\cE_1(\psi_n,f^2)  \end{eqnarray*}
defined on a suitable subset of $\bigcup_n\cF_{G_{n}}$ where $(G_n)_n$ denotes an exhaustions of $\X$ by quasi-open sets $G_{n}$ such that $\kappa\in \cF^{1}_{G_n}$ and where $\psi_n:=(-\sfL_{G_n}+1)^{-1}\kappa$.
We also provide a condition (see Theorem \ref{thm:condition-closable}) on $\kappa$ which guarantees that $\dot\cE^{\kappa}$ is closable, in which case $\cE^{\kappa}$ is its closure.

More generally, for $p\in\R_+$ we say that $\kappa\in\cF^{-1}_{\rm qloc}$ is \emph{$p$-moderate} if $p\,\kappa$ is moderate.
\medskip

\begin{definition}\label{def-tamed}
We say that a Dirichlet space $(\X,\cE,\mm)$ is {\bf tamed} if there exists a moderate distribution $\kappa\in \cF^{-1}_{\rm qloc}$ such that the following \emph{Bochner inequality}, briefly {\sf BE}$_1(\kappa,\infty)$, holds true:
\begin{equation}\label{1-Bochner}
 \cE^{\kappa/2}\big(\varphi,\Gamma(f)^{1/2}\big)+\int\varphi\,\frac1{\Gamma(f)^{1/2}}\Gamma\big(f,\sfL f\big)\,\d\mm\le0
 \end{equation}
for all $f$ and $\varphi\ge0$ in appropriate functions spaces (see Subsection~\ref{subsec:taming} for more details).

In this case, $\kappa$ will be called {\bf distribution-valued lower Ricci bound} or {\bf taming distribution} for the Dirichlet space $(\X,\cE,\mm)$. 
\end{definition}

\begin{Theorem} A moderate distribution $\kappa\in  \cF^{-1}_{\rm qloc}$ is taming  for the Dirichlet space $(\X,\cE,\mm)$ if and only if the following \emph{gradient estimate},  briefly {\sf GE}$_1(\kappa,\infty)$, holds true:
 \begin{equation}\label{1-grad-est}
 \Gamma(P_tf)^{1/2}\le  P^{\kappa/2}_t\big(\Gamma(f)^{1/2}\big)
 \end{equation}
 for all $f\in\cF$.
 \end{Theorem}
 
 Note that in the case of a constant $\kappa$, \eqref{1-grad-est}
 reads as
 $\Gamma(P_tf)^{1/2}\le e^{-\kappa t/2}P_t\big(\Gamma(f)^{1/2}\big)$
 which is the well-known, ``improved version'' of the gradient
 estimate in the Bakry-\'Emery theory.  As in the latter theory, the
 ``1-versions'' of Bochner inequality and gradient estimate imply the
 ``2-versions'', see Theorem \ref{thm:BE2GE2}, Proposition \ref{GE1toGE2} and
 Theorem \ref{thm:selfimp} below.

\begin{Theorem}\label{1nach2}  Let a Dirichlet space  $(\X,\cE,\mm)$ be given and a 2-moderate  $\kappa\in \cF^{-1}_{\rm qloc}$. 
\begin{enumerate}[$\blacktriangleright$]
\item Then the following properties are equivalent:
\begin{itemize}
\item[(i)] the  \emph{2-Bochner inequality}  {\sf BE}$_2(\kappa,\infty)$: \  $\forall f$ and $\varphi\ge0$ in appropriate  spaces,
\begin{equation}\label{2-Bochner}
 \cE^{\kappa}\big(\varphi,\Gamma(f)\big)+2\int\varphi\,\Gamma\big(f,\sfL f\big)\,\d\mm\le0;
 \end{equation}

\item[(ii)] the  \emph{2-gradient estimate} {\sf GE}$_2(\kappa,\infty)$: \  $\forall f\in\cF$,
 \begin{equation}\label{2-grad-est}
 \Gamma(P_tf)\le  P^{\kappa}_t\big(\Gamma(f)\big).
 \end{equation}
\end{itemize}
\item These properties follow from the corresponding ``improved'' versions \eqref{1-Bochner} and \eqref{1-grad-est}. 
\item The converse implication (the celebrated ``self-improvement'') holds if $\kappa$ is a signed measure such that  $\kappa^-$ satisfies the compatibility condition w.r.t.~$\kappa^+$.
\end{enumerate}
\end{Theorem}

The Bochner inequalities and   gradient estimates discussed so far are particular cases  (for $p=1,2$ and $N=\infty$)  of the more general Bochner inequalities  {\sf BE}$_p(\kappa,N)$ and  gradient estimates {\sf GE}$_p(\kappa,N)$ depending in addition  on a parameter $N\in [1,\infty]$, interpreted as synthetic upper bound on the dimension. For {\sf BE}$_2(\kappa,N)$, for instance, \eqref{2-Bochner} will be tightened to 
\begin{equation*}
 \cE^{\kappa}\big(\varphi,\Gamma(f)\big)+2\int\varphi\,\Gamma\big(f,\sfL f\big)\,\d\mm\le-\frac2N\big(\sfL f\big)^2
 \end{equation*}
and for {\sf GE}$_2(\kappa,N)$, \eqref{2-grad-est} will be tightend to 
$ \Gamma(P_tf)+\frac{2}{N}\int_0^t P_s^\kappa(\sfL P_{t-s}f)^2\,\d s \le P^{\kappa}_t\big(\Gamma(f)\big)$ (see also Theorem \ref{thm:BE2GE2} for different yet equivalent formulations). For these more general functional inequalities, the assertions of the previous Theorem hold true in analogous form.
\medskip

{\bf D.} 
Besides the fundamental gradient estimates, tamed spaces share many important properties with spaces which admit uniform lower Ricci bounds.
One of the crucial \emph{qualitative} properties is
\begin{Lemma} Assume that the Dirichlet space $(\X,\cE,\mm)$ is tamed by a signed measure $\kappa\in \cF^{-1}_{\rm qloc}$ which is in the extended Kato class $\KK_{1-}(\X)$. Then $\Gamma(f)^{1/2}\in\cF$ for each $f\in\fD(\sfL)$.
\end{Lemma}
This opens the door for defining Hessians and further objects of a second order calculus. A selection of important
\emph{quantitative} properties is listed below.
\begin{Theorem} Assume that the Dirichlet space $(\X,\cE,\mm)$ is tamed by a 2-moderate distribution $\kappa\in\cF^{-1}_{\rm qloc}$. Then the following functional inequalities hold true, say for $t\le 1$,
\begin{enumerate}[(i)]
\item  
\emph{Local Poincar\'e inequality:} \quad $P_t(f^2) - (P_t f)^2 \le Ct  \, P_t(\Gamma f)$;
\item
 \emph{Reverse local Poincar\'e inequality:} \quad $P_t(f^2) - (P_t f)^2 \ge  t/C \,
 \Gamma (P_{t} f)$;
\item 
 \emph{Local log-Sobolev inequality:}
\quad $P_t(f \log f) - P_t f \log (P_t f) \le \displaystyle\int_0^t   P_s P^\kappa_{t-s}\bigg(\dfrac{\,\Gamma f\,}{f}\bigg) \, \d s$;
\item  
 \emph{Reverse local log-Sobolev inequality:} 
 $P_t(f \log f) - P_t f \log (P_t f) \ge \displaystyle \int_0^t
\dfrac{\Gamma(P_t f)}{P^{\kappa/2}_s P_{t-s} f} \, \d s$.
\end{enumerate}
\end{Theorem}

We also could derive a remarkable \emph{\bf conservativeness
  criterion} which until recently was not known even in the
``classical'' setting of spaces with uniform lower Ricci bounds (more
precisely, neither for Dirichlet spaces with Ricci bounds in the sense
of Bakry-\'Emery nor for metric measure spaces with Ricci bounds in
the sense of Lott-Sturm-Villani). Recently, a similar result has been
obtained in \cite{Braun-Guneysu20} for smooth manifolds with Ricci
curvature bounded below by a function in the Kato (or more generally
Dynkin) class.

\begin{Theorem} Assume that the Dirichlet space $(\X,\cE,\mm)$ is tamed and {``intrinsically complete''} in the sense that
$\exists\, (\varphi_n)_n\subset\cF: \ 0\le\varphi\nearrow 1, \ 1\ge \Gamma(\varphi)\searrow0 \ \mm\text{-a.e.~on }\X$.
Then $(\X,\cE,\mm)$ is conservative. \end{Theorem}
\medskip

{\bf E.} Singular Ricci bounds occur especially if one wants to analyze diffusions on non-convex subsets of the state space of a given Dirichlet space (or metric measure space or Riemannian manifold).
Here both Neumann and Dirichlet boundary conditions are of interest.
Neumann boundary conditions are easier to treat since the resulting diffusions can be incorporated into the previous setting.

To simplify the presentation, let us focus now on the Riemannian
case. Let $(\hat\sfM,\sfg)$ be a Riemannian manifold -- as usual
complete and without boundary -- and let $\sfM$ be a closed
subset. Indeed, we do not assume that $\sfg$ is smooth but only
$\hat\sfM$ has a smooth differential structure, nor do we assume that
$\partial\sfM$ is smooth. Thus the Ricci tensor (if defined at all)
may have singularities inside of $\sfM$, and the same can happen with
the curvature of the boundary. For technical reasons, we will assume
that $\sfM$ is regularly exhaustible, i.e.~it can be exhausted by
domains with smooth boundary on which $\sfg$ is smooth and which have
some uniform control on the moderateness of the distributions induced
by Ricci and the boundary curvature, see Thm.~\ref{thm:stability}
below for a precise formulation.  Consider $(\sfM,\sfg)$ as a
\emph{Riemannian manifold with boundary}, put $\mm=\text{vol}_\sfg$,
and let $\sigma$ denote the surface measure of $\partial
\sfM$. Moreover, define
\[
\cE_\sfM(f):=\frac12\int_{\sfM^0} |\nabla f|^2\,\d\mm \qquad\text{with}\quad 
\fD(\cE_\sfM):=W^{1,2}(\sfM^0).
\]
Then we have the following result, see Thm.~\ref{thm:stability} below.
\begin{Theorem} 
  Assume that $k:\sfM^0\to\R$ is a lower bound on the Ricci curvature
  of $(\sfM^0,\sfg)$ (where defined) and that $\ell:\partial\sfM\to\R$
  is a lower bound for the second fundamental form of $\partial \sfM$
  (where defined). Moreover, assume that $(\sfM,\sfg)$ is regularly
  exhaustible, in particular that the distribution
\begin{equation}\label{c+c}\kappa:=k\,\mm+\ell\,\sigma\in \cF^{-1}_{\rm qloc}.\end{equation}
is moderate. Then the Dirichlet space $(\sfM,\cE_\sfM,\mm)$ satisfies
{\sf BE}$_1(\kappa,\infty)$.
\end{Theorem}
\begin{Corollary} In the setting of the previous Theorem, the Neumann heat semigroup $(P_t)_{t\geq 0}$ on $(\sfM,\sfg)$ satisfies 
$$|\nabla P_tf|\le P^{\kappa/2}_t|\nabla f|$$
with $(P^{\kappa/2}_t)_{t\geq 0}$ defined according to \eqref{FK-f} in terms of $(P_t)_{t \geq 0}$ and $\kappa$ from \eqref{c+c}.
\end{Corollary}

For smooth manifolds with boundary such a result was first proven by Hsu \cite{Hsu02a}.

We will provide several concrete examples of tamed manifolds with
singularities both in the interior or in the boundary in Section
\ref{sec:examples}.

Dealing with the Dirichlet heat semigroup is more sophisticated. No gradient estimate of the previous type will remain true if we impose  Dirichlet boundary conditions on the semigroups on both sides. Instead, the domination has to be based  on the Neumann heat semigroup.
\begin{Theorem} In the setting of the previous Theorem with $k\in{\mathcal C}_b(\sfM^0)$ and $\ell\in{\mathcal C}_b(\partial\sfM)$, the Dirichlet heat semigroup $(P^0_t)_{t\geq 0}$ on $(\sfM^0,\sfg)$ satisfies 
\[|\nabla P^0_tf|\le P^{\kappa/2}_t|\nabla f|\]
with $(P^{\kappa/2}_t)_{t \geq 0}$ as in the previous Corollary, that is, defined in terms of the Neumann heat semigroup $(P_t)_{t \geq 0}$ and $\kappa$ from \eqref{c+c}, provided either $n=2$ or $\ell\ge0$.
\end{Theorem}
\medskip

{\bf F.} The {\bf structure} of the present work is the following. {\bf Section 2} is devoted to a detailed and comprehensive investigation of singular zero-order perturbations of Dirichlet forms, which will play a crucial role in the definition of distribution-valued synthetic lower Ricci bounds. After reviewing basic notions of Dirichlet forms theory, we attach a Feynman-Kac semigroup to any quasi-local distribution and single out those (called ``moderate distributions'' in the sequel) for which such semigroup is exponentially bounded in $L^\infty$. It is then extremely important the bridge between moderate distributions and, in the terminology of \cite{Chen-Fuku}, ``smooth in the strict sense'' measures, as this allows us to define the perturbed energy form $\cE^\kappa$ associated to a moderate distribution $\kappa$ by sophisticated approximation and relaxation procedures.

The next four sections, representing the core of the paper, deal with definition, examples and properties of tamed spaces. More precisely:

\begin{itemize}
\item relying on the good class of moderate distributions singled out in Section \ref{sec:perturb}, in Section \ref{sec:tamed spaces} we introduce the taming condition for Dirichlet spaces as an $L^1$-Bochner inequality for the perturbed energy form and, using the semigroup approach of Bakry-\'Emery theory as blueprint, we characterize this condition in terms of an $L^1$-gradient estimate. The equivalence between the $L^2$-versions of Bochner inequality and gradient estimate is also established, as well as the implication $\GE_1 \Rightarrow \GE_2$, thus providing a preliminary Eulerian picture of tamed spaces;

\item in Section \ref{sec:examples} we provide the reader with a sample of motivating and diversified examples which show that our distributional approach to synthetic Ricci bounds comes to embed highly irregular spaces ruled out by previous theories. In this sense, the singularities covered by the taming condition concern both the behaviour of the curvature in the interior of the space and the roughness of the boundary;

\item aim of Section \ref{sec:funct-ineq} is to deduce suitable ``tamed'' versions of local (reverse) Poincar\'e inequality and local (reverse) logarithmic Sobolev inequality;

\item for an even stronger parallelism between the by-now classical Bakry-\'Emery setting and the tamed one, in Section \ref{sec:self-improve} we prove that for moderate distributions given by signed measures in the (extended) Kato class the taming condition is self-improving, in the spirit of \cite{Savare13}. The strategy of proof follows indeed Savar\'e's contribution, but caveats and technical difficulties are numerous and the arguments do not immediately carry over.
\end{itemize}

In the final Section \ref{sec:sub-tamed} we introduce the notion of
``sub-tamed space'' as a generalization of the taming condition. The
main motivation behind this is the fact that semigroups with Dirichlet
boundary conditions fail to be tamed spaces, yet they may be
sub-tamed. Following the arguments in Section 3 and 5 it is not
difficult to see that sub-tamed spaces share the same (properly
modified) properties of tamed ones. Moreover, we show that to check
whether a Dirichlet space is sub-tamed it is sufficient to verify the
taming condition for the ``doubled'' Dirichlet space associated to
it. We conclude the discussion by proving that the doubling of a
compact Riemannian surface with boundary is a tamed space with taming
distribution expressed in terms of pointwise lower bounds for the
Ricci curvature on the interior of the surface and for the second
fundamental form on the boundary.

\section{Singular Zero-Order Perturbations of Dirichlet Forms}\label{sec:perturb}
The goal of this chapter is to study perturbations of Dirichlet spaces
$(\X,\cE,\mm)$ by singular zero-order terms. These zero-order
perturbations will be given in terms of distributions
$\kappa$ which are locally --- or just quasi-locally --- in the dual space of
the form domain $\cF$. Indeed, the extension from
$\cF^{-1}$ to $\cF^{-1}_{\rm
  qloc}$ will be of fundamental importance. For
instance, this approach includes all perturbations by signed measures which are
smooth in the strict sense. It also includes perturbations by
distributions which can not be represented by signed Radon
measures. 
The initial Dirichlet forms will always be assumed to be strongly local and quasi-regular. 

\bigskip

\subsection{The (Unperturbed) Dirichlet Form}\label{sec:UnDir}

Throughout this chapter, we fix a strongly local, quasi-regular
Dirichlet space $(\X,\cE,\mm)$. That is, $\X$ is a topological Lusin
space, $\mm$ is a Borel measure with full topological support on $\X$,
and $\cE$ is a quasi-regular, strongly local Dirichlet form on
$L^2(\X,\mm)$ with domain $\cF=\fD(\cE)$.  Moreover, we assume that
the Dirichlet space admits a carr\'e du champ. That is, there exists a
symmetric bilinear map $\Gamma: \cF\times \cF\to L^1(\X,\mm)$
satisfying the Leibniz rule
\[
\Gamma(fg,h)=f\Gamma(g,h)+g\Gamma(f,g)\qquad (\forall f,g,h\in\cF\cap L^\infty(\X,\mm))
\]
such that
\[
\cE(f,g)=\frac12\int_\X \Gamma(f,g)\,\d\mm\qquad (\forall f,g\in\cF).
\]

The generator of the Dirichlet form $(\cE,\cF)$ will be denoted by
$(\sfL,\fD(\sfL))$. The associated resolvent and semigroup on
$L^2(\X,\mm)$ will be denoted by $(G_\alpha)_{\alpha>0}$ and
$(T_t)_{t\ge0}$, resp., such that formally
\[
G_\alpha=(\alpha-\sfL)^{-1}=\int_0^\infty e^{-\alpha t}\,T_t\, \dt, \qquad T_t=e^{\sfL t}
\]
on $L^2(\X,\mm)$. The latter extends to a positivity preserving,
$\mm$-symmetric, bounded semigroup $(T_t)_{t \geq 0}$ on each
$L^p(\X,\mm)$ with $\big\| T_t\big\|_{L^p, L^p}\le 1$ for each
$p\in [1,\infty]$ and strongly continuous on $L^p(\X,\mm)$ if
$p<\infty$.
  
All ``quasi''-notions in the sequel are understood w.r.t.~the fixed
initial Dirichlet form $\cE$. Quasi-regularity of $\cE$ implies that
each $f\in \cF$ admits a quasi-continuous version $\tilde f$ (and two
such versions coincide q.e.~on $\X$). Thus in particular, for each
$f\in\bigcup_{p\in[1,\infty]} L^p(\X,\mm)$ and $t>0$, there exists a
quasi-continuous version $\tilde T_tf$ of $T_tf$ (uniquely determined
q.e.).
 
We also fix an $\mm$-reversible, continuous, strong Markov process
$\big({\mathbb P}_x, B_t\big)_{x\in \X, t\ge0}$ (with life time
$\zeta$) which is properly associated with $\cE$ in the sense that
\begin{equation}\label{eq:properly associated semigroup}
\tilde T_tf=P_tf\ \, \mm\text{-a.e.~on }\X, \qquad\forall \text{ Borel function }f\in L^2(\X,\mm),
\end{equation}
see \cite[Theorems 1.5.2, 1.5.3 and 3.1.13]{Chen-Fuku}. Here and in
the sequel, $(R_\alpha)_{\alpha>0}$ and $(P_t)_{t\ge0}$ denote the
resolvent and semigroup, resp., induced by the Markov process
$\big({\mathbb P}_x, B_t\big)_{x\in \X, t\ge0}$. That is,
\begin{equation*}
P_{t}f(x):=
\E_x\big[f(B_{t})\big], \qquad R_\alpha f(x):=\E_x\Big[\int_0^\infty e^{-\alpha t}f(B_t)\, \d t\Big]
\end{equation*}
where, following the convention in \cite{Chen-Fuku}, we assume that $f(B_t) \equiv 0$ whenever $t \ge \zeta$.

In the following we will denote for a Borel function $f \colon X\to\R_+$
  \[\text{\rm q-\!}\sup_x f(x):=\inf\Big\{\sup_{x\in \X\setminus N}f(x): \ N  \text{ is $\cE$-polar}\Big\},\]
  \[\mm\text{-\!}\sup_x f(x):=\inf\Big\{\sup_{x\in \X\setminus N}f(x): \ \mm(N)=0  \Big\}=\big\|f\big\|_{L^\infty(\X,\mm)}.\]

 \bigskip

 For developing the concept of tamed spaces, Riemannian manifolds are
 our most important source of inspiration. Unless explicitly stated
 otherwise, {\bf Riemannian manifolds} are always assumed to be {\bf
   smooth}, {\bf complete} and {\bf without boundary}.
 
\begin{Example}\label{ex-r1} Every Riemannian manifold $(\sfM,\sfg)$ defines in a canonical way a Dirichlet space
  $\big(\sfM,\cE_\sfM,\mm)$. The canonical choice is
  $\mm:=\text{\rm vol}_\sfg$ and
\[
\cE_\sfM(f):=\frac12\int_\sfM |\nabla f|^2\,\d\mm,\quad \fD(\cE_\sfM):=W^{1,2}_0(\sfM).
\]
This Dirichlet space is always quasi-regular and strongly local and it
admits a carr\'e du champ, namely, $\Gamma(f)=|\nabla f|^2$.

Moreover, thanks to the completeness of $\sfM$, we always have
$W^{1,2}_0(\sfM)=W^{1,2}(\sfM)$.  However, the Dirichlet space will
not necessarily be conservative --- unless the Ricci curvature of
$(\sfM,\sfg)$ is bounded from below.
\end{Example}

\begin{Example} The construction in Example \ref{ex-r1} applies
  without any change also to \emph{incomplete manifolds}.  (Only the
  assertions on equality of Sobolev spaces and on conservativeness no
  longer hold.) The crucial point is that the form domain is chosen to
  be $W^{1,2}_0(\sfM)$, which in a certain sense means that
  \emph{Dirichlet boundary conditions} are incorporated.

Typically, incomplete manifolds appear by restricting a manifold to an \emph{open subset} $D \subset\sfM$.
Then 
\[\fD(\cE_{D}):=\Big\{f\in   \fD(\cE_{\sfM}): \ \tilde f=0 \text{ q.e.~on }\sfM\setminus D\Big\}=W^{1,2}_0(D)\]
and $\cE_{D}:=\cE_\sfM$ on $\fD(\cE_D)$.
The Dirichlet space $\big(D,\cE_D,\mm_D)$ will satisfy our basic assumptions (quasi-regularity, strong locality, existence of carr\'e du champ) without any regularity assumption on $\partial D$. (Indeed, one can even
extend this construction to quasi-open sets $D\subset \sfM$.)
\end{Example}

\begin{Example} Typically, manifolds with boundary appear by restricting a manifold to a \emph{closed subset} $F \subset\sfM$.

For a Riemannian manifold $(\sfM,\sfg)$ \emph{with boundary}, there are two ``canonical'' constructions of a Dirichlet space $\big(\sfM,\cE_\sfM,\mm)$. 
\begin{itemize}
\item The Dirichlet space for the metric measure space $(\sfM,\sfd,\mm)$ (see Example \ref{dir-cheeger} below) with $\sfd$ denoting the complete length geodesic metric induced by $\sfg$ on $\sfM$:
\[
\cE_\sfM(f):=\frac12\int_{\sfM} |\nabla f|^2\,\d\mm,\quad \fD(\cE_\sfM):=W^{1,2}(\sfM).
\]
\item The reflected Dirichlet space  (see Section \ref{reflDirspace}) for the Dirichlet space $\big(\sfM^0,\cE_{\sfM^0},\mm|_{\sfM^0}\big)$ 
associated with the incomplete manifold $(\sfM^0,\sfg)$ according to the previous Example:
\[
\cE_\sfM(f):=\frac12\int_{\sfM^0} |\nabla f|^2\,\d\mm,\quad \fD(\cE_\sfM):=W^{1,2}(\sfM^0).
\]
\end{itemize}
They will coincide if $\sfM$  has a Lipschitz boundary or more generally if $W^{1,2}(\sfM)=W^{1,2}(\sfM^0)$
(but not in general, see  \cite[Remark 6.1 and Example 6.2]{Sturm2019}).
Indeed, unless explicitly stated otherwise, we always assume that  a manifold with boundary  has a smooth  boundary.
\end{Example}

\begin{Example}
There are many ways to construct new Dirichlet spaces out of the Dirichlet space $\big(\sfM,\cE_\sfM,\mm)$ by means of a weight function $\psi\in L^\infty_{\rm loc}(\sfM)$. The most important transformations are
\begin{itemize}
\item \emph{Time change:}
\[\cE'_\sfM(f):=\frac12\int_\sfM |\nabla f|^2\,\d\mm, \qquad \big\|f\big\|^2_{L^2(\mm')}=\int_\sfM|f|^2\,e^{2\psi} \,\d\mm.\]
\item \emph{Drift transformation} or \emph{change of measure:}
\[\cE^\sharp_\sfM(f):=\frac12\int_\sfM |\nabla f|^2\,e^{2\psi}\,\d\mm, \qquad \big\|f\big\|^2_{L^2(\mm^\sharp)}=\int_\sfM|f|^2\,e^{2\psi} \,\d\mm.\]
\item \emph{Conformal transformation:} 
\[\cE^*_\sfM(f):=\frac12\int_\sfM |\nabla f|^2\,e^{(n-2)\psi}\,\d\mm, \qquad \big\|f\big\|^2_{L^2(\mm^*)}=\int_\sfM|f|^2\,e^{n\psi} \,\d\mm.\]
The latter is nothing but the 
Dirichlet space for the (not necessarily smooth) Riemannian manifold $(\sfM,\sfg^*)$ with $\sfg^*:=e^{2\psi}\sfg$.
\end{itemize}
\end{Example}

\begin{Example}\label{dir-cheeger}
Let $(\X,\sfd,\mm)$ be a complete and separable metric space, equipped with a non-negative Radon measure $\mm$. We introduce the convex and l.s.c. \emph{Cheeger energy} (\cite{Cheeger00, AmbrosioGigliSavare11})
\[
{\sf Ch} (f) := \inf \bigg\{  \liminf_{n \to \infty} \dfrac{1}{2} \int_\X |{\rm D} f_n|^2 \, \d \mm \, : \, f_n \in {\rm Lip}_b(\X),  \; f_n \to f \in L^2(\X, \mm)   \bigg\},
\]
where the metric slope $|{\rm D} f|$ of a Lipschitz function $f \colon \X \to \R$ is defined by $|{\rm D} f| (x) := \limsup_{y \to x} |f(y) - f(x)|/\sfd(x, y)$. We observe that the domain $\fD({\sf Ch}) := \{ f \in L^2(\X, \mm) : {\sf Ch}(f) < \infty  \}$ is a dense linear subspace of $L^2(\X, \mm)$. For any $f \in \fD({\sf Ch})$ the collection
\[
S(f) := \bigg\{  G \in L^2(\X, \mm) \, : \, \exists (f_n)_{n \in \N} \subset {\rm Lip}_b(\X), \; f_n \to f,  \; |{\rm D} f_n| \rightharpoonup G \, \text{ in } \, L^2(\X, \mm) \bigg\}
\]
admits a unique element of minimal norm, the \emph{minimal weak upper gradient} $|{\rm D} f|_w$, which is minimal also with respect to the order structure (see \cite{AmbrosioGigliSavare11}), i.e.,  $|{\rm D} f|_w \in S(f)$ and $ |{\rm D} f|_w \le G$ $\mm$-a.e. for every $G \in S(f)$. Hence, it is possible to represent ${\sf Ch}(f)$ in terms of $ |{\rm D} f|_w$, as
\[
{\sf Ch}(f) = \dfrac{1}{2} \int_\X  |{\rm D} f|_w^2 \, \d \mm.
\]
If ${\sf Ch}$ is a quadratic form in $L^2(\X,\mm)$, we say that $(\X,\sfd,\mm)$ is \emph{infinitesimally Hilbertian} (\cite{Gigli12}). 

In particular, according to \cite[Theorem 4.1]{Savare13}, if $(\X,\sfd,\mm)$ is an ${\rm RCD}(K,\infty)$ space for some $K \in \R$, meaning that it is an infinitesimally Hilbertian space with a bound from below on the Ricci curvature in the sense of Lott-Sturm-Villani (\cite{Lott-Villani09, Sturm06I}), then $\mathcal E := 2 {\sf Ch}$ is a strongly local and quasi-regular Dirichlet form which admits carr\'e du champ
\[
\Gamma(f) = |{\rm D} f|_w^2 \qquad \text{for every} \,f \in \fD({\sf Ch}).
\]
\end{Example}

\subsection{Feynman-Kac Semigroups Induced by Local Distributions}
 \subsubsection{First-Order Distributions}
 
Let $\cF^{-1}$ denote the dual space of $\cF$ and observe that 
\[
\psi\mapsto(-\sfL+1)\psi
\]
defines an isometry between $\cF$ and $\cF^{-1}$ with inverse given by $\kappa\mapsto(-\sfL+1)^{-1}\kappa$.

\begin{Example} 
Assume that $\X$ is locally compact and that the Dirichlet form $\cE$ is regular. Then every Radon measure $\mu$ of finite energy integral  (in the notation of \cite{FOT} and \cite{Chen-Fuku}, $\mu\in S_0$) defines -- or can be interpreted as -- a distribution $\kappa\in \cF^{-1}$ via
\[
\langle \kappa,\varphi\rangle:=\int \tilde\varphi\,\d\mu, \qquad\forall \varphi\in  \cF.
\]
Conversely, every non-negative distribution $\kappa\in \cF^{-1}$ is given by a Radon measure of finite energy integral.
\end{Example}
 
\begin{Lemma}\label{e&u}
For each $\kappa\in \cF^{-1}$ there exists a unique  \emph{continuous additive functional} $A^\kappa=(A^\kappa_t)_{t\ge0}$ associated with $\kappa$ given by 
\[
A^\kappa_t:=\int_0^t \psi(B_s)\, \d s+\frac12\Big(M^\psi_t+\hat M^\psi_t\Big),\qquad t<\zeta,
\]
provided $\kappa=(-\sfL+1)\psi$ for some quasi-continuous $\psi\in \cF$. Here $M^\psi$ denotes the martingale additive functional in the Fukushima decomposition  
\[
\psi(B_t)-\psi(B_0)=M^\psi_t+N^\psi_t,\qquad t<\zeta,
\]
w.r.t.~$\big({\mathbb P}_x, B_t\big)_{x\in \X, t\ge0}$ and $\hat M^\psi$ the correponding functional w.r.t.~the time-reversed process such that in the Lyons-Zheng decomposition
\[
\psi(B_t)-\psi(B_0)=\frac12M^\psi_t-\frac12\hat M^\psi_t,\qquad t<\zeta,
\]
see \cite[Theorems 4.2.6 and 6.7.2]{Chen-Fuku}. (Uniqueness of $A^\kappa$ is up to equivalence of  additive functionals.)
\end{Lemma}

  \begin{remark}
  In the particular case where $\kappa=f$ for some nearly Borel function   $f\in L^2(\X,\mm)$,
  \[A^\kappa_t=\int_0^t f(B_s)\,\d s,\qquad t<\zeta.\]
  \end{remark}
  
 \bigskip
 
The previous concepts can easily be restricted to a quasi-open set $G\subset \X$ by considering the Dirichlet form with Dirichlet boundary conditions on $\X\setminus G$ (or in other words, the process killed at the exit from $G$). More precisely, given a quasi-open set $G\subset \X$, we put
\[
\cF_G:=\Big\{\varphi\in \cF \text{ with }\tilde\varphi=0 \ \text{q.e.~on } \X \setminus G\Big\} \;.
\]
Its dual space will be denoted by $\cF^{-1}_G$. Let $\sfL_G$ denote the generator of the Dirichlet form $\big(\cE,\cF_G\big)$. The isometry $(-\sfL_G+1)^{1/2}: \cF_G\to L^2(G)$ extends to an isometry $(-\sfL_G+1)^{1/2}: L^2(G)\to \cF^{-1}_G$. Thus
\[
(-\sfL_G+1): \cF^{}_G\to \cF^{-1}_G
\]
is also an isometry.
 
Existence and uniqueness of continuous additive functionals $A^\kappa$ associated with $\kappa\in\cF^{-1}_G $ hold as formulated before in Lemma \ref{e&u} but now of course with the life time $\zeta$ of the process $({\mathbb P}_x, B_t)_{x\in \X, t\ge0}$ replaced by $\zeta_G:=\zeta\wedge \tau_G$. In particular, the \emph{resolvent with Dirichlet boundary condition} $R_{G,\alpha}=(-\sfL_G+\alpha)^{-1}: \  \cF^{-1}_G\to \cF^{}_G$ is given by
\[
R_{G,\alpha}\kappa(x):=\E_x\Big[\int_0^{\zeta\wedge \tau_G}e^{-\alpha t}\,\d A^\kappa_t\Big]\qquad  \text{for $\mm$-a.e. }x\in \X.
\]

 \subsubsection{Local Distributions}
 
 Given two quasi-open sets $G\subset G'\subset \X$, the obvious inclusion $\cF^{}_{G}\subset \cF^{}_{G'}$ implies that
 $$\cF^{-1}_{G}\supset \cF^{-1}_{G'}.$$
 Given any increasing sequence $(G_n)_n$ of quasi-open sets in $\X$, we define
 $$\cF^{-1}\Big({(G_n)_n}\Big):=\bigcap_n \cF^{-1}_{G_n}.$$

  \begin{Lemma}
 For each $\kappa\in \cF^{-1}\big({(G_n)_n}\big)$  there exists a unique local \emph{continuous additive functional}   $A^\kappa=(A^\kappa_t)_{t\ge0}$ associated with $\kappa$.
  It is the limit of the additive functionals $A^{\kappa,n}$ associated with $\kappa$ regarded as element of $\cF^{-1}_{G_n}$ for each $n\in\N$:
  \[ A_t^\kappa=A_t^{\kappa,n}\qquad \text{for }t<\tau_{G_n}\wedge\zeta\]
   and thus in particular $A_t^\kappa=\lim_{n\to\infty}A_t^{\kappa,n}$ for $t<\zeta$. 
  \end{Lemma}

\begin{Definition}\label{def:Pk}  Let $\kappa\in  \cF^{-1}\big((G_n)_n\big)$.
The \emph{Feynman-Kac semigroup} $(P^\kappa_t)_{t\ge0}$ associated with $\kappa$ is given by 
  \begin{eqnarray*}
  P^\kappa_{t}f(x)&:=&\E_x\Big[e^{-A^{\kappa}_t}\, f(B_t)\ \mathds{1}_{\{t<\zeta\}}\Big]\\
  &=&\uparrow\lim_{n\to\infty}\E_x\Big[e^{-A^{\kappa_n}_t}\, f(B_t)\ \mathds{1}_{\{t<\tau_{G_n}\wedge\zeta\}}\Big]
   \end{eqnarray*}
      for non-negative nearly Borel functions $f$ on $X$. For given $t$ and $x$, it is extended by 
   $P^\kappa_{t}f(x):=P^\kappa_{t}f^+(x)-P^\kappa_{t}f^-(x)$
   to arbitrary nearly Borel functions $f=f^+-f^-$ for which 
   $P^\kappa_{t}|f|(x)<\infty$.\end{Definition}
A \emph{quasi-open nest} is an increasing sequence of quasi-open sets $G_n\subset \X$ such that $\X\setminus \bigcup_n G_n$ is $\cE$-polar (or equivalently, that $\bigcup_n{\mathcal F}_{G_n}$ is dense in $\mathcal F$). Without restriction, we always may assume that $\bigcup_n G_n=\X$. We say that $\kappa$ lies quasi-locally in $\cF^{-1}$ if $\kappa\in \cF^{-1}\big((G_n)_n\big)$ for some quasi-open nest $(G_n)_n$ and we put
\[
\cF^{-1}_{\rm qloc}:=\bigcup_{\text{quasi-open nests } (G_n)_n}\cF^{-1}\Big({(G_n)_n}\Big).
\]
  
  \begin{Lemma}\label{caf}
   For each $\kappa\in \cF^{-1}_{\rm qloc}$, the local {continuous additive functional} 
  $A^{\kappa} = \big(A^{\kappa}_t\big)_{t\ge0}$ associated according to the previous Lemma with $\kappa$ and some  quasi-open nest $(G_n)_n$ does not depend on the choice of the nest (up to equivalence of local continuous additive functionals as introduced in \cite[p. 226]{FOT}).
 Thus also the semigroup $(P^\kappa_t)_{t\ge0}$ does not depend on the choice of the nest.
 
 It defines a semigroup on the space of non-negative, nearly Borel functions on $\X$. For each $t\ge0$, the operator  $P^\kappa_t$ is symmetric w.r.t.~$\mm$ and it maps $\mm$-equivalence classes onto $\mm$-equivalence classes. It extends to a bounded linear operator on $L^p(\X,\mm)$ provided $\|P^\kappa_t\|_{L^p,L^p}<\infty$ where
 \[\|P^\kappa_t\|_{L^p,L^p}:=\sup\Big\{ \big\|P_t^\kappa f\big \|_{L^p}: f\in L^p(\X,\mm),\, f\ge0,\, \big\|f\big\|_{L^p}\le1\Big\}.\]
   \end{Lemma}
  
  \begin{proof} Given two quasi-open nests $(G'_n)_n$ and $(G''_n)_n$, put $G_n:=G_n'\cap G_n''$. Then also $(G_n)_n$ is a quasi-open nest and $A^{\kappa}$ is uniquely defined on this nest. Thus it is unique.
  
  The semigroup property of $(P^\kappa_t)_{t\ge0}$ follows from the (local) additivity of $A^{\kappa}$. The symmetry w.r.t.~$\mm$ follows from the same property for the heat operator $P_t$ and from the fact that by construction $A^{\kappa}_t$ is invariant w.r.t.~time reversal. Invariance w.r.t.~$\mm$-equivalence follows from the same property for the heat operator $P_t$. Finally, the norm estimate and the extendability to $L^p$ follows from the simple fact that
  $|P^\kappa_tf|\le P^\kappa_t|f|$.
  \end{proof}

\begin{Example}
Following \cite[p. 227]{FOT}, and \cite[p. 163]{Chen-Fuku}, let $\dot\cF_{\rm loc}$ or $\cF_{\rm qloc}$ denote the set of $\mm$-equivalence classes of functions  which are locally in ${\mathcal F}$ in the broad sense. That is, $\psi\in \dot\cF_{\rm loc}$ if there exist  an increasing sequence $(G_n)_n$ of quasi-open sets  such that $\bigcup_n G_n=\X$ (or, equivalently, nearly Borel finely open sets  such that $\X\setminus \bigcup_n G_n$ is $\cE$-polar)   and a sequence $(\psi_n)_n$ in $\cF$ such that $\psi_n=\psi$ $\mm$-a.e.~on $G_n$, for each $n$. 

For $\psi\in \dot{\mathcal F}_{\rm loc}$, we define the distribution $\kappa=(-\sfL+1)\psi$ by testing against $\bigcup_n \cF_{G_n}$.
Then
\[
\psi\in \dot{\mathcal F}_{\rm loc}\qquad\Longrightarrow\qquad \psi, \sfL \psi, (-\sfL+1)\psi\in \cF^{-1}_{\rm qloc}.
\]
Indeed, for each $\varphi\in \cF_{G_n}$
\begin{eqnarray*}
\langle \kappa, \varphi\rangle&=&\int_{G_n}\Big(\frac12\Gamma(\psi,\varphi)+\psi\,\varphi\Big)\,\d\mm\\
&=&\int_{G_n}\Big(\frac12\Gamma(\psi_n,\varphi)+\psi_n\,\varphi\Big)\,\d\mm\le C\cdot \big\|\varphi\big\|_{\cF_{G_n}}
\end{eqnarray*}
Therefore $\kappa\in \cF^{-1}_{G_n}$ and thus $\kappa\in \cF^{-1}_{\rm qloc}$. The claim for $\sfL\psi$ follows analogously.
\end{Example}

 \subsubsection{Moderate Distributions}

\begin{Definition} 
The distribution $\kappa\in  \cF^{-1}_{\rm qloc}$ is called \emph{moderate}, briefly  $\kappa\in \cF_{\rm qloc,mod}^{-1}$, iff \begin{equation}\label{moder}
\sup_{t\in[0,1]}\, \sup_{x\in \X} \E_x\Big[e^{-A^{\kappa}_t}\Big]<\infty,
\end{equation}
where $A^{\kappa}$ is extended by 0 for $t\ge\zeta$. We say that  $\kappa$ is $p$-moderate for $p\in(0,\infty)$ if $p\kappa$ is moderate.
\end{Definition}

  \begin{Remark}\label{rem:moderate}
  (i) A distribution $\kappa$ is moderate if and only if the associated Feynman-Kac semigroup $(P^\kappa_t)_{t\ge0}$ defines an exponentially bounded semigroup on $L^\infty$, in the sense that
  \begin{equation}\label{moder2}
 \big \|P^\kappa_t1\big\|_{L^\infty, L^\infty}\le C\, e^{Ct}.
  \end{equation}
  
  (ii)
 For $\kappa\in \cF^{-1}_{\rm qloc,mod}$, the Feynman-Kac semigroup  $(P^\kappa_t)_{t\ge0}$  extends to an exponentially bounded   semigroup on $L^p(\X,\mm)$ for each $p\in[1,\infty]$.
   Moreover, for each $q\in(1,\infty)$
  \[\big|P^\kappa_tf\big|^q(x)\le P^{q\kappa}_t\big(|f|^q\big)(x).\]
  (The right-hand side of the last formula is finite provided $\kappa$ is $q$-moderate and $f\in L^q(\X,\mm)$. Otherwise, it is still well defined but might be $+\infty$.)

  (iii)
  If $\kappa_1$ is $p_1$-moderate and $\kappa_2$ is $p_2$-moderate then $\kappa:=\kappa_1+\kappa_2$ is $p$-moderate for $\frac1p=\frac1{p_1}+\frac1{p_2}$. In particular, if $\kappa$ is $p$-moderate then it also $q$-moderate for each $q\in (0,p]$.
  More generally, the set of moderate distributions is closed under convex combinations. 
    \end{Remark}
  
  \begin{proof}
   (ii)
  Put $C_t:=\sup_{x\in \X} P^\kappa_t\mathds{1}(x)$, which for sufficiently small $t>0$ will be finite according to \eqref{moder2}. Then obviously $\|P^\kappa_t\|_{L^\infty,L^\infty}=C_t$ and by symmetry $\|P^\kappa_t\|_{L^1,L^1}\le C_t$. Thus by interpolation $\|P^\kappa_t\|_{L^p,L^p}\le C_t$ for each $p\in(1,\infty)$. 
  
  (iii)
  \[\E_x\Big[e^{-A^{\sum_i \alpha_i \kappa_i}_t}\Big]=\E_x\Big[\prod_ie^{-\alpha_i\, A^{\kappa_i}_t}\Big]
  \le\prod_i \E_x\Big[e^{-A^{\kappa_i}_t}\Big]^{\alpha_i}\]
  for each sequence of positive numbers $\alpha_i$ with $\sum_i\alpha_i=1$.

  \end{proof}

  \begin{example}\label{1not2} Let $(\X,\cE,\mm)$ denote the classical Dirichlet space on $\R^n$, $n\ge2$. 
  
  (i) For $m,k>0$ put
  \[V(x)=k\,|x|^{-2-2m}\cdot\Big[2\sin^+\big(|x|^{-m}\big)-\sin^-\big(|x|^{-m}\big)\Big].\]
  Then there exists $k_c\in \big[\frac18 m^2, \frac94 m^2\big]$ such that
  \begin{itemize}
  \item $V$ is moderate for  $k\in(0,k_c)$;
  \item $V$ is not moderate for $k>k_c$.
  \end{itemize}
  In particular, for $k=\frac23 k_c$ the function $V$ is moderate but not 2-moderate.
  
    (ii) Similarly, for
  $$V(x)=k\,|x|^{-2-2m}\cdot\Big[2\sin\big(|x|^{-m}\big)+1\Big],$$
there exists $k_c\in (0,\infty)$ such that
  $V$ is moderate for  $k\in
  \big[\frac2{\sqrt3-1} m^2, \frac94 m^2\big]$ and  not moderate for $k>k_c$.
  \end{example}
  
  \begin{proof} (i) According to \cite[Theorem 1.4]{Sturm92HS}, the function $V$ is moderate if $k<\frac18 m^2$ and it is not moderate if $k>\frac94 m^2$. According to the previous Remark, moderateness for some $k$ implies moderateness for $k'\in(0,k)$. This proves the existence of a critical $k_c$ within the given bounds.
  
  (ii) Since $2\sin(r)+1\ge  2\sin^+(r)-\sin^-(r)$ for all $r\in\R$, moderateness of the potential in (i) implies moderateness of the potential in (ii). 
  
  To prove the unboundedness of $P^V_t$ in the case of sufficiently large $k$, we follow the argumentation from \cite[Theorem 3.1]{Sturm92HS},   
  now with $r_n:=\big[(2n-\frac14)\pi\big]^{-1/m}$, $R_n:=\big[(2n-\frac34)\pi\big]^{-1/m}$, $k_n^-:=k\,(\sqrt3-1)\, R_n^{-2-2m}$
  and thus with 
  \[\lambda_n\approx\big[-k\,(\sqrt3-1)+2m^2\big]\,R_n^{-2-2m}\]
  which diverges to $-\infty$ if $k>\frac{2m^2}{\sqrt 3-1}$.
  \end{proof}
    
\begin{Proposition} We define $W^{-1,\infty}(\X)$ to be the dual space of
\[
W^{1,1+}(\X):=\Big\{f\in L^1(\X,\mm): f_{[n]}\in \cF \ (\forall n\in\N)  \text{ and } \sup_{n \in \N} \big\| |f_{[n]}| +|Df_{[n]}|\big\|_{L^1}<\infty\Big\},
\]
where $f_{[n]}:=(f\wedge n)\vee(-n)$ denotes the truncation of $f$ at levels $\pm n$. Then 
\[
W^{-1,\infty}(\X)\subset \cF_{\rm qloc,mod}^{-1}.
\]
\end{Proposition}
  
\begin{proof} 
We refer to \cite[Section 2.1]{Sturm2019} for a proof of this result.
\end{proof}
  
\subsection{Jensen and H\"older Inequalities}
\label{sec:Jensen}

Let us recall that, by Definition \ref{def:Pk}, the Schr\"odinger semigroup associated with a quasi-local distribution $\kappa\in \cF^{-1}_{\rm qloc}$ is given by
\begin{align*}
 P_t^{\kappa} f(x) = \mathbb E_x\Big[e^{-A^{\kappa}_{t}}f(B_{t})\Big]
\end{align*}
for any bounded function $f$, and $f(B_t) \equiv 0$ whenever $t \ge \zeta$.

Let us denote
\begin{align}\label{eq:Holder-const}
C^{\kappa}_t:=\sup_{x\in\X}P^{\kappa}_t\mathds{1}(x) = \sup_{x\in\X}\mathbb E_x\Big[e^{-A^{\kappa}_{t}}\Big]\;.
\end{align}
Then $\kappa$ is $q$-moderate by definition if and only if
\begin{align*}
\sup_{t\in[0,1]}C^{q\kappa}_t<\infty\;.  
\end{align*}

\begin{Lemma}[H\"older estimates]\label{lem:Holder}
  Let $\kappa\in \cF^{-1}_{\rm qloc}(\X)$ be moderate. If $\kappa$ is also $q$-moderate for $q\in (0,\infty)$, then we have for any non-negative $f$ (with $p=q/(q-1)$):
  \begin{align}\label{eq:Holder1}
    |P^{\kappa}_tf|&\leq \big(C_t^{q\kappa}\big)^{1/q} \big(P_tf^p\big)^{1/p}\;.
  \end{align}
  If $-\kappa$ is $\frac{q}{p}$-moderate, then we have
  \begin{align}\label{eq:Holder2}
    |P_tf|&\leq \Big(C_t^{-\frac{q}{p}\kappa}\Big)^{1/q} \big(P^{\kappa}_tf^p\big)^{1/p}\;.    
  \end{align}
\end{Lemma}

\begin{proof}
  By H\"older's inequality, we have
  \begin{align*}
    |P^{\kappa}_tf(x)| &= \Big|\mathbb E_x\Big[e^{-A^{\kappa}_t}f(B_{t})\Big]\Big|
    \leq \mathbb E_x\Big[e^{-qA^{\kappa}_{t}}\Big]^{1/q}
                       \cdot\mathbb E_x\Big[f(B_{t})^p\Big]^{1/p}\\
    &\leq \Big(C_t^{{q}\kappa}\Big)^{1/q}\Big(P_tf^p\Big)^{1/p}\;.
  \end{align*}
  The second statement follows from
\begin{align*}
  |P_tf(x)| &= \big|\mathbb E_x\big[f(B_{t})\big]\big|
    \leq \mathbb E_x\Big[ \, e^{\frac{q}{p}A^{\kappa}_{t}}\,\Big]^{1/q}
                       \cdot\mathbb E_x\Big[e^{-A^{\kappa}_{t}}f(B_{t})^p\Big]^{1/p}\\
    &\leq \Big(C_t^{-\frac{q}{p}\kappa}\Big)^{1/q}\Big(P^\kappa_tf^p\Big)^{1/p}\;.
  \end{align*}
\end{proof}

\begin{lemma}[Jensen inequality]\label{lem:Jensen}
Let $\kappa\in \cF^{-1}_{\rm qloc}(\X)$ be moderate and let $\Phi:\R^d\to[0,\infty]$ be convex and $1$-homogeneous. Then, for any bounded functions $f_1,\dots,f_d$ we have:
  \begin{align}\label{eq:Jensen}
    \Phi\big(P_t^{\kappa} f_1(x),\dots,P_t^{\kappa} f_d(x)\big) \leq P^{\kappa}_t\big(\Phi(f_1,\dots f_d)\big)(x)\;.
  \end{align}
\end{lemma}

\begin{proof}
  This follows immediately from the $1$-homogeneity and convexity of
  $\Phi$ by applying Jensen's inequality with the normalized expectation
 \[\tilde{\mathbb E}_x(\,\cdot\,):=\mathbb E_x\Big[e^{-A^{\kappa}_{t}}\Big]^{-1}\mathbb E_x\Big[e^{-A^{\kappa}_{t}}(\,\cdot\,)\Big]\;.\]
\end{proof}

\subsection{Kato and Dynkin Classes} 
  
  Recall that, as introduced in Section \ref{sec:UnDir}, $(T_t)_{t\ge0}$ denotes the semigroup on $L^2(\X,\mm)$ associated with the Dirichlet form $\cE$, while $(P_t)_{t\ge0}$ denotes the transition semigroup for the diffusion process associated with $\cE$.
  The respective Laplace transforms (called resolvents) will be denoted by $(G_\alpha)_{\alpha>0}$ and $(R_\alpha)_{\alpha>0}$, resp. Let $f \colon \X \to \R$ be any Borel function, then $P_t f$ is a quasi-continuous version of $T_t f$,  for each $t > 0$, while $R_\alpha f$ is a quasi-continuous version of $G_\alpha f$, for each $\alpha > 0$
 (see \cite[Proposition 3.1.9]{Chen-Fuku}).
  
    \begin{Lemma} For a nearly Borel function $f\colon \X\to\R$ and a number $\rho>0$, the following are equivalent:
  \begin{eqnarray*}
   \lim_{t\to0}\ \mm\text{-\!}\sup_x
 \int_0^tT_s|f|(x)\,\d s\le \rho,\\
\lim_{t\to0}\ \text{\rm q-\!}\sup_x
 \int_0^tP_s|f|(x)\,\d s\le \rho,\\
 \lim_{\alpha\to\infty}\ \mm\text{-\!}\sup_x
 G_\alpha|f|(x)\le \rho,\\
\lim_{\alpha\to\infty}\ \text{\rm q-\!}\sup_x
 R_\alpha|f|(x)\le \rho.
   \end{eqnarray*}

    \end{Lemma}

\begin{Definition} For $\rho>0$, the  extended Kato class $\K_{\rho}(\X)$ consists of those nearly Borel functions $f\colon \X\to\R$ that satisfy the equivalent properties of the previous Lemma.

Moreover, we put  
\[
\K_{0}(\X):=\bigcap_{\rho>0}\K_{\rho}(\X), \quad\K_{1-}(\X):=\bigcup_{\rho<1}\K_{\rho}(\X), \quad \K_{\infty}(\X):=\bigcup_{\rho>0}\K_{\rho}(\X).
\]
$\K_{0}(\X)$ is called \emph{Kato class} and $\K_{\infty}(\X)$ is called \emph{Dynkin class}.
\end{Definition}

\begin{Definition}   
We say that a signed measure $\mu$ on $\X$ belongs to the extended Kato class, $\mu\in\KK_\rho(\X)$, iff $\mu$ does not charge $\cE$-polar sets and 
  \begin{equation}\label{kato-meas}
  \lim_{t\to0}\ \text{\rm q-\!}\sup_x \E_xA_t^\mu
\le \rho\end{equation}
where $A_t^\mu$ denotes the positive continuous additive functional (PCAF, for short) associated with $|\mu|$. 
  \end{Definition}
  
  \begin{Remark}\label{rem:kato-equiv} Each of the following conditions is equivalent to \eqref{kato-meas}:
 \begin{equation*}
\lim_{\alpha\to\infty}\ \text{\rm q-\!}\sup_x
 U^\alpha_A\mathds{1}\le \rho
 \end{equation*}
 where  $U^\alpha_A\mathds{1}(x):=\E_x\big[\int_0^\zeta e^{-\alpha t} \, \d A_t \big]$, $x \in \X \setminus N$, denotes the $\alpha$-potential of the PCAF $A=A^\mu$ associated with $|\mu|$
 (cf. \cite{Chen-Fuku}, (4.1.4));
\begin{equation*}
\lim_{\alpha\to\infty}\ \big\| U_\alpha\mu\big\|_{L^\infty}\le \rho
\end{equation*}
 where  $U_\alpha\mu \in \cF$ denotes the  $\alpha$-potential of the measure $|\mu|$ (cf. \cite{Chen-Fuku}, (2.3.6)).
  \end{Remark}
  
\begin{remark} 
  Assume that the absolute continuity hypothesis holds. That is, the
  semigroup $(P_t)_{t\ge0}$ is given in terms of a symmetric heat
  kernel $(p_t(x,y))_{t\ge0, x,y\in \X}$ and the resolvent
  $(R_\alpha)_{\alpha>0}$ admits a density given by
  $r_\alpha(x,y)=\int_0^\infty e^{-\alpha t} p_t(x,y)\, \d t$. For a
  measure $\mu$ on $\X$ define
  $P_t\mu(x):=\int_\X p_t(x,y)\,\d \mu(y)$ and
  $R_\alpha\mu(x):=\int_\X r_\alpha(x,y)\,\d \mu(y)$. Then
\[
\mu\in\KK_\rho(\X) \quad \Longleftrightarrow \quad \lim_{t\to0}\ \text{\rm q-\!} \sup_x
 \int_0^tP_s\mu(x)\,\d s\le \rho \quad \Longleftrightarrow \quad \lim_{\alpha\to\infty} \ \text{\rm q-\!} \sup_x R_\alpha\mu(x)\le \rho.
\]
\end{remark}
  
  \begin{Lemma}\label{Khas} For all $\rho<1$, it holds:
  \begin{itemize}
  \item[(i)] $ \text{\rm q-\!}\sup_x\E_x \big[A^\mu_t\big]\le \rho \quad \Rightarrow\quad  \text{\rm q-\!}\sup_x\E_x \big[e^{A^\mu_t}\big]\le \frac1{1-\rho}$,
   \item[(ii)] $\|U_\alpha\mu\|_{L^\infty}\le \rho\quad \Rightarrow\quad \big\|R^{A^{-\mu}_t}_\alpha\big\|_{L^\infty, L^\infty}\le \frac1{\alpha(1-\rho)}$,
  \end{itemize}
  where, for a Borel function $f \colon \X \to \R$ and for the PCAF $A = A^\mu$ associated with $|\mu|$, we define $R_\alpha^{A} f(x) := \mathbb{E}_x \bigg[\displaystyle \int_0^\infty e^{-\alpha t} e^{-A^\mu_t} f(B_t) \, \d t \bigg]$ (cf. \cite{Chen-Fuku}, (4.1.5)).
  \end{Lemma}
  
  \begin{proof} These are well-known facts. (i) is the celebrated
    Khasminskii Lemma. For the reader's convenience, let us briefly
    sketch the proof of (ii).  Appropriate generalizations of the
    resolvent identity yield (cf. \cite[Exercise 4.1.2]{Chen-Fuku})
\[R_\alpha f = R_{\alpha}^{A^{-\mu}_t} f - U^\alpha_{A^\mu_t}\big(R^{A^{-\mu}_t}_\alpha f\big)=\Big(I-U^\alpha_{A^\mu_t}(\, . \,) \Big) R_{\alpha}^{A^{-\mu}_t} f = \Big(I-U_\alpha(\, .  \cdot  \mu \,) \Big) R_{\alpha}^{A^{-\mu}_t} f, \]
using the fact that $U_A^\alpha f$ is the quasi-continuous version of $U_\alpha(f \cdot \mu)$ (see \cite[Lemma 4.1.5]{Chen-Fuku}). This in turn implies
\[\big\|R_{\alpha}^{A^{-\mu}_t} f\big\|_{L^p,L^p}\le \Big(I-\big\|U_\alpha\big(. \cdot \mu\big) \big\|_{L^p,L^p}\Big)^{-1}\cdot \big\|R_{\alpha}\big\|_{L^p,L^p}\]
for each $p\in[1,\infty]$, provided that $\big\|U_\alpha\big(. \cdot \mu\big) \big\|_{L^p,L^p}<1$.
  \end{proof}
  
  \begin{Corollary}\label{Khas-Cor} For each $\mu\in\KK_\rho(\X)$ and each $\rho'>\rho$ there exists $\alpha'\in\R$ such that for all $f$   \[\int_\X f^2\,\d\mu\le \rho'\,\cE(f)+\alpha'\,\int_\X f^2\,\d\mm.\]
  \end{Corollary}
  \begin{proof} Given $\mu$ and $\rho'$ as above, put $\mu':=\frac1{\rho'}\mu$. Then $\mu'\in\KK_{\rho^*}(\X)$ with $\rho^*=\frac \rho{\rho'}<1$. Thus 
  \[\big\|R_{\alpha}^{A^{-\mu'}_t}\big\|_{L^p,L^p}<\infty\]
  for sufficiently large $\alpha$ which  implies
 $\big\|R_{\alpha}^{A^{-\mu'}_t}\big\|_{L^2,L^2}<\infty.$ This in turn implies
  \[\cE(f)-\int_\X f^2\,\d\mu'+\alpha\int_\X f^2\,\d\mm\ge0\quad (\forall f)\]
  which can be rewritten as
  $\int_\X f^2\,\d\mu\le \rho'\cE(f)+\alpha'\int_\X f^2\,\d\mm$
  with $\alpha':=\rho'\,\alpha$.
  \end{proof}
  
  \begin{Lemma} Every finite measure $\mu\in\KK_\infty(\X)$ defines (or can be interpreted as) a distribution $\kappa\in \cF^{-1}$ via
  \[\langle \kappa,\varphi\rangle:=\int \tilde\varphi\,\d\mu,\qquad\forall \varphi\in  \cF.\]
  \end{Lemma}
  
\begin{proof} 
$\mu\in\KK_\infty(\X)$ implies that the $\alpha$-potential $U_\alpha\mu$ is (essentially) bounded for some $\alpha>0$. Let $\tilde U_\alpha\mu$ denote its quasi-continuous version. Note that $\mu$ does not charge $\cE$-polar  sets. Thus according to \cite[Theorem 2.2.2]{FOT},
\[
\int_\X \tilde\varphi\,\d\mu=\cE_\alpha(\varphi,U_\alpha\mu) \le C\cdot\|\varphi\|_{\cF}
\]
since $\cE_\alpha(U_\alpha\mu)=\int \tilde U_\alpha\mu\,\d\mu<\infty$.
\end{proof}

\subsubsection{Examples on $\R^n$}

 For the subsequent results, let $(\X,\cE,\mm)$ denote the classical Dirichlet space on $\X=\R^n, n\ge1$.  Then  $p_t(x,y)=(2\pi t)^{-n/2}\, \exp(-|x-y|^2/2t)$ is the heat kernel, and the $\alpha$-potential is given by
 $R_\alpha\mu(x)=\int_{\R^n}\int_0^\infty e^{-\alpha t}\, p_t(x,y)\,\d t\,\d\mu(y)$
 for any measure $\mu$ on $\R^n$ and any number $\alpha>0$. If $n\ge3$, the same formula with $\alpha=0$ will be used to define $R_0\mu(x)$, which yields 
 $$R_0\mu(x)=c_n\, \int_{\R^n}|x-y|^{2-n}\,\d\mu(y)$$
 with $c_n=\frac{\Gamma(n/2-1)}{2\pi^{n/2}}$.
 In the case  $n=2$, we define instead
 $R_0\mu(x)=\frac1{4\pi}\,\int_{\R^n}\log(1/|x-y|)\,\d\mu(y)$
 provided the latter is well defined. In their seminal paper, 
 Aizenman and Simon \cite{Aiz-Sim} 
 derived the following powerful characterization.   
  
  \begin{Lemma} For each $n\ge2$ and each $\rho>0$:
  \[\mu\in\KK_\rho(\R^n)\quad\Longleftrightarrow\quad \lim_{r\to0}\sup_{x\in\R^n} R_0\Big(\mathds{1}_{B_r(x)}\mu\Big)(x)  <\rho.\]
  \end{Lemma}
  This immediately also yields an analogous characterization for functions in $\K_0(\R^n)$.
  
  \begin{Corollary} If $f\in L^p(\R^n)$ with $p>n/2$, then $f\in\K_0(\R^n)$.
  \end{Corollary}
  
 From  \cite[Corollary 4.8]{SturmTh}, we quote the following useful criterion (together with its proof).
 
 \begin{Lemma}\label{unif-crit} Let $\mu\ge0$ be a measure on $\X=\R^n, n\ge1$. If $R_\alpha\mu$ is bounded and uniformly continuous on $\R^n$ for some $\alpha\ge0$ (with $\alpha>0$ if $n\le 2$), then $\mu$
   is in the Kato class $\KK_0(\R^n)$.
 \end{Lemma}
 
 \begin{proof} Let us first note that $P_tf\to f$ uniformly on $\R^n$ as $t\to0$ for each bounded and uniformly continuous $f \colon \R^n\to\R$. Indeed, given such an $f$ and $\epsilon>0$ there exist $\delta>0$ and $t>0$ such that
 $|f(x)-f(y)|\le\epsilon$ for all $x,y$ with $|x-y|\le \delta$ and such that $P_s\mathds{1}_{\R^n\setminus B_\delta(x)}(x)\le\epsilon$ for all $x$ and all $s\le t$. Thus
\[
\Big| P_sf(x)-f(x)\Big| \le P_t\Big(\mathds{1}_{B_\delta(x)}\-\big|f(x)-f(y)\big|\Big)(x) +2\big\|f\big\|_{L^\infty}\cdot P_t\mathds{1}_{\R^n\setminus B_\delta(x)}(x) \le \epsilon\, \big(1+2\big\|f\big\|_{L^\infty}\big).
\]
Also note that (as a consequence of the previous) $\beta R_\beta f\to f$ uniformly on $\R^n$ as $\beta\to\infty$.
 
 Now assume that  $R_\alpha\mu$ is bounded and uniformly continuous. By the resolvent equation and the previous observation, we obtain 
 \[R_{\alpha+\beta}\mu=R_\alpha\mu-\beta\,R_\beta\big(R_\alpha\mu\big)\to 0\]
uniformly $\R^n$  on  as $\beta\to\infty$.
 \end{proof}
 
 \begin{Corollary} Let $\X=\R^n$, $n\ge1$. Then for each $z\in\R^n$ and $r>0$, 
 the uniform distribution on the sphere,
 \[\mu=\sigma_{\partial B_r(z)},\]
 is in the Kato class $\KK_0(\R^n)$.
 \end{Corollary}

 \begin{proof} Assume $n\ge3$ or $\alpha>0$. Then the $\alpha$-potential $R_\alpha\mu$ is bounded and Lipschitz continuous. Indeed, the maximum of $R_\alpha\mu$ is attained on 
 $\partial B_r(z)$, and 
 \[R_\alpha\mu(x)=r_\alpha(0,x)\]
 for $x\in \R^n\setminus B_r(z)$.
  \end{proof}
 Note that   in the case $n\ge3$, for all $x\in\R^n$
 \[R_0\mu(x)=c_n (r\vee |x|)^{2-n}.\]
 
 \begin{Corollary} Let $\X=\R^n$, $n\ge1$. Then for each  $r\in(0,\infty)$, 
 \[\mu=\sum_{z\in\Z^n}\sigma_{\partial B_r(z)}\]
 is in the Kato class $\KK_0(\R^n)$.
 \end{Corollary}
 
 \begin{proof}
 By the maximum principle, $R_\alpha\mu$ attains its maximum on $\bigcup_z \partial B_r(z)$. Hence, by translation invariance the maximum is attained on  $ \partial B_r(0)$. For $x\in\partial B_r(0)$,
 \begin{equation}\label{g,a,z}
 R_\alpha\mu(x)=\sum_{z\in\Z^n} r_\alpha(0,x+z).
 \end{equation}
For $\alpha>0$, the latter sum is bounded  since $r_\alpha(0,y)\approx\exp(-C_\alpha\cdot |y|)$ as $y\to\infty$.

On each compact subset  $K\subset \R^n$, the previous series \eqref{g,a,z} converges uniformly. Thus $R_\alpha\mu$ is uniformly continous on $K$. By invariance w.r.t. translations in $\Z^n$, therefore, $R_\alpha\mu$ is uniformly continuous on $\R^n$.
 \end{proof}

\begin{Corollary} 
Let $\mu=\sum_{z\in\Z^n}\sigma_{\partial B_r(z)}$ as in the previous Corollary and put $\X_0:=\R_+\times \R^{n-1}$. Then
\[
\mu_0:=\mu\big|_{X_0}
\]
is in the Kato class $\KK_0(\X_0)$ (w.r.t. reflected Brownian motion).
\end{Corollary}
 
\begin{proof} 
If $R^0_\alpha $ denotes the $\alpha$-Green operator w.r.t. to reflected Brownian motion, then $R^0_\alpha(\mu_0)=(R_\alpha\mu)\big|_{\X_0}$.
\end{proof}
 
\subsubsection{Harnack-type Dirichlet spaces}
 
Let $(\X,\cE,\mm)$ be a Harnack-type Dirichlet space in the sense of
Gyrya and Saloff-Coste \cite{GyrSal}. That is, $\cE$ is stricly local and
regular, its intrinsic distance $\rho$ induces the original topology
of $\X$, $(\X,\rho)$ is a complete metric space and the volume
doubling condition and a scale invariant Poincar\'e inequality on
balls hold, see \cite[Def.~2.29, Thm.~2.31]{GyrSal}. Assume in
addition that the Dirichlet space admits a carr\'e du champ.
Actually, for our purpose here it would be sufficient that the doubling
and Poincar\'e inequalities -- or, equivalently, the parabolic Harnack
inequality -- hold on balls of radius $\le1$. An important example
are manifolds with non-negative Ricci curvature, or more generally
manifolds equipped with a Riemannian metric $\sfg$ that is uniformly
equivalent to a metric $\sfg'$ of non-negative Ricci curvature, i.e.
\[
\frac{1}{\lambda} \sfg'\leq \sfg\leq \lambda \sfg'\quad\text{for some }\lambda>0\;,
\]
see \cite{SC92}. Harnack-type Dirichlet spaces satisfy upper and lower Gaussian bounds
on the heat kernel. Thus, criteria for the Kato class can be
transferred from $\R^n$. In particular, we have the following.

\begin{lemma}\label{lem:KatoLp-mfd}
  Let $(\X,\cE,\mm)$ be a Harnack-type Dirichlet space and let
 $$k\in \bigcup_{p>n/2} L^p(\X,\mm)\;.$$ 
 Then the distribution $\kappa:=k\,\mm\in \cF^{-1}_{\rm qloc}$ belongs to
 the Kato class $\mathcal K_0(\X)$.
\end{lemma}

Next, we discuss how the Kato class behaves under restriction to
sufficiently regular subdomains. Let $\Y\subset \X$ be an open
connected subset, which is \emph{inner uniform} in the sense of
\cite{GyrSal}, i.e.~there are constants $c,C>0$ such that any
$x,y\in \Y$ can be connected by a continuous curve
$(\gamma_t)_{t\in[0,1]}$ with length at most $C \rho_\Y(x,y)$ such
that for all $z\in\gamma([0,1])$
\begin{equation}\label{eq:inner-unif}
\rho(z,\partial\Y)\geq c \min\{\rho_\Y(z,x),\rho_\Y(z,y)\}\;,
\end{equation}
where $\rho_\Y$ is the intrinsic length distance in $\Y$ induced by $\rho$.
Moreover, assume that
\[
\inf\bigg\{\frac{\mm\big(B_r(y)\cap \Y)\big)}{\mm\big(B_r(y))\big)} \,:\, \ r>0,\, y\in \Y\bigg\}>0.
\]

\begin{Lemma}\label{lem:kato-comp} 
  Under the given assumptions, any signed measure on $\Y$ belongs to
  the Kato class w.r.t.\ the Neumann heat flow on $\Y$ if and only if
  it belongs to the Kato class w.r.t.\ the heat flow on $\X$.
\end{Lemma}

\begin{proof} 
  Under the given assumptions, the distances $\sfd_\X$ and $\sfd_\Y$
  are comparable and so are the volumes of balls $\mm_X(B^\X_r(y))$
  and $\mm_Y(B^\Y_r(y))$. The restricted space
  $(\Y,\mm\big|_\Y,\cE^\Y,\fD(\cE^\Y))$ will also be a Harnack-type
  Dirichlet space, see \cite[Theorem 3.10]{GyrSal}. Thus, according to
  the uniform upper and lower heat kernel estimates of Gaussian type,
  which are valid in such Harnack-type Dirichlet space, the heat
  kernels on $\X$ and on $\Y$ are comparable in the sense that for
  some constant $C>0$.
\[
\frac1C\cdot p^\X_{Ct}(x,y)\le  p^\Y_t(x,y)\le C\cdot p^\X_{t/C}(x,y)\qquad (\forall x,y\in \Y, \forall t>0).
\]
This implies that the Kato class w.r.t.\ the heat flow on $\Y$ coincides with the Kato class w.r.t.\ the heat flow on $\X$.
\end{proof}
 
\begin{Example} The assumptions of the previous Lemma are in particular satisfied for each domain $\Y$ in a Riemannian manifold $\X$ provided the boundary of $\Y$ is locally given as the graph of a Lipschitz function.
\end{Example}

\subsubsection{An $L^p$-Criterion for the Density of the Surface Measure}
 
Let a complete $n$-dimensional Riemannian manifold $(\sfM,\sf g)$ be given with the property that
\begin{equation*}
V\in \K(\sfM)\quad \Longleftrightarrow\quad \lim_{r\to0}\sup_{x\in \sfM}\int_{B_r(x)}\frac{|V(y)|}{\sfd(x,y)^{n-2}}\, \text{vol}_\sfM(\d y)=0.
\end{equation*}
This property is always fulfilled if $\sfM=\R^n$ or if $\sfM$ is compact. 
It immediately carries over to the analogous characterization of signed measures in the Kato class.
 
Let $\Y\subset \sfM$ be an open, connected subset with a boundary
which is Lipschitz in the following weak sense: there exists a constant $C>0$, a covering $(U_i)_{i=1,\ldots, k}$ of $\partial \Y$ by open sets $U_i$ in $\sfM$ and $C$-Lipschitz maps $\varphi_i \colon U_i\to \R^{n-1}$ such that
\[
(\varphi_i)_* \text{vol}_{\partial \Y}\le C\cdot \text{vol}_{\R^{n-1}}\qquad\text{on }U_i.
\]
Note that this is satsified if $\partial \Y\cap U_i$ is given as the graph of a Lipschitz function.
  
\begin{Theorem}\label{thm:Kato-bdry} 
  In additon to the previous assumptions on $\sfM$ and $\Y$, assume
  that $V\in L^p(\partial \Y, \text{\rm vol}_{\partial \Y})$ for some
  $p>n-1$. Then $\mu:=V\, \text{\rm vol}_{\partial \Y}$ is a signed
  measure in the Kato class $\KK_0(\sfM)$.
\end{Theorem}

\begin{proof} For $r>0$ small enough, each ball $B_r(x)$ which intersects with $\partial \Y$ is contained in one of the $U_i$. Thus  with $q$ being the exponent dual to $p$,
\begin{eqnarray*}\int_{B_r(x)}\frac{|V(y)|}{\sfd(x,y)^{n-2}}\, \text{vol}_{\partial \Y}(\d y)
\le
  \Big(
\int_{\partial \Y}|V(y))|^p\, \text{vol}_{\partial \Y}(\d y)\Big)^{1/p}\cdot\Big(\int_{B_r(x)}\frac{1}{\sfd(x,y)^{q(n-2)}}\, \text{vol}_{\partial \Y}(\d y)\Big)^{1/q}
\end{eqnarray*} 
where
\begin{eqnarray*}
\int_{B_r(x)}\frac{1}{\sfd(x,y)^{q(n-2)}}\, \text{vol}_{\partial \Y}(\d y)&\le&
 \int_{\{y\in U_i \,:\, |\varphi(x)-\varphi(y)|<r/C\}}\frac{C^{q(n-2)}}{|\varphi(x)-\varphi(y)|^{q(n-2)}}\, \text{vol}_{\partial \Y}(\d y)\\
 &\le& \int_{\{z\in \R^{n-1} \,:\, |z|<r/C\}}\frac{C\cdot C^{q(n-2)}}{|z|^{q(n-2)}}\text{vol}_{\R^{n-1}}(\d z),
\end{eqnarray*} 
which in turn is finite (and converges to 0 as $r\to0$) provided $-q(n-2)+n-2>-1$. The latter  is equivalent to $p>n-1$.
\end{proof}

  \subsubsection{Kato Class and Moderate Distributions}

\begin{Proposition}\label{prop:Kato-moderate-mu}  
Every  signed measure in $\KK_{0}(\X)$  is moderate. More generally, a signed smooth measure  $\mu=\mu^+-\mu^-$ is moderate if $\mu^-\in\KK_{1-}(\X)$. 
\end{Proposition}

\begin{proof} 
Let $\kappa\in \cF$ be given as $\kappa=\mu^+-\mu'$ with $\mu^-\in\KK_{1-}(\X)$ and let $A^+$ and $A^-$ denote the PCAF's associated with  $\mu^+$ and $\mu^-$, resp. Then by Khasminskii's lemma 
\[
\text{\rm q-\!}\sup_x \E_x \Big[e^{-A_t^++A_t^-}\Big]\le \big(1- C_t\big)^{-1}<\infty
\]
with $C_t:=\sup_x \E_x \big[A_t^-\big]$ which by assumption is less than 1 for all sufficiently small $t>0$. This obviously  implies \eqref{moder}.
\end{proof}
  
  \begin{Proposition}\label{grad-tame} If $\kappa=-\sfL\psi$ for some $\psi\in \dot \cF_{\rm loc}$ with $\Gamma(\psi)\in  \K_{0}(\X)$, then $\kappa$ is moderate.
  \end{Proposition}

  \begin{proof} $\kappa=-\sfL\psi$ implies  
  \[A^\kappa_t=M^\psi_t+\hat M^\psi_t.\]
  Hence  (assuming for simplicity $\zeta=\infty$)
  \begin{eqnarray*}
 \E_x\big[e^{ A^\kappa_t}\big]&=&\E_x\big[e^{M^\psi_t}\cdot e^{\hat M^\psi_t}\big]\\
  &\le&
  \E_x\big[e^{3M^\psi_t-\frac{9}2\langle M^\psi\rangle_t}\big]^{1/3}\cdot
   \E_x\big[e^{3\hat M^\psi_t-\frac{9}2\langle M^\psi\rangle_t}\big]^{1/3}\cdot
   \E_x\big[e^{9{}\langle M^\psi\rangle_t}\big]^{1/3}\\
   &=& \E_x\big[e^{9{}\int_0^t \Gamma(\psi)(B_s)\d s}\big]^{1/3}\le C\cdot e^{Ct}
  \end{eqnarray*}
  quasi-uniformly in $x$.
  \end{proof}
  
  \begin{Example}\label{dense-osc} Let $\X=\R^n$ for $n\ge2$   equipped with classical Dirichlet form $\cE$ and Lebesgue measure $\mm$. 
  
  (i) Then for $\ell,m\ge0$,  according to {\cite{Sturm92HS}},
  \begin{equation*}
  V(x):=\|x\|^{-\ell}\cdot \sin\big(\|x\|^{-m}\big)
  \end{equation*}
  is moderate if and only if $\ell<2+m$. In contrast to that, $V\in\K_0(\R^n)$ (or, equivalently, $V\in\K_\infty(\R^n)$ )  if and only if $\ell<2$.
  
  (ii) More generally, given any $\ell,m\ge0$ with $\ell<2+m$,  a dense set $\{z_i\}_{i\in\N}\subset\R^n$ and an absolutely summable sequence of  numbers $(k_i)_{i\in\N}$,  
  the potential
  \begin{equation*}
  V(x):=\sum_{i=1}^\infty k_i\cdot \|x-z_i\|^{-\ell}\cdot \sin\big(\|x-z_i\|^{-m}\big)
  \end{equation*}
  will be moderate. 
  
 (iii)  Note that for $\ell\ge n$,  these potentials will not be locally integrable. (Even worse, the latter will be nowhere locally integrable.) In particular, the associated distributions will not be given by signed Radon measures.
  \end{Example}

\subsubsection{A Powerful Approximation Property}
  
\begin{Lemma}\label{Dyn-appr} For each Borel function $f\in L^1(\X,\mm)$ and each $\rho>0$ there exists an increasing sequence of finely open, nearly Borel sets $(G_n)_{n\in\N}$ such that $\X\setminus \cup_n G_n$ is $\cE$-polar and 
\[
\mathds{1}_{G_n}\,f\in \K_\rho(\X)\qquad(\forall n\in\N).
\]
\end{Lemma}

\begin{proof} 
Assume without restriciton that $f\ge0$, and for $n\in\N$ consider the functions $u_n(x):=R_n f(x)=\E_x\big[\int_0^\zeta e^{-n t}f(B_t) \, \d t\big]$ which are $n$-excessive and thus finely continuous. Define an increasing sequence of finely open, nearly Borel sets by
\begin{equation}\label{def:Gn}
G_n:=\big\{u_n<\rho\big\}\qquad (\forall n\in\N).
\end{equation}
Note that $\int_\X u_n\,\d\mm\le\frac1n \int_\X f \,\d\mm<\infty$, which implies $u_n<\infty$ $\mm$-a.e.~on $\X$ and thus in turn $u_n<\infty$ q.e.~on $\X$ (\cite[Theorem A.2.13]{Chen-Fuku}). Hence, $(G_n)_{n\in\N}$ is a quasi-open nest.

Moreover, for each $n$ by construction
\[
R_n(\mathds{1}_{G_n}f)< \rho \quad\text{on }G_n
\]
which (by fine continuity of the LHS) implies 
\[
R_n(\mathds{1}_{G_n}f)\le \rho \quad\text{on }\tilde G_n
\]
where $\tilde G_n$ denotes the fine closure of $G_n$. Since the LHS is $n$-harmonic in the finely open set $\X\setminus \tilde G_n$, by maximum principle this in turn implies ({\cite[Theorem A.1.22]{Chen-Fuku}})	 
\[
R_n(\mathds{1}_{G_n}f)\le \rho \quad\text{q.e.~on }X.
\]
Hence, in particular, $\mathds{1}_{G_n}f\in \K_\rho(\X)$.
\end{proof}

\begin{Corollary}\label{Dyn-appr-loc} 
The same as in the previous Lemma is true for each  Borel function $f\in L^1_{\rm qloc}(\X,\mm)$, where the latter is defined as the $\mm$-equivalence class of Borel functions $f\colon \X\to\R$ for which there exists an increasing sequence of finely open, nearly Borel sets $(E_n)_{n\in\N}$ such that $\X\setminus \cup_n E_n$ is $\cE$-polar and $f\in L^1(E_n,\mm\big|_{E_n})$.
\end{Corollary}

\begin{proof} Given the finely open nest $(E_n)_{n\in\N}$, according to the previous Lemma, for each $n$ there exists an increasing seqeuence   $(G_{n,k})_{k\in\N}$ of finely open sets such that $E_n\setminus\bigcup_k G_{n,k}$ is $\cE$-polar and 
\[
\mathds{1}_{G_{n,k}}\,f\in \K_{\rho\,2^{-n}}(\X) \qquad(\forall k\in\N).
\]
Without restriction, we may assume that $\text{cap}_1(\X\setminus E_n)\le 1/n$ and $\text{cap}_1(E_n\setminus G_{n,n} )\le 1/n$.
Then
\[G_k:=\bigcup_{n=1}^k G_{n,k}\] defines an increasing sequence of finely open sets such that $\X\setminus\bigcup_k G_{k}$ is $\cE$-polar and 
\[
\mathds{1}_{G_{k}}\,f\in \K_{\rho}(\X),\quad\forall \, k\in\N.
\]
\end{proof}

Recall that a measure $\mu$ defined on the Borel $\sigma$-field of $\X$ is called \emph{smooth in the strict sense} (in the notation of \cite{Chen-Fuku}, $\mu\in S_1$) if it does not charge $\cE$-polar sets  and if it admits a nest $(G_n)_n$ of finely open Borel sets $G_n\subset \X$ such that $\mathds{1}_{G_n}\,\mu\in \KK_\infty(\X)$ and $\mu(G_n)<\infty$ for each $n$.
  
\begin{Proposition}\label{prop:meadistr} 
Every measure $\mu$ on $\X$ which is smooth in the strict sense defines (or can be interpreted as) a distribution $\kappa\in \cF^{-1}_{\rm qloc}$. Indeed, given a nest $(G_n)_{n \in \N}$ as above, $\kappa\in \cF^{-1}\big((G_n)_n\big)$ can be defined via
\[
\langle \kappa,\varphi\rangle:=\int_{G_n} \tilde\varphi\,\d\mu\qquad\forall n, \forall \varphi\in  \cF_0(G_n).
\]
Conversely, every non-negative distribution $\kappa\in  \cF^{-1}_{\rm qloc}$ defines a measure on $\X$ which is smooth in the strict sense.
\end{Proposition}

The first part of the previous Proposition  easily extends to \emph{signed smooth measures in the strict sense.}

\medskip
  
Recall that the Revuz correspondence
\[
\int_\X f\,\d\mu=\lim_{t\to0}\frac1t \E_\mm\Big[\int_0^tf(B_t)\,\d A_t\Big]\qquad \forall \text{ Borel }f \colon \X\to\R
\]
establishes a one-to-one correspondence between smooth measures and PCAF's. Under the so-called \emph{absolute continuity hypothesis}, this induces also a one-to-one correspondence between smooth measures in the strict sense  and PCAF's in the strict sense, see \cite[Theorem 4.1.11]{Chen-Fuku}.
  
\begin{Corollary} 
Assuming the absolute continuity hypothesis, every PCAF $(A_t)_{t\ge0}$  in the strict sense is uniquely associated to some $\kappa=\kappa^A\in \cF^{-1}_{\rm qloc}$.
\end{Corollary}

 \subsection{Singularly-Perturbed Energy Forms}
 
Our next goal is to define the energy $\cE^\kappa$ associated with a distribution $\kappa \in \cF^{-1}_{\rm qloc}$. And we will prove that this quadratic form is always associated to the Feynman-Kac semigroup $(P^\kappa_t)_{t\ge0}$ already defined by means of the Feynman-Kac formula in terms of the local additive functional $A^\kappa$ associated with $\kappa$.
 Our approach is inspired by the work of Chen, Fitzsimmons, Kuwae, and Zhang \cite{CFKZ} and partly based on their result together with two  approximation procedures. In contrast to them, we restrict ourselves to strongly local, symmetric Dirichlet forms $\cE$ but we admit a larger class of singular perturbations. Moreover, with a more detailed analysis we succeed to identify the energy $\cE^\kappa$  as the closure of the limiting objects and not just as the relaxations.

\begin{Lemma}\label{lemeins} 
Let $\kappa\in \cF^{-1}$ and put $\psi:=(-\sfL+1)^{-1}\kappa$. Assume that $|\psi|+ 2 \Gamma(\psi)\in \K_{1-}(\X)$ or, more generally, that $\frac1{1-\delta}|\psi|+\frac1{\delta(1-\delta)}\Gamma(\psi)\in \K_{1-}(\X)$ for some $\delta\in(0,1)$. Then 
\[
\cE^\kappa(f):=\cE(f)+\cE(f^2,\psi)+\int_\X f^2\,\psi\,\d\mm
\]
with $\fD(\cE^\kappa)=\fD(\cE)=\cF$ defines a closed, lower bounded and densely defined quadratic form on $L^2(\X,\mm)$. The associated strongly continuous semigroup on $L^2(\X,\mm)$ is given by $(P^\kappa_t)_{t\ge 0}$. 
 
Moreover, the semigroup $(P^\kappa_t)_{t\ge 0}$ on $L^2(\X,\mm)$ extends to an exponentially bounded semigroup on $L^p(\X,\mm)$ for each $p\in[1,\infty]$ (which again is strongly continuous provided $p<\infty$).
\end{Lemma}
 
 \begin{proof}
 To check the lower boundedness of the form $\cE^\kappa$, we us the chain rule for the energy measure $\mu_{\langle\psi\rangle}$ and a simple application of the Cauchy-Schwarz inequality  to deduce
 \[\cE^\kappa(f)\ge (1-\delta)\cE(f)-\frac1\delta \int_\X f^2\,\Gamma(\psi)\,\d\mm+\int_\X f^2\,\psi\,\d\mm\]
 for arbitrary $\delta>0$.
 According to Corollaray \ref{Khas-Cor}, the right-hand side is bounded from below provided
  $\frac1{1-\delta}\big[\frac1{\delta}\Gamma(\psi)+|\psi|\big]\in \K_{1-}(\X)$.
  The remaining results then are particular cases of the more general basic result in \cite{CFKZ}.
  \end{proof}
 
\begin{Proposition}\label{lemzwo} 
Given $\kappa\in \cF^{-1}$, put $\psi:=(-\sfL+1)^{-1}\kappa$. Then there exists a quasi-open nest $(G_\ell)_\ell$ such that $[|\psi|+ 2 \Gamma(\psi)]\,\mathds{1}_{G_\ell}\in \K_{1-}(\X)$, for each $\ell \in \N$. Given such a nest, define a quadratic form $(\dot\cE^\kappa, \fD(\dot\cE^\kappa))$ by $\fD(\dot\cE^\kappa):=\bigcup_\ell \cF_0(G_\ell)$ and 
\[
\dot\cE^\kappa(f):=\cE(f)+\cE(f^2,\psi)+\int_\X f^2\,\psi\,\d\mm.
\]
Put $\lambda_0^\kappa:=\inf\{\dot\cE^\kappa(f): f\in \fD(\dot\cE^\kappa), \|f\|_{L^2}\le 1\}$. Then the following are equivalent:
 \begin{itemize}
 \item[(i)] $\lambda_0^\kappa>-\infty$,
 \item[(ii)] $\|P^\kappa_t\|_{L^2,L^2}=e^{-\lambda\, t}$ for all $t\ge0$ and some $\lambda\in\R$,
 \item[(iii)] $\|P^\kappa_t\|_{L^2,L^2}<\infty$ for some $t>0$.
 \end{itemize}
In this case, $\lambda=\lambda_0^\kappa$ and the semigroup $(P^\kappa_t)_{t\ge0}$ is strongly continuous and exponentially bounded on $L^2(\X,\mm)$. Moreover, the quadratic form $(\dot\cE^\kappa, \fD(\dot\cE^\kappa))$ is lower bounded  on $L^2(\X,\mm)$. Its relaxation $(\cE^\kappa, \fD(\cE^\kappa))$ is the closed, densely defined, lower bounded quadratic form associated with $(P^\kappa_t)_{t\ge0}$. 

The quadratic form $(\cE^\kappa, \fD(\cE^\kappa))$ does not depend on the choice of the nest $(G_\ell)_{\ell \in \N}$. 
  \end{Proposition}

Here and in the sequel, a quadratic form $Q$ with domain $\fD(Q)\subset L^2(\X,\mm)$ will always be extended to $L^2(\X,\mm)$ by assigning to it the value $+\infty$ on $L^2(\X,\mm)\setminus \fD(Q)$. The \emph{relaxation} of a lower bounded quadratic form $\big(Q, \fD(Q)\big)$ on $L^2(\X,\mm)$ denotes the largest lower bounded closed quadratic form $\big(\underline Q, \fD(\underline Q)\big)$ on $L^2(\X,\mm)$ which is dominated by $\big(Q, \fD(Q)\big)$. It is explicitely given by
\[
\underline Q(f):= \inf\Big\{ \liminf_{n\to\infty}Q(g_n) \,:\, (g_n)_{n\in \N} \subset L^2(\X,\mm), \,\, g_n\to f\Big\}
\]
and $\fD(\underline Q):=\{f\in L^2(\X,\mm) \,:\, \ \underline Q(f)<\infty\}$. The form  $\big(Q, \fD(Q)\big)$ is \emph{closable} if and only if $\underline Q=Q$ on $\fD(Q)$.

  \begin{proof} Since  $\Gamma(\psi)\in L^1(\X,\mm)$ and $\psi\in L^1(\X,\mm)+L^\infty(\X,\mm)$, the existence of a quasi-open nest with the requested properties follows from Lemma \ref{Dyn-appr} and the fact that essentially bounded functions are contained in the extended Kato class.
    According to Lemma \ref{lemeins}, each of the forms  $\big(\dot\cE^\kappa, \cF_0(G_\ell)\big)$ is lower bounded and closed.
  The semigroups $(P^{\kappa,\ell}_t)_{t\ge0}$ associated with these lower bounded, closed forms are given by
   \begin{eqnarray}\label{heat-ell}
  P^{\kappa,\ell}_{t}f(x)=\E_x\Big[e^{-A^{\kappa}_{t}}\, f(B_{t})\Big]
   \end{eqnarray}
   (with life time $\tau_{G_\ell}\wedge\zeta$) and
   $\|P^{\kappa,\ell}_t\|_{L^2,L^2}=e^{-\lambda_0^{\kappa,\ell}\, t}$
   with
   $$\lambda_0^{\kappa,\ell}:=\inf\big\{\dot\cE^\kappa(f) \,:\, f\in \cF_0(G_\ell), \|f\|_{L^2}\le 1\big\}.$$
   Obviously, these numbers are decreasing in $\ell$ and \[\lambda_0^{\kappa}=\lim_{\ell\to\infty}\lambda_0^{\kappa,\ell}.\]
   
   Now for each $t>0$, put $\lambda(t):=-\frac1t \|P^{\kappa}_t\|_{L^2,L^2}$ with $\lambda(t):=-\infty$ if $\|P^{\kappa}_t\|_{L^2,L^2}=\infty$. Since  $(G_\ell)_\ell$ is a nest, for non-negative $f$ the functions $P^{\kappa,\ell}_{t}f$ are non-decreasing in $\ell$ and converge monotonically to $P^{\kappa,\ell}_{t}f$ as $\ell\to\infty$. Therefore, $\lambda(t)\le\lambda_0^{\kappa,\ell}$ for each $\ell$ and thus $\lambda(t)\le\lambda_0^{\kappa}$.
   
   On the other hand, for each $t>0$ and each  $C>\lambda(t)$, there exists $f\not=0$ with  $\|P^{\kappa}_tf\|_{L^2}>e^{-Ct}\|f\|_{L^2}$. Thus also $\|P^{\kappa,\ell}_tf\|_{L^2}<e^{-Ct}\|f\|_{L^2}$ for all large enough $\ell$ and therefore $\lambda_0^{\kappa,\ell}<C$ which in turn implies $\lambda_0^{\kappa}\le C$. Hence, $\lambda_0^{\kappa}=\lambda(t)$.
   
 Now assume that $\lambda_0^{\kappa}>-\infty$. Then the non-negative, densely defined, closed forms 
\[
\big(\dot\cE^\kappa+\lambda_0^{\kappa}\,\|\cdot\|_{L^2}^2, \cF_0(G_\ell)\big)
\]
(extended to functionals on $L^2(\X,\mm)$) are decreasing in $\ell$. Hence, according to \cite{Reed-Simon}, the semigroups $(e^{\lambda_0^{\kappa}t}\, P^{\kappa,\ell}_t)_{t\ge0}$ associated with the respective forms will converge to a  semigroup  which in turn is associated to a non-negative, densely defined, closed form which is obtained as the relaxation of  \[\big(\dot\cE^\kappa+\lambda_0^{\kappa}\,\|\cdot\|_{L^2}^2, \bigcup_\ell \cF_0(G_\ell)\big).\] 
  
\medskip

   Uniqueness of  $\big(\cE^\kappa, \fD(\cE^\kappa)\big)$ follows from the fact that it is uniquely associated to the semigroup $(P^\kappa_t)_{t\ge0}$ which will not depend on the choice of the  nest  $(G_\ell)_\ell$ .
  \end{proof}

Given $\kappa\in \cF^{-1}_{\rm qloc}$, choose a quasi-open nest $(G_n)_n$ such that $\kappa\in \bigcap_n \cF^{-1}(G_n)$. For each $n$, put $\psi_n = (-\sfL_{G_n}+1)^{-1}\kappa$ and choose a quasi-open nest $(G_{n,\ell})_\ell$ in $G_n$ such that $[|\psi_n|+ 2 \Gamma(\psi_n)]\,\mathds{1}_{G_{n,\ell}}\in \K_{1-}(G_n)$ for each $\ell$. 
  Define the closed quadratic form $\big(\cE^{\kappa,n}, \fD(\cE^{\kappa,n})\big)$ as the relaxation of the quadratic form
 \[\dot\cE^{\kappa,n}(f):=\cE(f)+\cE(f^2,\psi_n)+\int_\X f^2\,\psi_n\,\d\mm\]
with $\fD(\dot\cE^{\kappa,n}):=\bigcup_\ell \cF_0(G_{n,\ell})$.
According to Proposition \ref{lemzwo}, $\big(\cE^{\kappa,n}, \fD(\cE^{\kappa,n})\big)$ is associated to the strongly continuous semigroup 
   \begin{eqnarray*}
  P^{\kappa,n}_{t}f(x)=\E_x\Big[e^{-A^{\kappa}_{t}}\, f(B_{t})\Big].
   \end{eqnarray*}
 Note that for $f\ge0$, this obviously is increasing in $n$. Hence, $(\cE^{\kappa,n})_n$ constitutes a decreasing  sequence of closed quadratic functionals on $L^2(\X,\mm)$.
   
 \begin{Proposition}\label{lem-3} Given $\kappa\in \cF^{-1}_{\rm qloc}$, choose a quasi-open nest $(G_n)_{n\in\N}$ such that $\kappa\in \bigcap_n \cF^{-1}(G_n)$. For each $n$,  let $\big(\cE^{\kappa,n}, \fD(\cE^{\kappa,n})\big)$ be the closed quadratic form constructed as above, and define a 
quadratic form  with $\fD(\tilde\cE^{\kappa}):=\bigcup_n \fD(\cE^{\kappa,n})$ by
 \[\tilde\cE^{\kappa}(f):=\lim_{n\to\infty} \cE^{\kappa,n}(f).\]
Put
$\lambda_0^\kappa:=\inf\{\tilde\cE^\kappa(f): f\in \fD(\tilde\cE^\kappa), \|f\|_{L^2}\le 1\}$.
Then again, the properties (i), (ii), and (iii) of the previous Proposition are equivalent. 

Moreover, if $\lambda_0^\kappa>-\infty$ then  the semigroup $(P^{\kappa}_t)_{t\ge0}$  is  strongly continuous on $L^2(\X,\mm)$ and the form $\big(\tilde\cE^\kappa, \fD(\tilde\cE^\kappa)\big)$ is lower bounded. Its relaxation  $\big(\cE^{\kappa}, \fD(\cE^{\kappa})\big)$ is uniquely characterized as the lower bounded, densely defined, closed quadratic form  associated with 
  the    semigroup $(P^{\kappa}_t)_{t\ge0}$.
 
 The construction of $\big(\cE^{\kappa}, \fD(\cE^{\kappa})\big)$ does not depend on the choice of the quasi-open nests $(G_n)_n$ and  $(G_{n,\ell})_\ell$. 
  \end{Proposition} 
  
 \begin{proof}  Analogously to the proof of the previous Proposition.
 \end{proof}
 
 Let us now finally show that the two-fold relaxation in the construction of the closed form  $\big(\cE^{\kappa}, \fD(\cE^{\kappa})\big)$ can be replaced by relaxation in one step.
 
 \begin{Theorem}\label{relaxThm}
 Given $\kappa\in \cF^{-1}_{\rm qloc}$, choose a quasi-open nest $(G_n)_n$ such that $\kappa\in \bigcap_n \cF^{-1}(G_n)$. For each $n$, put $\psi_n = (-\sfL_{G_n}+1)^{-1}\kappa$ and choose a quasi-open nest $(G_{n,\ell})_\ell$ in $G_n$ such that $[|\psi_n|+ 2 \Gamma(\psi_n)]\,\mathds{1}_{G_{n,\ell}}\in \K_{1-}(G_n)$ for each $\ell$. Define a quadratic form $\big(\dot\cE^\kappa, \fD(\dot\cE^\kappa)\big)$ with
 $\fD(\dot\cE^\kappa):=\bigcup_n\ \bigcup_{\ell=1}^\infty \cF_0(G_{n,\ell})$ by
 \begin{eqnarray}\label{eny}
 \dot\cE^\kappa(f)&:=& \cE(f)+\cE(f^2,\psi_n)+\int_\X f^2\,\psi_n\,\d\mm
 \end{eqnarray}
 for $f\in \bigcup_{\ell=1}^\infty \cF_0(G_{n,\ell})$. Put $\lambda_0^\kappa:=\inf\{\dot\cE^\kappa(f): f\in \fD(\dot\cE^\kappa), \|f\|_{L^2}\le 1\}$.
Then again, the properties (i), (ii), and (iii) of Proposition \ref{lemzwo} are equivalent. 

If $\lambda_0^\kappa>-\infty$, then  the semigroup $(P^{\kappa}_t)_{t\ge0}$ is strongly continuous on $L^2(\X,\mm)$ and the form $\big(\dot\cE^\kappa, \fD(\dot\cE^\kappa)\big)$ is lower bounded. Its relaxation is uniquely characterized as the lower bounded, densely defined, closed quadratic form  $\big(\cE^{\kappa}, \fD(\cE^{\kappa})\big)$ associated with the semigroup $(P^{\kappa}_t)_{t\ge0}$.
 
 The form $\big(\cE^{\kappa}, \fD(\cE^{\kappa})\big)$ does not depend on the choice of the nests $(G_n)_n$ and $(G_{n,\ell})_\ell$. It coincides with the closed form constructed in the previous Proposition.
  \end{Theorem} 
  
  \begin{proof} Let quasi-open nests  $(G_n)_n$ and  $(G_{n,\ell})_\ell$ be given as for the construction in the above Theorem. Observe that the form $\big(\dot\cE^\kappa, \fD(\dot\cE^\kappa)\big)$ is well defined since  
 \[\cE(f^2,\psi_n)+\int_\X f^2\,\psi_n\,\d\mm=\cE(f^2,\psi_j)+\int_\X f^2\,\psi_j\,\d\mm\]
 for  $f\in \Big(\bigcup_{\ell=1}^\infty \cF_0(G_{n,\ell})\Big)\cap\Big( \bigcup_{\ell=1}^\infty \cF_0(G_{j,\ell})\Big)$.
    Let  $\underline\cE^\kappa$ denote the relaxation of the form $\dot\cE^\kappa$ defined in \eqref{eny}.
           Moreover, for each $n\in\N$, let $\cE^{\kappa,n}$ denote the closed form constructed as relaxation of the form  $\dot\cE^{\kappa,n}$
   in Proposition \ref{lemzwo}, and let  $\cE^{\kappa}$ denote the closed form constructed as relaxation of $\lim_{n\to\infty}\cE^{\kappa,n}$ in Proposition \ref{lem-3}.
   
      Then obviously
  \[\dot\cE^\kappa\le\dot\cE^{\kappa,n} \text{ on }L^2(\X,\mm)\]  for each $n$ and thus
  \[\underline\cE^\kappa\le \cE^{\kappa,n} \text{ on }L^2(\X,\mm)\] for each $n\in\N$. This implies
  $\underline\cE^\kappa\le \cE^{\kappa}$.
  
  On the other hand, for each $f\in\fD(\underline\cE^{\kappa})$ and each $n\in\N$ there exists $g_n\in \fD(\dot\cE^{\kappa,n})$ with $\|f-g_n\|_{L^2}\le1/n$ and $\dot\cE^{\kappa,n}(g_n)\le \underline\cE^{\kappa}(f)+1/n$. Hence, in particular, $\cE^{\kappa,n}(g_n)\le \underline\cE^{\kappa}(f)+1/n$ and thus also $\dot\cE^{\kappa}(g_n)\le \underline\cE^{\kappa}(f)+1/n$ for each $n\in\N$. This finally implies $\cE^{\kappa}(f)\le \underline\cE^{\kappa}(f)$. 
That is, the form  $\cE^{\kappa}$ constructed in Proposition \ref{lem-3} coincides with the relaxation of the form $\dot\cE^\kappa$ defined in \eqref{eny}.
   \end{proof}
  
Let us add now a brief discussion on the question whether the form $\big(\dot\cE^\kappa, \fD(\dot\cE^\kappa)\big)$ in the previous Theorem is closable (in which case $\big(\cE^{\kappa}, \fD(\cE^{\kappa})\big)$ simply will be its closure). Let us first consider the  basic case of distributions $\kappa$ which are non-negative.

\begin{Proposition} 
Assume that $\kappa\in \cF^{-1}_{\rm qloc}$ is non-negative (in other words, $\kappa$ is a smooth measure in the restricted sense). 
Then the quadratic form $\big(\dot\cE^\kappa, \fD(\dot\cE^\kappa)\big)$ as introduced in the previous Theorem with $\fD(\dot\cE^\kappa):=\bigcup_n\ \bigcup_{\ell=1}^\infty \cF_0(G_{n,\ell})$ and
 \begin{eqnarray*}\dot\cE^\kappa(f):=
 \cE(f)+\cE(f^2,\psi_n)+\int_\X f^2\,\psi_n\,\d\mm \end{eqnarray*}
is closable. Its closure $\big(\cE^\kappa, \fD(\cE^\kappa)\big)$ is given by $\fD(\cE^{\kappa})=\Big\{f\in\cF : \ \int_\X \tilde f^2\,\d\kappa<\infty\Big\}$ and
\[
\cE^\kappa(f) := \cE(f)+ \int_\X \tilde{f}^2\,\d\kappa.
\]
Here $\tilde f$ denotes the quasi-continuous modification of $f$ so that the integral w.r.t.~the measure $\kappa$ (which charges no $\cE$-polar sets) is well defined.
\end{Proposition}

\begin{proof} Choose a quasi-open nest $(G_n)_n$ such that $\kappa\in \bigcap_n \cF^{-1}(G_n)$. For each $n$, put $\psi_n=(-\sfL_{G_n}+1)^{-1}\kappa$ and choose a quasi-open nest $(G_{n,\ell})_\ell$ in $G_n$ such that  $[|\psi_n|+ 2 \Gamma(\psi_n)]\,\mathds{1}_{G_{n,\ell}}\in \K_{1-}(G_n)$ for each $\ell$.
Observe that 
\[
\cE(f^2,\psi_n)+\int_\X f^2\,\psi_n\,\d\mm= \int_\X \tilde f^2\,\d\kappa
\]
for all $n\in\N$ and all $f\in \bigcup_{\ell=1}^\infty \cF_0(G_{n,\ell})$. Hence, in particular, 
\[
\fD(\dot\cE^\kappa)\subset \Big\{f\in\cF:\ \int_\X \tilde f^2\,d\kappa<\infty\Big\}=:\cF^\kappa.
\]
Moreover, note that $\cF^\kappa$ is closed w.r.t.~the norm $\big(\dot\cE^\kappa(\cdot)+\|\cdot\|_{L^2}^2\big)^{1/2}$.
Thus $\big(\dot\cE^\kappa, \fD(\dot\cE^\kappa)\big)$ is closable and its closure satisfies $\fD(\cE^\kappa)\subset\cF^\kappa$. 

To prove equality in the last assertion, let $f\in \cF^\kappa$ be given; without restriction, $f\ge0$. Let $f_n$ and $f_{n,\ell}$ denote the projections of $f\in \cF$ onto $\cF_0(G_n)$ or on $\cF_0(G_{n,\ell})$, resp. Then  obviously $f_{n,\ell}\to f_n$ as $\ell\to\infty$ and $f_{n}\to f$ as $n\to\infty$ w.r.t.~$\cE_1$. Moreover, passing to suitable subsequences (which we do not indicate in the notation), we obtain 
$\tilde f_{n,\ell}\to\tilde f_n$ q.e.~on $\X$ as $\ell\to\infty$ and $\tilde f_{n}\to\tilde f$ q.e.~on $\X$ as $n\to\infty$ (see \cite[Theorem 2.3.4]{Chen-Fuku}). Hence, $\int_\X \tilde f^2_{n,\ell}\d\kappa\to\int_\X \tilde f^2_n \d\kappa$ and $\int_\X \tilde f^2_{n}\d\kappa\to\int_\X \tilde f^2 \d\kappa$ and therefore finally $\dot\cE\big(f_{n,\ell}\big)\to \dot\cE\big(f_{n}\big)$  as $\ell\to\infty$ and $\dot\cE\big(f_{n}\big)\to \dot\cE\big(f\big)$ as $n\to\infty$. This proves that $f$ is contained in the closure of $\fD(\dot\cE^\kappa)$.
\end{proof}
  
\begin{Theorem}\label{thm:condition-closable} 
Assume that $\kappa\in \cF^{-1}_{\rm qloc}$ admits a decomposition $\kappa=\mu+\kappa_0$ with $\mu,\kappa_0\in \cF^{-1}_{\rm qloc}$, $\mu\ge0$ such that $f\mapsto\langle\kappa_0,f^2\rangle$ is form bounded w.r.t.~$\cE^\mu$ with bound $<1$  in the sense that
\begin{equation}\label{pert-bound}
\big|\langle\kappa_0,f^2\rangle\big|\le \alpha\, \cE^\mu(f,f)+\beta\,\|f\|_{L^2}^2\qquad(\forall f\in\fD(\cE^\mu))
\end{equation}
for some $\alpha,\beta\in\R_+, \alpha<1$. Then the form $\big(\dot\cE^\kappa, \fD(\dot\cE^\kappa)\big)$ as introduced in the previous Theorem is closable and
\[
\fD(\cE^{\kappa})=\Big\{f\in\cF:\ \int_\X \tilde f^2\,d\mu<\infty\Big\}.
\]
\end{Theorem}

\begin{proof} Form boundedness with bound $< 1$ implies that
\[\cE^\kappa(f):=\cE^\mu(f)+\langle\kappa_0,f^2\rangle\]
can be defined as closed form with $\fD(\cE^{\kappa})=\fD(\cE^{\mu})$ (see \cite{Reed-Simon}).

It remains to prove that the form $\cE^\kappa$ defined in this way coincides with the form $\dot\cE^\kappa$ (as defined in the previous Theorem) on the domain of the latter. Without restriction, we may choose the quasi-open nest $(G_n)_n$ and the quasi-open nests $(G_{n,\ell})_\ell$ in $G_n$ in such a way that (in addition to the requested properties as formulated in the previous Theorem) also $\mu\in \bigcap_n \cF^{-1}(G_n)$ and $[|\phi_n|+ 2 \Gamma(\phi_n)]\,\mathds{1}_{G_{n,\ell}}\in \K_{1-}(G_n)$ for each $n,\ell$ where $\phi_n = (-\sfL_{G_n}+1)^{-1}\mu$.
Then following the proof of the previous Proposition, we conclude that  $\fD(\dot\cE^\kappa)\subset \fD(\cE^\mu)$ and 
\[
\cE^\mu(f)+\langle\kappa_0,f^2\rangle=\cE(f)+\int_\X\tilde f^2 \, \d\mu+\cE_1(f^2,\psi_n^0)=\dot\cE(f)
\]
for all $n\in\N$ and $f\in \cF_0(G_n)$ where $\psi_n^0=(-\sfL_{G_n}+1)^{-1}\kappa_0$.
\end{proof}

\begin{Corollary} \label{cor:domDsm0} 
Assume that $\kappa\in \cF^{-1}_{\rm qloc}$  and that $f\mapsto\langle\kappa_0,f^2\rangle$ is form bounded w.r.t.~$\cE$ with bound $<1$. Then the form $\big(\dot\cE^\kappa, \fD(\dot\cE^\kappa)\big)$ is closable and $\fD(\cE^{\kappa})= \cF$.
\end{Corollary}

\begin{Corollary}\label{cor:domDsm} 
Assume that $\kappa\in \cF^{-1}_{\rm qloc}$ admits a decomposition $\kappa=\mu-\nu$ with $\mu,\nu\in \cF^{-1}_{\rm qloc}$, $\mu,\nu\ge0$ such that $\nu$ belongs to the Kato class $\KK_0(\mu)$ (or extended Kato class $\KK_{1-}(\mu)$) w.r.t.~the Dirichlet form $\cE^\mu$. Then the form $\big(\dot\cE^\kappa, \fD(\dot\cE^\kappa)\big)$ is closable and $\fD(\cE^{\kappa})=\Big\{f\in\cF:\ \int_\X \tilde f^2\,\d\mu<\infty\Big\}$.
\end{Corollary}

\section{Tamed Spaces}\label{sec:tamed spaces}
 
In this chapter, we introduce the notion of taming for a Dirichlet space, via an extension of the classical $L^1$-Bochner inequality to distribution-valued Ricci bounds. We show that it is equivalent to an $L^1$-gradient estimate for the semigroup and that it implies corresponding $L^2$ versions of the Bochner inequality and gradient estimate. Moreover, we show that under a metric completeness assumption on the space it implies stochastic completeness.

\smallskip

Throughout this chapter, we fix a strongly local, quasi-regular Dirichlet space $(\X,\cE,\mm)$ admitting a carr\'e du champ $\Gamma$. In particular, $\cE(f)=\frac12\int_\X \Gamma(f)\,\d\mm$ for all $f\in\cF:=\fD(\cE)$.

\subsection{The Taming Condition}\label{subsec:taming}

\begin{Definition}[$L^1$-Bochner inequality]\label{def:L1Bochner}
  Given a moderate distribution $\kappa\in \cF^{-1}_{\rm qloc}$ and
  $N\in[1,\infty]$, we say that the Bochner inequality
  $\BE_1(\kappa,N)$ holds, if for all $f \in \fD_{\cF}(\sfL)$ and all non-negative $\varphi \in \fD(\sfL^{\kappa/2})$
\begin{equation}\label{eq:L1Bochner}
\int \sfL^{\kappa/2} \varphi \Gamma(f)^{1/2} \d\mm - \int \varphi \frac{\Gamma(f,\sfL f)}{\Gamma(f)^{1/2}}\d\mm \geq \frac{2}{N}\int \varphi\frac{(\sfL f)^2}{\Gamma(f)^{1/2}}\,\d\mm\;,
\end{equation}
where the right-hand side is read as $0$ if $N=\infty$.
\end{Definition}

Here the first integral is considered over the whole space $\X$, whereas the second and third integrals in \eqref{eq:L1Bochner} are intended to be taken over the set $\{\Gamma(f)>0\}$. This is consistent, since $\Gamma(f,\sfL f)$ and, by locality of $\cE$, also $\sfL f$ vanish a.e.~on $\{\Gamma(f)=0\}$. Similarly in the sequel, we will implicitely intend such integrals to be taken over the suitable set.

\begin{Definition}[Taming] We say that the Dirichlet space
  $(\X,\cE,\mm)$ is \emph{tamed} if there exists a moderate
  distribution $\kappa\in\cF^{-1}_{\rm qloc}$ such that
  $\BE_1(\kappa,\infty)$ holds. In this case, $\kappa$ will be called
  \emph{distribution-valued lower Ricci bound} or \emph{taming
    distribution} for the Dirichlet space $(\X,\cE,\mm)$. If moreover
  this $\kappa$ is also $p$-moderate for some $p\in[1,\infty)$, then
  the space is called \emph{$p$-tamed}.  $(P^{p\kappa/2}_t)_{t\ge0}$
  will be called \emph{$p$-taming semigroup} and $\cE^{p\kappa/2}$
  will be called \emph{$p$-taming energy form} for $(\X,\cE,\mm)$.
 \end{Definition}

 We will show that the taming condition is equivalent to gradient
 estimates involving the semigroup and the taming semigroup.
 
 \begin{Definition}[$L^1$-gradient estimates] For a moderate
   distribution $\kappa\in\cF^{-1}_{\rm qloc}$ we say that
   $\GE_1(\kappa, \infty)$ is satisfied if for any $f \in \cF$ and
   any $t > 0$:
\begin{equation}\label{eq:L1gradient-infty}
\Gamma(P_t f)^{1/2} \leq P_t^{\kappa/2} \Gamma(f)^{1/2}\;.
\end{equation}
Moreover, given $N\in[1,\infty)$, we say that $\GE_1(\kappa, N)$ is satisfied
  if for any $f \in \cF$ and any $t > 0$:
\begin{equation}\label{eq:L1gradient}
\Gamma(P_t f)^{1/2} + \frac{2}{N}\int_0^t P_s^{\kappa/2}\Big(\frac{(\sfL P_{t-s} f)^2}{\Gamma(P_{t-s} f)^{1/2}}\Big) \, \d s \leq P_t^{\kappa/2} \Gamma(f)^{1/2}\;.
\end{equation}
\end{Definition}

Note, that for $N<\infty$, it is part of the assumption that the
second term on the left-hand side is finite. 

\begin{Theorem}\label{thm:BE1GE1}
  For a moderate distribution $\kappa\in\cF^{-1}_{\rm qloc}$ and
  $N\in[1,\infty]$, the Bochner inequality $\BE_1(\kappa,N)$ is
  equivalent to the gradient estimate $\GE_1(\kappa,N)$.
\end{Theorem}

\begin{proof}
\noindent{\bf  $\BE_1\Rightarrow\GE_1$}: Fix $f \in \cF$, $\varphi \in \fD(\sfL^{\kappa/2})$, $t>0$ and set
\[
\Phi(s) := \int \Gamma(P_{t-s}f)^{1/2} P_s^{\kappa/2}\varphi\,\d\mm, \qquad s \in [0,t],
\]
which is well defined since $\Gamma(P_{t-s}f)^{1/2}$ and $P_s^{\kappa/2}\varphi$ belong to $L^2(\X,\mm)$. Moreover, the continuity in $\cF$ of $s \mapsto P_{t-s}f$ and the continuity in $L^2(\X,\mm)$ of $s \mapsto P_s^{\kappa/2}\varphi$ ensure that $\Phi$ is continuous on $[0,t]$. In order to prove that $\Phi$ is actually $\mathcal{C}([0,t]) \cap \mathcal{C}^1([0,t))$, notice that
\begin{equation}\label{eq:pointwise}
\lim_{h \to 0}\frac{\Gamma(P_{t-(s+h)}f)^{1/2} - \Gamma(P_{t-s}f)^{1/2}}{h} = -\frac{\Gamma(P_{t-s}f,\sfL P_{t-s}f)}{\Gamma(P_{t-s}f)^{1/2}} \quad \mm\textrm{-a.e. in } \X
\end{equation}
for $s \in [0,t)$ and,
\[
\bigg|\frac{\Gamma(P_{t-(s+h)}f)^{1/2} - \Gamma(P_{t-s}f)^{1/2}}{h}\bigg|^2 \leq \Gamma\Big(\frac{P_{t-(s+h)}f - P_{t-s}f}{h}\Big)\;.
\]
Since the right-hand side is convergent in $L^1(\X,\mm)$ as $h \to
0$, by \eqref{eq:pointwise} and dominated convergence we deduce that
\[
\lim_{h \to 0}\frac{\Gamma(P_{t-(s+h)}f)^{1/2} - \Gamma(P_{t-s}f)^{1/2}}{h} = -\frac{\Gamma(P_{t-s}f,\sfL P_{t-s}f)}{\Gamma(P_{t-s}f)^{1/2}} \quad \textrm{strongly in } L^2(\X,\mm)
\]
for $s \in [0,t)$. Since in addition
\[
\lim_{h \to 0} \frac{P_{s+h}^{\kappa/2}\varphi - P_s^{\kappa/2}\varphi}{h} = \sfL^{\kappa/2}P_s\varphi \quad \textrm{strongly in } L^2(\X,\mm)
\]
for all $s \geq 0$, we precisely get $\Phi \in \mathcal{C}([0,t]) \cap \mathcal{C}^1([0,t))$ with
\[
\Phi'(s) = \int \Gamma(P_{t-s}f)^{1/2}\sfL^{\kappa/2} P_s^{\kappa/2} \varphi\,\d\mm - \int\frac{\Gamma(P_{t-s}f,\sfL P_{t-s}f)}{\Gamma(P_{t-s}f)^{1/2}} P_s^{\kappa/2}\varphi\,\d\mm\;.
\]
Now observe that $P_{t-s}f \in \fD_\cF(\sfL)$, and $P_s^{\kappa/2}\varphi \in \fD(\sfL^{\kappa/2})$, so that by $\BE_1(\kappa,N)$ we obtain
\[
\Phi'(s) \geq \frac{2}{N}\int P_s^{\kappa/2}\varphi \frac{(\sfL P_{t-s}f)^2}{\Gamma(P_{t-s}f)^{1/2}}\,\d\mm\;,
\]
for all $s \in [0,t)$ and integration in time together with symmetry of $P_t^{\kappa/2}$ in $L^2(\X,\mm)$ yields
\begin{equation}\label{eq:integratedGE1}
\int \varphi P_t^{\kappa/2}\big(\Gamma(f)^{1/2}\big)\,\d\mm - \int \varphi\Gamma(P_tf)^{1/2}\,\d\mm \geq \frac{2}{N} \int_0^t \int \varphi P_s^{\kappa/2}\frac{(\sfL P_{t-s}f)^2}{\Gamma(P_{t-s}f)^{1/2}}\,\d\mm\d s\;.
\end{equation}
By the arbitrariness of $\varphi$, \eqref{eq:L1gradient} follows.

\noindent{\bf $\GE_1\Rightarrow\BE_1$}: Choose $f$ and $\varphi$ as in Definition \ref{def:L1Bochner}, and fix $t>0$. Write \eqref{eq:L1gradient} with $h>0$ in place of $t$ and for a function of the form $P_{t-s}f$, for some $s \in (h,t)$; then multiply both sides of the inequality by $P_{s-h}^{\kappa/2}\varphi$ and integrate w.r.t.\ $\mm$, so that by the self-adjointness of $P_s^{\kappa/2}$ in $L^2(\X,\mm)$ we obtain
\[
\begin{split}
\int P_s^{\kappa/2}\varphi \Gamma(P_{t-s}f)^{1/2}\,\d\mm & - \int P_{s-h}^{\kappa/2}\varphi \Gamma(P_{t-(s-h)}f)\,\d\mm \\
& \geq \frac{2}{N}\int_0^h\int \frac{(\sfL P_{t+h-s-r} f)^2}{\Gamma(P_{t+h-s-r} f)^{1/2}} P_{s-h+r}^{\kappa/2}\varphi\, \d \mm\,\d r\;.
\end{split}
\]
Arguing as in the first part of the proof, the left-hand side is
absolutely continuous as a function of $s \in (h,t)$, hence locally
absolutely continuous in $(0,t)$. Moreover, the integral in the
right-hand side of \eqref{eq:integratedGE1} is finite. Therefore,
using Lebesgue density theorem, if we divide by $h>0$ the inequality
above and let $h \downarrow 0$, we get
\[
\begin{split}
\int \sfL^{\kappa/2} \big(P_s^{\kappa/2}\varphi\big) \Gamma(P_{t-s}f)^{1/2}\,\d\mm & - \int P_s^{\kappa/2}\varphi \frac{\Gamma(P_{t-s}f,\sfL P_{t-s}f)}{\Gamma(P_{t-s}f)^{1/2}}\,\d\mm \\
& \geq \frac{2}{N} \int P^{\kappa/2}_s\varphi \frac{(\sfL P_{t-s} f)^2}{\Gamma(P_{t-s}f)^{1/2}} \,\d\mm\;,
\end{split}
\]
for a.e.\ $s \in (0,t)$. Since $f \in \fD_\cF(\sfL)$ and $\varphi \in \fD(\sfL^{\kappa/2})$, if we let $s,t \downarrow 0$, then the left-hand side above converges to
\[
\int \sfL^{\kappa/2}\varphi \Gamma(f)^{1/2}\,\d\mm - \int \varphi \frac{\Gamma(f,\sfL f)}{\Gamma(f)^{1/2}}\,\d\mm\;,
\]
while, up to extract a subsequence along which $\mm$-a.e.\ convergence is satisfied, by Fatou's lemma it holds
\[
\liminf_{s,t \downarrow 0} \int P^{\kappa/2}_s\varphi \frac{(\sfL P_{t-s} f)^2}{\Gamma(P_{t-s}f)^{1/2}} \,\d\mm \geq \int \varphi \frac{(\sfL f)^2}{\Gamma(f)^{1/2}} \,\d\mm\;,
\]
whence \eqref{eq:L1Bochner}.
\end{proof}

\subsection{$L^2$-Bochner Inequality and Gradient Estimate}

We show that in analogy to Theorem \ref{thm:BE1GE1} an $L^2$-version
of the Bochner inequality is equivalent to an $L^2$-gradient
estimate. Further we show that these two equivalent properties are
implied by the taming condition i.e.~the $L^1$-Bochner inequality,
provided the taming distribution is also $2$-moderate.

\smallskip

We set $\fD_{L^\infty}(\sfL^\kappa):=\{f\in \fD(\sfL^\kappa)\,:\, f,\sfL^\kappa f\in L^\infty(\X,\mm)\}$.

\begin{Definition}[$L^2$-Bochner inequality]\label{def:L2Bochner}
Given a $2$-moderate distribution $\kappa\in\cF^{-1}_{\rm qloc}$ and $N \in [1,\infty]$, we say that the $L^2$-Bochner inequality $\BE_2(\kappa,N)$ holds, if 
\begin{equation}\label{eq:L2Bochner}
  \int \sfL^{\kappa} \varphi \Gamma(f)\,\d\mm - 2\int \varphi \Gamma(f,\sfL f)\,\d\mm \geq \frac{4}{N}\int\varphi (\sfL f)^2\,\d\mm\;,
\end{equation}
for all $f\in \fD_{\cF}(\sfL)$ and $\varphi\in \fD_{L^\infty}(\sfL^{\kappa})$.
\end{Definition}

\begin{Theorem}\label{thm:BE2GE2}
For a $2$-moderate distribution $\kappa\in\cF^{-1}_{\rm qloc}$ and
  $N\in[1,\infty]$, the following are equivalent:
\begin{itemize}
\item[(1)] The Bochner inequality $\BE_2(\kappa,N)$ holds.
\item[(2)] The gradient estimate $\GE_2(\kappa,N)$ holds:
\begin{equation}
\label{eq:L2gradient}
\Gamma(P_t f) + \frac{4}{N}\int_0^t P_s^{\kappa}(\sfL P_{t-s} f)^2\, \d s \leq P_t^{\kappa}(\Gamma(f)), \qquad \forall f \in \cF,\, t > 0\;.
\end{equation}
\item[(2')] We have:
\begin{equation}\label{eq:L2gradient2}
\Gamma(P_t f) + \frac{4}{N}\int_0^t \big(P_s^{\kappa/2}(\sfL P_{t-s} f)\big)^2 \, \d s \leq P_t^{\kappa}(\Gamma(f)), \qquad \forall f \in \cF,\, t > 0\;.
\end{equation}
\end{itemize}
If moreover $-\kappa$ is also $2$-moderate, then the previous properties are eqivalent to:
\begin{itemize}
\item[(2'')] We have:
\begin{equation}
\label{eq:L2gradient3}
\Gamma(P_t f) + \frac{4t}{N C_{t}}(\sfL P_t f)^2 \leq P_t^{\kappa}(\Gamma(f)), \qquad \forall f \in \cF,\, t > 0\;,
\end{equation}
where $C_t=\sup_{s\in[0,t]}C^{-\kappa}_s$ and the constants $C_s^{-\kappa}$ are given by
\eqref{eq:Holder-const}.
\end{itemize}
\end{Theorem}

\begin{proof}
\noindent{\bf (1) $\Rightarrow$ (2)}. Fix $f \in \cF$, $\varphi \in \fD_{L^\infty}(\sfL^{\kappa})$, $t>0$ and set
\[
\Phi(s) := \int \Gamma(P_{t-s}f) P_s^{\kappa}\varphi\,\d\mm\;, \qquad s \in [0,t].
\]
The fact that $P_s^{\kappa}$ maps $L^\infty(\X,\mm)$ into itself ensures
that $\Phi$ is well defined, while the fact that $s \mapsto P_{t-s}f$
is continuous with values in $\cF$ and the weak-* continuity in
$L^\infty(\X,\mm)$ of $s \mapsto P_s^{\kappa}\varphi$ ensure that
$\Phi$ is continuous on $[0,t]$. Since
$\varphi \in \fD_{L^\infty}(\sfL^{\kappa})$, we have
\[
\lim_{h \to 0} \frac{P_{s+h}^{\kappa}\varphi - P_s^{\kappa}\varphi}{h} = \sfL^{\kappa}P_s\varphi \quad \textrm{weakly-* in } L^\infty(\X,\mm)\;,
\]
for $s \in [0,t)$. Since in addition
\[
\lim_{h \to 0}\frac{\Gamma(P_{t-(s+h)}f) - \Gamma(P_{t-s}f)}{h} = -2\Gamma(P_{t-s}f,\sfL P_{t-s}f) \quad \textrm{strongly in } L^1(\X,\mm)\;,
\]
for $s \in [0,t)$ one obtains $\Phi \in \mathcal{C}([0,1]) \cap \mathcal{C}^1([0,t))$ with
\[
\Phi'(s) = \int \Gamma(P_{t-s}f) \sfL^{\kappa}P_s^{\kappa}\varphi\,\d\mm - 2\int \Gamma(P_{t-s}f,\sfL P_{t-s}f) P_s^{\kappa}\varphi\,\d\mm\;.
\]
Now observe that $P_{t-s}f \in \fD_\cF(\sfL)$. Moreover $P_s^{\kappa}\varphi \in \fD_{L^\infty}(\sfL^{\kappa})$ because $P_s^{\kappa}$ is continuous from $L^\infty(\X,\mm)$ into itself, $\sfL^{\kappa}$ and $P_s^{\kappa}$ commute and $\sfL^{\kappa}\varphi \in L^\infty(\X,\mm)$ by assumption. Hence by (1) we deduce that 
\[
\Phi'(s) \geq \frac{4}{N}\int P_s^{\kappa}\varphi (\sfL P_{t-s}f)^2\,\d\mm\;,
\]
for all $s \in [0,t)$. By integration this yields
\[
\int \varphi P_t^{\kappa}(\Gamma(f))\,\d\mm - \int \varphi \Gamma(P_t f)\,\d\mm \geq \frac{4}{N}\int_0^t\int \varphi P_s^{\kappa}(\sfL P_{t-s}f)^2\,\d\mm\,\d s\;,
\]
and by the arbitrariness of $\varphi$, \eqref{eq:L2gradient} follows.

\noindent{\bf (2) $\Rightarrow$ (2')}. This follows from Jensens's
inequality, see Lemma \ref{lem:Jensen}.

\noindent{\bf (2') $\Rightarrow$ (1)}. Choose $f$ and $\varphi$ as in
Definition \ref{def:L2Bochner}, and fix $t > 0$. Write the equations
\eqref{eq:L2gradient2} with $h > 0$ in
place of $t$ and taking a function of the form $P_{t-s} f$, for some
$s \in (h,t)$; then multiply both sides of the inequality by
$P_{s-h}^{\kappa}\varphi$, and integrate w.r.t.\ $\mm$:
\[
\int P_s^{\kappa}\varphi \Gamma(P_{t-s}f)\,\d\mm - \int P_{s-h}^{\kappa}\varphi \Gamma(P_{t-(s-h)}f)\,\d\mm \geq \frac{4}{N}\int_0^h\int \big(P_{r}^{\kappa/2}(\sfL P_{t+h-s-r)} f)\big)^2 P^{\kappa}_{s-h} \varphi \,\d\mm\d r\;.
\]
Arguing as in the the first part of the proof, we see that the
left-hand side is absolutely continuous as a function of
$s \in (h,t)$, hence locally absolutely continuous in $(0,t)$. Hence
if we divide by $h>0$ the inequality above and let $h \downarrow 0$,
we obtain
\[
\int \sfL^{\kappa} \big(P_s^{\kappa}\varphi\big) \Gamma(P_{t-s}f)\,\d\mm - 2\int P_s^{\kappa}\varphi \Gamma(P_{t-s}f,\sfL P_{t-s}f)\,\d\mm \geq \frac{4}{N} \int (\sfL P_{t-s} f)^2 P^{\kappa}_s\varphi \, \d \mm\;,
\]
where we have used the continuity of the curve
$t \mapsto P^{\kappa}_t \varphi \in L^2(\X,\mm)$ in $[0, \infty)$.
At this point, since $f \in \fD_\cF(\sfL)$ and
$\varphi \in \fD_{L^\infty}(\sfL^{\kappa})$, we can let
$s,t \downarrow 0$, thus getting \eqref{eq:L2Bochner}.

\noindent{\bf (2) $\Leftrightarrow$ (2'')}. For $f \in \cF$, Lemma \ref{lem:Holder} provides
\[
P_s^{\kappa}(\sfL P_{t-s}f)^2 \geq \frac{1}{C_{s}^{-\kappa}}(\sfL P_t f)^2\;.
\]
Plugging this inequality into \eqref{eq:L2gradient} we obtain
\eqref{eq:L2gradient3}.  On the other hand, since $C_t^{-\kappa}$
tends to $1$ as $t\to0$, we can argue as in the proof of
(2') $\Rightarrow$ (1) to see that (2'') implies (1).
\end{proof}

\begin{Proposition}[$\GE_1(\kappa, N)$ implies $\GE_2(\kappa, N)$]\label{GE1toGE2}
  Given a $2$-moderate distribution $\kappa\in \cF^{-1}_{\rm qloc}$,
  the condition $\GE_1(\kappa, N)$ implies the condition
  $\GE_2(\kappa, N)$.
\end{Proposition}

\begin{proof}
We start observing that a direct application of $\GE_1(\kappa, N)$ yields
\begin{equation}\label{eq:GE1to2}
\begin{split}
  \Gamma(P_t f) &=\big( \Gamma(P_t f)^{1/2} \big)^2 \le \bigg( P^{\kappa/2}_t \Gamma(f)^{1/2} - \dfrac{2}{N} \int_0^t P_s^{\kappa/2} \dfrac{(\sfL P_{t-s} f)^2}{\Gamma (P_{t-s} f)^{1/2}} \, \d s \bigg)^2\\
  &= \big( P^{\kappa/2}_t \Gamma(f)^{1/2} \big)^2 - \dfrac{4}{N} \int_0^t P_s^{\kappa/2} \dfrac{(\sfL P_{t-s} f)^2}{\Gamma (P_{t-s} f)^{1/2}} P^{\kappa/2}_t \Gamma(f)^{1/2} \, \d s\\
  &\quad+  \bigg( \dfrac{2}{N} \int_0^t P_s^{\kappa/2} \dfrac{(\sfL P_{t-s} f)^2}{\Gamma (P_{t-s} f)^{1/2}} \, \d s \bigg)^2.
\end{split}
\end{equation}
At this point, Jensen's inequality and the $1$-homogeneity of $(s,t)\mapsto s^2/t$ guarantee, see Lemma \ref{lem:Jensen}, that
\[
\big( P^{\kappa/2}_t \Gamma(f)^{1/2} \big)^2 \le P^{\kappa}_t \Gamma(f) \quad \text{and} \quad P^{\kappa/2}_s \bigg( \dfrac{\big(\sfL P_{t-s} f\big)^2}{\Gamma(P_{t-s} f)^{1/2}} \bigg) \ge \dfrac{(P_s^{\kappa/2} \sfL P_{t-s} f)^2}{P_s^{\kappa/2} \Gamma(P_{t-s} f)^{1/2}}\;.
\]
A further direct application of $\GE_1(\kappa, N)$ provides
\[
\begin{split}
\dfrac{4}{N} \int_0^t P_s^{\kappa/2} &\dfrac{(\sfL P_{t-s} f)^2}{\Gamma (P_{t-s} f)^{1/2}} P^{\kappa/2}_t \Gamma(f)^{1/2} \, \d s\\
&\ge \dfrac{4}{N} \int_0^t P_s^{\kappa/2} \dfrac{(\sfL P_{t-s} f)^2}{\Gamma (P_{t-s} f)^{1/2}} \bigg( P_s^{\kappa/2} \Gamma(P_{t-s} f)^{1/2} + \dfrac{2}{N} \int_0^s P^{\kappa/2}_r \dfrac{(\sfL P_{t-r} f)^2}{\Gamma(P_{t-r} f)^{1/2}} \, \d r \bigg)\, \d s\\	
&\ge \dfrac{4}{N} \int_0^t \dfrac{(P_s^{\kappa/2} \sfL P_{t-s} f)^2}{P_s^{\kappa/2} \Gamma(P_{t-s} f)^{1/2}} P_s^{\kappa/2} \Gamma(P_{t-s} f)^{1/2}\,\d s \\
&\qquad + \dfrac{8}{N^2} \int_0^t \int_0^s P^{\kappa/2}_s \dfrac{(\sfL P_{t-s} f)^2}{\Gamma(P_{t-s} f)^{1/2}} P^{\kappa/2}_r \dfrac{(\sfL P_{t-r} f)^2}{\Gamma(P_{t-r} f)^{1/2}} \, \d r \, \d s\\
&= \dfrac{4}{N} \int_0^t (P_s^{\kappa/2} \sfL P_{t-s} f)^2 \, \d s + \dfrac{8}{N^2} \int_0^t \int_0^s P^{\kappa/2}_s \dfrac{(\sfL P_{t-s} f)^2}{\Gamma(P_{t-s} f)^{1/2}} P^{\kappa/2}_r \dfrac{(\sfL P_{t-r} f)^2}{\Gamma(P_{t-r} f)^{1/2}} \, \d r \, \d s\\
&= \dfrac{4}{N} \int_0^t (P_s^{\kappa/2} \sfL P_{t-s} f)^2 \, \d s + \dfrac{4}{N^2} \int_0^t \int_0^t P^{\kappa/2}_s \dfrac{(\sfL P_{t-s} f)^2}{\Gamma(P_{t-s} f)^{1/2}} P^{\kappa/2}_r \dfrac{(\sfL P_{t-r} f)^2}{\Gamma(P_{t-r} f)^{1/2}} \, \d r \, \d s\\
&=  \dfrac{4}{N} \int_0^t (P_s^{\kappa/2} \sfL P_{t-s} f)^2 \, \d s + \bigg(\dfrac{2}{N}\int_0^t  P^{\kappa/2}_s \dfrac{(\sfL P_{t-s} f)^2}{\Gamma(P_{t-s} f)^{1/2}} \, \d s\bigg)^2\;.
\end{split}
\] 
Plugging these inequalities in \eqref{eq:GE1to2}, we obtain
\[
\Gamma(P_t f) \le P^{\kappa}_t \Gamma(f) - \dfrac{4}{N} \int_0^t (P_s^{\kappa/2} \sfL P_{t-s} f)^2 \, \d s\;,
\]
which is exactly \eqref{eq:L2gradient2}.
\end{proof}

\subsection{Stochastic Completeness}

We show that the taming condition, together with an appropriate metric
completeness assumption on the Dirichlet space, implies the
stochastic completeness of the semigroup.

\begin{definition}[Intrinsic completeness]\label{def:intrinsic-complete}
  We say that the Dirichlet space $(\X,\cE,\mm)$ is
  \emph{intrisically complete} if there exists a sequence of functions
  $(\eta_k)_k$ in $\fD(\mathcal E)$ such that 
  $\mm(\{\eta_k>0\})<\infty$ as well as $0\leq\eta_k\leq 1, \Gamma(\eta_k)\le1$, and $\eta_k\to 1$,
  $\Gamma(\eta_k)\to0$ $\mm$-a.e.~as $k\to\infty$.
\end{definition}

\begin{remark}\label{rem:complete}
  This terminology is motivated by the fact that the existence of such
  cut-off functions is strongly related to properties of the intrinsic
  metric of the Dirichlet form $\mathcal E$. Recall the latter is
  defined by
  \begin{align*}
    \rho(x,y)= \sup\big\{u(x)-u(y) : u\in \cF_{\rm{loc}}\cap \mathcal{C}(\X),\d\Gamma(u)\leq\d\mm\big\}\;.
  \end{align*}
  In general, $\rho$ might be degenerate, in the sense that
  $\rho(x,y)=+\infty$ or $\rho(x,y)=0$ for some $x\neq y$. Let us
  assume that the topology induced by the pseudo-distance $\rho$ is
  equivalent to the original topology on $\X$. When $\X$ is locally
  compact, $(\X,\rho)$ is a complete metric space if and only if all
  balls $B_r(x)=\{y\in\X:\rho(x,y)<r\}$ are relatively compact, see
  \cite[Theorem 2]{Sturm95}. In
  this case, cut-off functions as in Definition
  \ref{def:intrinsic-complete} can be constructed by considering
  functions of the form
  \begin{align*}
    \rho_{x,a,b}: y\mapsto \big(a-b\rho(x,y)\big)_+\;,
  \end{align*}
  for $x\in\X$ and suitable $a,b>0$. Indeed, by
  \cite[Lemma 1]{Sturm96I} the distance function
  $\rho_x:y\mapsto\rho(x,y)$ satisfies $\d\Gamma(\rho_x)\leq
  \d\mm$. Thus, $\rho_{x,a,b}\in\mathcal \fD(\mathcal E)\cap \mathcal{C}_c(\X)$
  and $\d\Gamma(\rho_{x,a,b})\leq b\d\mm$.

  However, in contrast to conservativeness (or 'stochastic completeness'), completeness with respect to the intrinsic (pseudo-) distance is not invariant under quasi-isomorphisms of Dirichlet forms. Our more general, new notion of intrinsic completeness perfectly makes sense for  arbitrary quasi-regular, strongly local Dirichlet forms on general state spaces and it is invariant under quasi-isomorphisms of Dirichlet forms.
  \end{remark}

We consider here the following more general notion of taming. 

\begin{definition}\label{def:taming-general}
  We say that the Dirichlet space $(\X,\mathcal E,\mm)$ is \emph{weakly tamed} if there exists 
   an exponentially bounded semigroup $(Q_t)$ on $L^1(\X,\mm)$. 
such that the following gradient estimate holds:
  \begin{align}\label{eq:taming-general}
    \sqrt{\Gamma(P_tf)} \leq Q_t\sqrt{\Gamma(f)}\;,\quad\text{ for all }f\in \cF\;.
  \end{align} \end{definition}

  An example of the previous situation is given by a space
  $(\X,\mathcal E,\mm)$ tamed by a moderate distribution
  $\kappa\in \cF^{-1}_{\rm qloc}$. Here the taming semigroup is
  $Q_t=P^{\kappa/2}_t$ and \eqref{eq:taming-general} is nothing but the
  condition $\GE_1(\kappa,\infty)$.

\begin{theorem}[Stochastic completeness]
  Assume that the Dirichlet space $(\X,\mathcal E,\mm)$ is
  intrinsically complete and weakly tamed. Then
  the heat semigroup $(P_t)_{t \geq 0}$ is stochastically complete, i.e.~we
  have $P_t\mathds{1}\equiv \mathds{1}$ for all $t>0$.
\end{theorem}

\begin{proof}
  It suffices to show that $\int P_tu\,\d\mm=\int u\,\d\mm$ for
  every non-negative $u\in L^1(\X,\mm)\cap L^2(\X,\mm)$. To this end, let $\eta_n$
  be a sequence of cut-off functions as in Definition
  \ref{def:intrinsic-complete}. 
  Approximating $u$ by $P_\varepsilon (\eta_ku)$
  we can assume that
  $u\in \fD(\sfL)\subset \fD(\mathcal E)$ and (thanks to \eqref{eq:taming-general}) also that $\Gamma(u)^{1/2}\in
  L^1(\X,\mm)$. 
  Then we have, using the gradient estimate \eqref{eq:taming-general}:
  \begin{align*}
    \int \eta_kP_tu\d\mm - \int \eta_ku\d\mm &= \int_0^t \int\eta_k\sfL P_su\,\d\mm\,\d s
    = -\int_0^t\int\Gamma(P_s\eta_k,u)\d\mm\,\d s\\
  &\leq \int_0^t\int \sqrt{\Gamma(\eta_k)}Q_s\sqrt{\Gamma(u)}\,\d\mm\,\d s\;.
  \end{align*}
  Now, as $k\to\infty$ the last expression goes to zero, since
  $\sqrt{\Gamma(u)}\in L^1(\X,\mm)$, and $\Gamma(\eta_k)$ is uniformly bounded
  and goes to $0$.
\end{proof}

\section{Examples}\label{sec:examples}

\subsection{A tamed manifold with lower Ricci bound that is nowhere Kato}

We will consider a time change of the Euclidean space which yields a
Dirichlet space which is tamed but the pointwise lower bound of the
Bakry-\'Emery-Ricci curvature of which is nowhere locally in the Kato
class.

For $n\in\N, n\ge 2$, let $(X,\d,\mm)$ be the Euclidean space $\R^n$
equipped with the classical Dirichlet energy $\cE$ and the
$n$-dimensional Lebesgue measure $\mm$.  
For $j\in\N$ choose increasing functions $\vartheta_j\in{\mathcal C}^\infty(\R_+,\R_+)$ with $\vartheta_j'\le1$ and
$$\vartheta_j(r)=\left\{\begin{array}{ll}
\frac2j,\quad&0\le r\le\frac1{3j},\\
r,& \frac1j\le r\le1\\
2,&r\ge3
\end{array}\right.$$
and put $\vartheta(r):=\lim_{j\to\infty}\vartheta_j(r)$.
Given real numbers
$m,\ell>0$, put
$$\Psi(r):=r^{2+2m-\ell}\cdot \sin\big(r^{-m}\big)$$
and $\psi^*(x):=\Psi(\vartheta(|x|))$.
Moreover, given a sequence $(z_i)_{i\in\N}$ of points in $X$ and a
summable sequence $(\lambda_i)_{i\in\N}$ of positive numbers, put
$$\psi(x):=\sum_i \lambda_i \cdot \psi^*(x-z_i).$$

\begin{Theorem} Assume $2\le \ell <m+2$. 
\begin{itemize}
\item[(i)] Then the function $$k:=-(n-2)|\nabla\psi|^2-\Delta\psi$$ is
  a moderate distribution and it is not in the Kato class
  $\K_0(\R^n)$. If $(z_i)_{i\in\N}$ is dense in $\R^n$, then $k$ is
  even nowhere locally in the Kato class $\K_0(\R^n)$.
\item[(ii)]
The Dirichlet space $(X,\cE,\mm')$ with $\mm':=e^{2\psi}\,\mm$ satisfies BE$_1(k,\infty)$.
\end{itemize}
\end{Theorem}

\begin{Remark} $(X,\cE,\mm')$ is the Dirichlet space of a weighted
  Riemannian manifold with $M=\R^n$,
  $\text{\sf g}'=e^{2\psi}\text{\sf g}_{euclid}$, and
  $\mm'=e^{-(n-2)\psi}\mm_{\text{\sf g}'}$.  If $n=2$, this is indeed
  a Riemannian manifold.
\end{Remark}

\begin{proof} Without restriction, we assume that
  $\sum_i \lambda_i=1$.

  (i) The singularity of $|\nabla\psi^*|$ at the origin is of the
  order $|x|^{1+m-\ell}$. Thus under the assumption $\ell<m+2$ of the
  Theorem, $|\nabla\psi^*|^2\in\K_0(\R^n)$. Since
  $|\nabla\psi|^2(x)\le \sum_i \lambda_i |\nabla\psi^*|^2(x-z_i)$,
  this implies that also $|\nabla\psi|^2\in\K_0(\R^n)$ and so will be
  $p^2\,|\nabla\psi|^2$ for each $p\in\R$.  Moreover, according to
  Proposition \ref{grad-tame} the latter in turn implies that
  $-p \Delta\psi$ is moderate.  This proves that
  $-(n-2)|\nabla\psi|^2-\Delta\psi$ is moderate.  (Indeed, it is even
  $p$-moderate for each $p\in\R$.)

  On the other hand, the singularity of $\Delta\psi^*$ at the origin
  is of the order $|x|^{-\ell}$ which implies that under the
  assumption $\ell\ge2$ of the Theorem,
  $\Delta\psi^*\not\in\K_0(\R^n)$. If the $(z_i)_{i\in\N}$ are dense,
  this in turn implies that $\Delta\psi$ is nowhere locally in the
  Kato class $\K_0(\R^n)$.

\medskip

(ii) For $j\in\N$, put
d $\psi_j^*(x):=\Psi(\vartheta_j(|x|))$.
and
$$\psi_j(x):=\sum_i \lambda_i \cdot \psi_j^*(x-z_i).$$
Then obviously
$\big\|\psi^*-\psi^*_j\big\|_{L^\infty}\le2\cdot j^{\ell-2-2m}$ and
therefore also
 \begin{equation}\big\|\psi-\psi_j\big\|_{L^\infty}\le  2\cdot j^{\ell-2-2m} . \end{equation}
 Moreover,  $$\big\|\nabla\psi_j^*\big\|_{L^\infty}\le C_{m,\ell}\cdot j^{{\ell-1-m}}$$
which in turn immediately implies
 \begin{equation}\big\|\nabla\psi_j\big\|_{L^\infty}\le C_{m,\ell}\cdot j^{{\ell-1-m}}. \end{equation}
 Thus we have constructed a sequence of bounded Lipschitz functions 
 $\psi_j$ that uniformly converge to $\psi$  as $j\to\infty$.

 According to \cite[Thm.~4.7]{Sturm2019}, for each $j\in\N$,
 the Dirichlet space $(X,\cE,\mm_j)$ with $\mm_j:=e^{2\psi_j}\,\mm$
 satisfies BE$_1(k_j,\infty)$ with
 $k_j:=-(n-2)|\nabla\psi_j|^2-\Delta\psi_j$ (the distributional valued
 Laplacian of $\psi_j$ is indeed given by a function since
 $2+2m-\ell>2-n$). Following the
 argumentation from the proof of \cite[Thm.~4.7]{Sturm2019},
 one can pass to the limit in the associated gradient estimate
 GE$_1(k_j,\infty)$ which yields the estimate GE$_1(k,\infty)$ for the
 Dirichlet space $(X,\cE,\mm')$ with $\mm':=e^{2\psi}\,\mm$.
\end{proof}

\subsection{A manifold which is tamed but not 2-tamed}

\begin{theorem}
  As before, let $(\X,\cE,\mm)$ be the classical Dirichlet space on
  $\R^n$, now with $n=2$, and choose
  $\psi\in{\mathcal C}^\infty(\R^2\setminus\{0\})$, supported in
  $B_2(0)$ such that for $x\in B_1(0)\setminus\{0\}$
$$\psi(x):=a\cdot\Big[-\frac18|x|^{-2m}+\sin\big(|x|^{-m}\big)+\big(1-\frac{n-2}m\big)\,|x|^{m}\,\cos\big(|x|^{-m}\big)
\Big].$$ Then there exists $a_c\in(0,\infty)$ such that
$k:=-\Delta\psi$ is moderate if $a\in (0,a_c)$ and not moderate if
$a>a_c$.  In the former case, the Dirichlet space $(\X,\cE,\mm')$ with
$\mm':=e^{2\psi}\mm$ satisfies {\sf BE}$_1(k,\infty)$.
 \end{theorem}
 
 Indeed, this Dirichlet space is associated with the incomplete
 Riemannian manifold $(\R^2\setminus\{0\},\sfg')$ with the (smooth)
 Riemannian tensor $\sfg'=e^{2\psi}\sfg_{\R^2}$ given as a conformal
 transformation of the Euclidean tensor $\sfg_{\R^2}$ and degenerating
 at the origin. The Ricci curvature at $x\not=0$ is exactly given by
 $k(x)$.

 \begin{proof} (i) Straightforward calculation yields
 $$\frac1a\, k(x)=\frac{-1}a\Delta\psi(x)=m^2\,|x|^{-2-2m}\,\Big[\frac12+\sin\big(|x|^{-m}\big)\Big]+k_1(x)$$
 with $k_1(x)=a'\cdot \sin\big(|x|^{-m}\big)\cdot |x|^{-2}$.
 According to Example \ref{dense-osc}, $k_1$ is $\beta$-moderate for
 all $\beta\in\R$.  Moreover, according to Example \ref{1not2} (ii),
 $k_0(x)=a\,m^2\,|x|^{-2-2m}\,\Big[\frac12+\sin\big(|x|^{-m}\big)\Big]$
 is moderate for sufficiently small $a>0$ and not moderate for large $a$.
 Thus the assertion on moderateness vs. non-moderateness of $k$ follows.
 
 (ii) To verify the {\sf BE}$_1(k,\infty)$ condition for the
 time-changed Dirichlet space, we approximate $\psi$ monotonically
 from above by $\psi_\ell$ which we define by modifying the definition
 of $\psi$ as follows:
 \begin{itemize}
 \item truncate $-\frac18|x|^{-2m}$ at level $-\ell$
 \item replace $\sin\big(|x|^{-m}\big)$ by $+1$ if $|x|^{-m}\ge (2\ell+\frac12)\pi$
 \item depending on the sign of $\big(1-\frac{n-2}m\big)$, replace  $\pm|x|^{m}\, \cos\big(|x|^{-m}\big)$ by $\pm|x|^{m}$ if $|x|^{-m}\ge (2\ell\mp1)\pi$.
 \end{itemize}
 \end{proof}

\subsection{Manifolds with boundary and potentially singular curvature}
\label{sec:RicciKato}

For smooth Riemannian manifolds with boundary the taming distribution
of $(\sfM,\cE_\sfM,\mm)$ is determined by pointwise lower bounds on
the Ricci curvature and a measure-valued contribution coming from the
curvature of the boundary as first shown by Hsu \cite{Hsu02a}, see
also the monograph of Wang \cite[Thm.~3.2.1]{Wang14}, where gradient
estimates in terms of the Schr\"odinger semigroup were shown.

Let $(\sfM,\sfg)$ be a smooth compact Riemannian manifold with
boundary. Let $\mm=\text{vol}_\sfg$ denote the volume measure and let
$\sigma$ denote the surface measure of $\partial \sfM$. Let us denote
by $\sfM^0$ the interior of $\sfM$, and define
\[
\cE_\sfM(f):=\frac12\int_{\sfM^0} |\nabla f|^2\,\d\mm \qquad\text{with}\quad 
\mathcal F:=\fD(\cE_\sfM):=W^{1,2}(\sfM^0)\;,
\]
the canonical Dirichlet form on $\sfM$ with Neumann boundary
conditions.

\begin{theorem}\label{thm:smooth-domain} Let $k:\sfM^0\to\R$ and $\ell:\partial\sfM\to\R$ be continuous
  functions providing lower bounds on the Ricci curvature and the
  second fundamental form of $\partial\sfM$, respectively. Then the
  space $(\sfM,\cE_{\sfM},\mm)$ is tamed by the moderate distribution
  $\kappa=k\cdot\mm + l\cdot\sigma_{\partial\sfM}$,
  i.e.~$\BE_1(\kappa,\infty)$ holds.
\end{theorem}

\begin{proof}
  By Theorem \ref{thm:Kato-bdry}, $\kappa$ belongs to the Kato class
  and hence by Proposition \ref{prop:Kato-moderate-mu} is
  moderate. The gradient estimate $\GE_1(\kappa,\infty)$ is shown in
  \cite[Thm.~5.1]{Hsu02}.
                                     \end{proof}

In the setting of metric measure spaces, examples of tamed spaces with
distributional curvature coming from the boundary have been
constructed by Sturm as subsets of RCD spaces with locally semiconvex
and sufficiently regular boundary, see \cite[Thm.~6.14]{Sturm2019}.
\medskip

In the remainder of this section, we will discuss examples of tamed
spaces with non-smooth boundary giving rise to more singular taming
distributions. For simplicity, these will be realized as subdomains of
Euclidean space. We approach these examples via approximation by
smooth domains. To this end, we first state a general stability result
for the taming condition in this context which also allows for
interior singularities of the metric.

Here, we merely assume that $\sfM$ has a smooth differentiable
structure and not necessarily that $\sfg$ or $\partial\sfM$ are
smooth. However, we require that $\sfM$ can be exhausted up to a polar
set by smooth subdomains on which $\sfg$ is smooth as well as some
uniform control on the taming distributions of the subdomains.

\begin{theorem}\label{thm:stability}
  Let the Ricci curvature of $(\sfM,\sfg)$ (where defined) be bounded
  below by $k:\sfM^0\to\R$ and let the second fundamental form of
  $\partial \sfM$ (where defined) by bounded below by
  $\ell:\partial\sfM\to\R$. Moreover, assume that $\sfM$ is
  \emph{regularly exhaustible}, i.e.~there exists an increasing
  sequence $(\X_n)_{n\in\N}$ of domains $\X_n\subset \sfM^0$ with
  smooth boundary such that $\sfg$ is smooth on $\X_n$ and the
  following properties hold:
  \begin{itemize}
  \item[A1)] The closed sets $(\overline X_n)_n$ constitute a nest for $\cE_\sfM$; 
    \item[A2)] For all compact sets $K\subset\sfM^0$ there exists
  $N\in\N$ s.t.~$K\subset \X_n$ for all $n\geq \N$;  

  \item[A3)] There are lower bounds $\ell_n:\partial \X_n\to \R$ for
    the curvature of $\partial\X_n$ with $\ell_n=\ell$ on
    $\partial \sfM\cap\partial \X_n$ such that the distributions
    $\kappa_n= k\,\mm_{\X_n}+ \ell_n\cdot \sigma_{\partial \X_n}$ are
    uniformly 1- and 2-moderate, i.e.
\begin{align*}
    \sup_n\sup_{t\in[0,1]}\sup_{x\in\X_n}\mathbb E_x^{(n)}\big[e^{-\alpha A^{\kappa_n}_t/2}\big]<\infty\quad \alpha=1,2\;.
  \end{align*}
\end{itemize}

Then the Dirichlet space $(\sfM,\cE_\sfM,\mm)$ satisfies {\sf
  BE}$_1(\kappa,\infty)$ with the moderate distribution
  \[
  \kappa = k\,\mm_{\sfM} + \ell\, \sigma_{\partial \sfM}\;.
   \] 
\end{theorem}

\begin{proof} 
  Let $(\X_n)_{n\in\N}$ be an exhausting sequence. Denote by
  $\mathcal E_{n}$ the standard Dirichlet form on $\X_n$ with Neumann
  boundary conditions and by $(P^{(n)}_t)_t$, and
  $\big((B^{(n)}_t)_{t\geq0},(\mathbb P^{n}_x)_{x\in\X_n}\big)$ the
  associated semigroup and process. Let $(P_t)_t$, and
  $\big((B_t)_{t\geq0},(\mathbb P_x)_{x\in\X_n}\big)$ be the
  corresponding objects for $\sfM^0$. Denote by
  $Z_n=\partial\X_n\cap \sfM^0$ the relative boundary of $\X_n$ in
  $\sfM^0$ and by $\tau^{(n)}_{Z_m}$, $\tau_{Z_m}$ be the first
  hitting time of $Z_n$ by $B^{(n)}$, resp.~$B$. A1) entails that
  $\tau_{Z_n}\nearrow \infty$ as $n\to\infty$ a.s.~under
  $\mathbb P_x$.

  By Theorem \ref{thm:smooth-domain}, $(\X_n,\mathcal E_n,\mm)$
  satisfies the taming condition $\BE_1(\kappa_n,\infty)$. We argue by
  passing to the limit in the equivalent condition
  $\GE_1(\kappa_n,\infty)$.
  
    {\bf i)} We first show convergence of the semigroups
  $P^{(n)}_t$. For every $u\in {\mathcal C}_b(\sfM)$ we have (viewing $u$ as a
  function on $\X_n$ by restriction)
  \begin{align*}
    P^{(n)}_t u (x) \rightarrow P_tu(x) \quad \text{for all } x\in \sfM^0\;.
  \end{align*}
  Indeed, we have
  \begin{align*}
    \Big|P_t u (x)-P^{(n)}_t u (x)\Big|
    \leq &\Big|\mathbb E_x^{(n)}\big[u (B^{(n)}_t)\boldsymbol{1}_{\{\tau^{(n)}_{Z_n}>t\}}\big]-
    \mathbb E_x\big[ u (B_t)\boldsymbol{1}_{\{\tau_{Z_n}>t\}}\big]\Big|\\
    &+||u||_{L^\infty}\Big(\mathbb P^{(n)}_x\big[\tau^{(n)}_{Z_n}\leq t\big] + \mathbb P_x\big[\tau_{Z_n}\leq t\big]\Big)\;.
  \end{align*}
  Now, note that the two expectations coincide, since the processes
  $B^{(n)}, B$ can be realized such that they coincide up to the
  hitting time of $Z_n$. Hence, the first term in the right hand side
  vanishes. Similarly, the latter two probabilities coincide and by
  assumption A1) $\tau_{Z_n}\nearrow \infty$ almost surely under
  $\mathbb P_x$, hence $P_x\big[\tau_{Z_n}\leq t\big]$ vanishes as
  $n\to\infty$.  \smallskip

  {\bf ii)} Next, we show convergence of the tamed semigroups
  $P^{\kappa_n/2}_t$. For every non-negative
  $g\in {\mathcal C}_b(\sfM)$ we have
  \begin{align*}
    P^{\kappa_n/2}_t g (x) \rightarrow P^{\kappa/2}_tg(x) \quad \text{for all } x\in \sfM^0\;.
  \end{align*}
  Indeed, we have by construction that
  \begin{align*}
    P^{\kappa/2}_tg(x)=\uparrow\lim_{n\to\infty}\mathbb E_x\big[e^{-A^{\kappa_n/2}_t}g(B_t)\boldsymbol{1}_{\{\tau_{Z_n}>t\}}\big]\;.
  \end{align*}
  On the other hand, arguing as above, we have
  \begin{align*}
    P^{\kappa_n/2}_tg(x)&=\mathbb E_x^{(n)}\big[e^{-A^{\kappa_n/2}_t}g(B^{(n)}_t)\boldsymbol{1}_{\{\tau^{(n)}_{Z_n}>t\}}\big]
    + \mathbb E_x^{(n)}\big[e^{-A^{\kappa_n/2}_t}g(B^{(n)}_t)\boldsymbol{1}_{\{\tau^{(n)}_{Z_n}\leq t\}}\big]\\
                      &= E_x\big[e^{-A^{\kappa_n/2}_t}g(B_t)\boldsymbol{1}_{\{\tau_{Z_n}>t\}}\big]
                        + \mathbb E_x^{(n)}\big[e^{-A^{\kappa_n/2}_t}g(B^{(n)}_t)\boldsymbol{1}_{\{\tau^{(n)}_{Z_n}\leq t\}}\big]\;,
  \end{align*}
  and it suffices to argue that the last term vanishes. But, we have
  \begin{align*}
    \mathbb E_x^{(n)}\big[e^{-A^{\kappa_n/2}_t}g(B^{(n)}_t)\boldsymbol{1}_{\{\tau^{(n)}_{Z_n}\leq t\}}\big]
    &\leq \mathbb E_x^{(n)}\big[e^{-A^{\kappa_n}_t}\big]^{\frac12}\mathbb P_x\big[\tau^{(n)}_{Z_n}\leq t\big]^{\frac12}||g||_{L^\infty}\\
    &=\mathbb E_x^{(n)}\big[e^{-A^{\kappa_n}_t}\big]^{\frac12}\mathbb P_x\big[\tau_{Z_n}\leq t\big]^{\frac12}||g||_{L^\infty}\;.
  \end{align*}
  By assumption A3), the first factor is uniformly bounded in $n$
  while as above the second factor vanishes as $n\to\infty$.  \medskip

  {\bf iii)} We now argue that $\kappa$ is moderate. Since the
  processes $B^{(n)}$ and $B$ can be assumed to coincide up to
  $\tau_{Z_n}$, we have by construction of the Schr\"odinger semigroup
  \begin{align*}
    \mathbb E_x\big[e^{-A_t^{\kappa/2}}\big]&=\sup_nE_x\big[e^{-A_t^{\kappa/2}}\boldsymbol{1}_{\{\tau_{Z_n}>t\}}\big]=\sup_nE_x^{(n)}\big[e^{-A_t^{\kappa_n/2}}\boldsymbol{1}_{\{\tau^{(n)}_{Z_n}>t\}}\big]\leq \sup_nE_x^{(n)}\big[e^{-A_t^{\kappa_n/2}}\big]\;.
  \end{align*}
  Thus assumption A3) above immediately yields
  \begin{align*}
    \sup_{t\in[0,1]}\sup_{x\in\X_\infty}\mathbb E_x\big[e^{-A_t^{\kappa/2}}\big]<\infty\;,
  \end{align*}
  i.e.~$\kappa$ is moderate.
  \smallskip
  
  {\bf iv)} To establish the taming condition
  $\GE_1(\kappa,\infty)$ for $(\sfM,\mathcal E_\infty)$ we
  have to show that
  \begin{align*}
  \Gamma(P_tu)^{\frac12}\leq
  P^{\kappa/2}_t\Gamma(u)^{\frac12}\;,
  \end{align*}
  for all $t$ and all $u$ in a dense class of functions in
  $\cF$. Let $u$ be a {function ${\mathcal C}^1(\sfM)$} and
  denote by $u$ also its restriction to any $\X_k$. Let
  $(\gamma_s)_{s\in[0,1]}$ be a {Lipschitz curve} in $\sfM^0$ and
  note that by A2) its image is contained in all
  $\X_n$ for $n$ suficiently large. The taming condition $\GE_1(\kappa_n,\infty)$ for $(\X_n,\mathcal E^{(n)})$ yields that
    \begin{align*}
     |P^{(n)}_tu(\gamma_1)-P^{(n)}_tu(\gamma_0)|\leq \int_0^1P^{\kappa_n/2}_t|\nabla u|(\gamma_s)|\dot\gamma_s|\d s\;.
  \end{align*}
  Steps i) and ii) allow to pass to the limit $n\to\infty$ in the left
  and right hand side respectively and obtain the same estimate with
  $P_t$ and $P_t^{\kappa/2}$. Hence,
  $|\nabla Pu|\leq P^{\kappa/2}_t|\nabla u|$ in $\sfM^0$. To conclude
  it suffices to note that any function in $\cF$ can be approximated
  by restrictions of {${\mathcal C}^1$ functions} to $\sfM_0$. Indeed,
  by $\X_n$ being a quasi-open nest, $\bigcup_{k=1}^\infty \cF_{\X_k}$
  is dense in $\cF$. Further, any $u\in \cF_{\X_k}$ can be extended to
  a function in $W^{1,2}(\sfM)$ by regularity of the boundary of
  $\X_n$. If $u_n$ is a {${\mathcal C}^1$ approximation} of this
  extension in $W^{1,2}(\sfM^0)$, then obviously the restrictions of
  $u_n$ to $\sfM_0$ converge to $u$ in $\cF$.
\end{proof}

Next, we give an example of singular boundary behavior by considering
a Euclidean domain with cusp-like singularity of the boundary such
that its curvature is controlled in $L^p$.

Consider the domain $\Y\subset \R^3$ given by
\[\Y:=\big\{(x,y,z)\in\R^3~:~z>\phi\big(\sqrt{x^2+y^2}\big)\big\}\;,\]
where $\phi:[0,\infty)\to[0,\infty)$ is $C^2$ on $(0,\infty)$ with
$\phi(r)=r-r^{2-\alpha}$, $\alpha\in(0,1)$ for $r\in[0,1]$ and $\phi$ constant
for $r\geq 2$. Let us denote by $\cE_\Y$ the standard Dirichlet form on
$\Y$ with Neumann boundary conditions and let $\mm$ the Lebesgue
measure restricted to $\Y$ and $\sigma$ the $2$-dimensional Hausdorff
measure on $\partial \Y$.

Parametrizing the surface of revolution $\partial\Y$ as
$\big\{\big(r\cos\theta,r\sin\theta,\phi(r)\big):r\geq0, \theta\in
[0,2\pi]\big\}$, one readily computes the smallest eigenvalue of the
second fundamental form of $\partial Y$ to be
\[
  \ell(r,\phi) = \frac{\phi''(r)}{\lambda(r)^3}\;,
\]
for $r\leq 1$ and $\ell = 0$ for $r\geq 2$ and where
$\lambda(r)=\sqrt{1+|\phi'(r)|^2}$ is the length element of the
revolving curve. Note that $\ell(r,\phi)\sim -r^{-\alpha}$ for $r$
small.

\begin{Theorem}\label{thm:singBdry}
  The Dirichlet space
  $(\Y,\cE_\Y,\mm_Y)$ satisfies {\sf BE}$_1(\kappa,\infty)$ with the
  moderate distribution
  \[
  \kappa =  \ell\, \sigma\;,
\]
where $\ell$ is the smallest eigenvalue of the second fundamental form
of $\partial\Y$.
\end{Theorem}

\begin{proof}
 First observe that $\kappa$ is moderate. Indeed, one sees that
\[
  \int_{\partial\Y\cap\{r\leq1\}}|\ell|^p\d\sigma=\int_0^1\int_0^{2\pi}|\ell(r,\theta)|^p\lambda(r)r\d r\d\theta\;.
 \]
 The latter is comparable to $\int_0^1r^{1-\alpha p}\d r$. Hence,
 $\ell\in L^p(\partial\Y,\sigma)$ iff $\alpha p<2$. Further note that
 $\partial \Y$ is the graph of a Lipschitz function. Thus, choosing
 $1<p<2/\alpha$, Theorem \ref{thm:Kato-bdry} yields that
 $\kappa\in\cK_0(\R^3)$ and since $\Y$ is inner uniform also
 $\kappa\in\cK_0(\Y)$, see Lemma \ref{lem:kato-comp}, in particular
 $\kappa$ is moderate. The taming condition will follow from
 Theorem \ref{thm:stability}. We claim that a regular exhaustion
 is given by
 $\Y_n:=\big\{(x,y,y):z>\phi_n\big(\sqrt{x^2+y^2}\big)\big\}$ where
 $\phi_n$ is $C^2$ with $\phi_n(r)=\phi(r)$ for $r\geq 1/n$ and a
 degree 3 polynominal on $[0,1/n]$ with $\phi_n'(0)=0$ (i.e.~we round
 the cusp of $\partial \Y$ at scale $1/n$). One can check that
 $\ell_n$, the minimal eigenvalue of the second fundamental form of
 $\partial\Y_n$, is controlled by $-r^{-\alpha}$ uniformly in $n$ and
 $\|\ell_n\|_{L^p(\partial\Y_n)}$ is uniformly bounded in $n$ for a
 $p$ as above. Moreover, $\partial \Y_n$ are graphs of Lipschitz
 functions with constants bounded uniformly in $n$. The proof of
 Theorem \ref{thm:Kato-bdry} gives that
 $\kappa_n=\ell_n\sigma_{\partial\Y}$ is uniformly bounded in the Kato
 class $\mathcal K_0(\R^3)$. One checks that $\Y_n$ is inner uniform
 with constants independent of $n$ and thus, by Lemma
 \ref{lem:kato-comp}, $\kappa_n$ is also uniformly bounded in
 $\mathcal K_0(Y_n)$,
 i.e.~$\lim_{t\to0}\sup_n\sup_x\mathbb E^{(n)}[A^{\kappa_n}_t]=0$,
 which entails that $\kappa_n$ is uniformly $1$- and $2$-moderate,
 i.e.~A3) holds. A2) is obvious. Finally, since $\Y$ is inner uniform,
 its Neumann heat kernel is comparable to the Euclidean one, which
 allows to check A1).
\end{proof}

\subsection{A tamed domain with boundary that is not semiconvex}
 
We will construct here another tamed space with boundary that has no
lower bound on the second fundamental form by adding to a Euclidean
halfspace a sequence of smaller and smaller bumps.

Let $\X_0:=\big\{(x_1,x_2,x_3)\in\R^3:x_3> 0\big\}$ be a halfspace
in $\R^3$. For $r,h>0$ consider the function $f_{r,h}:\R\to\R$ given by
\begin{align*}
  f_{r,h}(t)=h\cdot \big(-1-\cos(\pi t/r)\big)\boldsymbol{1}_{[-r,r]}(t)\;,
\end{align*}
and define the set
\begin{align*}
  O_{r,h}:=\Big\{x\in\R^3:  x_1^2+x_2^2\leq r^2,~ 0\geq x_3> f_{r,h}\Big(\sqrt{x_1^2+x_2^2}\Big)\Big\}\;,
\end{align*}
which will serve as the basic bump.
Let $\mu_{r,h}=\ell_{r,h}\sigma_{\partial O_{r,h}}$ denote the
curvature measure of the ``lower'' boundary of $O_{r,h}$, i.e.~the
surface measure of that part of the boundary weighted by the curvature
of the boundary. For $\eta\in \R^2$, let $\mu_{r,h}^\eta$ denote its
translation by the vector $(\eta,0)\in\R^3$.

Let sequences $(r_i)_i$ in $(0,1/2)$ and $(\xi_i)_i$ in $\R^{2}$ be
given such that the sets $A_i$ for $i\in\N$ are disjoint where
$A_i:=\cup_{z\in \Z^{n-1}}B_{r_i}(z+\xi_i)\subset \R^2$. Choose
$h_i\in\R_+$ with $h_i\leq r_i$ such
that \begin{align*}
\sum_i \big\|G_1^0\mu_i\big\|_{L^\infty}<\infty\;,\qquad\mu_i:=\sum_{z\in\Z^2}\mu^{z+\xi_i}_{r_i,h_i}\;,
\end{align*}
where $G^0_1\mu$ denotes the $1$-potential of $\mu$ in the half-space
$\X_0$. Consider then the set
\begin{align*}
  \X_\infty:=\X_0\cup\bigcup_{i\in\N}\bigcup_{z\in\Z^2}
  \big(O_{r_i,h_i}+z+\xi_i\big)\;,
\end{align*}
whose boundary is obviously is not semiconvex.

 \begin{Theorem} 
   The measure $\mu:=\sum_i\mu_i$ belongs to the Kato class
   $\KK_0(\X_0)$ and $\KK_0(\X_\infty)$. The Neumann heat semigroup on
   the set $\X_\infty$ satisfies $\BE_1(-\mu,\infty)$.
 \end{Theorem}
 
  \begin{proof} {\bf (i):} For each $i$, the quantity
 $$ \Big\|G_1^0\sum_{z\in\Z^2}\mu^{z+\xi_i}_{r_i,h} \Big\|_{L^\infty}$$
 depends continuously on $h$. Thus there exist $h_i>0$ such that
 $\sum_i \big\|G_1^0\mu_i\big\|_{L^\infty}<\infty$ as requested.
 Hence, $G_1^0\mu$ is bounded.  Moreover, $G_1^0\mu$ is uniformly
 continuous since it is the uniform limit of uniformly continuous
 functions. An easy argument (based on reflection symmetry) allows to
 carry over the criterion of Lemma \ref{unif-crit} to the heat
 semigroup on the half space. Finally, we note that $\X_\infty$ is an
 inner uniform domain in the sense of \cite{GyrSal}. Hence,
 Lemma \ref{lem:kato-comp} shows that $\mu$ belongs also to
 $\KK_0(\X_\infty)$. \smallskip

 {\bf (2):} Consider the sets 
 \begin{align*}
    \X_n:=\X_0\cup\bigcup_{i=1}^n\bigcup_{z\in\Z^2}
  \big(O_{r_i,h_i}+z+\xi_i\big)\cap B_n\;,
 \end{align*}
 where $B_n$ is essentially the ball of radius $n$ in $\R^3$ suitably
 modified near the $\{x_3=0\}$ so that $\X_n$ has smooth
 boundary. Note that $\X_n$ contains only the ``bumps'' at above scale
 $r_n$ of $\X_\infty$. Thus, by Theorem \ref{thm:smooth-domain} $X_n$
 satisfies $\BE_1(-\mu_n,\infty)$ with
 \begin{align*}
   \mu_n:=\sum_{i=1}^{n}\mu_i +\lambda_n\;,
 \end{align*}
 where $\lambda_n$ is essentially $1/n\cdot\sigma_{S_n^+}$ with
 $S_n^+$ the upper hemisphere of radius $n$. The relative boundary
 $Z_n$ of $\X_n$ in $\X_\infty$ consists of a countably many discs of
 sizes $r_i$ with $i>n$ together with $S_n^+$. We want to apply the
 stability result Theorem \ref{thm:stability} (with
 $\sfM^0=\X_\infty$ and $k=0$). Thus we need to check that the sets
 $\X_n$ provide a regular exhaustion. Note that the capacity of $Z_n$
 in $\X_0$ vanishes as $n\to\infty$ if $r_i$ are chosen appropriately
 and hence does its capacity in $\X_\infty$, i.e.~condition A1) is
 satisfied. Condition A2) is readily checked. One checks that the
 Green potential $\mathsf G_0\mu_n$ is bounded and uniformly
 continuous uniformly in $n$ (note that the contribution of
 $\lambda_n$ is negligible as $n\to\infty$. By the argument of Lemma
 \ref{unif-crit} and taking into account the inner uniformity of
 $\X_\infty$, we infer that
 $\lim_{t\to0}\sup_{x}\sup_n\mathbb E_x[A^{\mu_n}_t]=0$ yielding that
 the $\mu_n$ are uniformly moderate and $2$-moderate, i.e.~condition
 A3) holds. Thus we conclude that $X_\infty$ satisfies
 $BE_1(-\mu,\infty)$.
 \end{proof}

 \subsection{A Tamed Manifold with Highly Irregular Boundary}
 Our next example will provide a domain
 $X_\infty\subset X_0=\R\times \R_+$ such that the curvature measure
 $\mu$ of its boundary $\partial X$ is a moderate distribution which
 is not in the Kato class. Even, more $|\mu|$ is not a Radon measure.
 The proof of the former property will be based on the following
 useful criterion.
  
 \begin{Lemma}\label{osc-est} A distribution
   $\kappa\in \cF^{-1}_{\rm qloc}(X)$ is moderate provided there exists a
   function $\psi$ with $\sfL\psi=\kappa$ and
   $\|\psi\|_{L^\infty}<1$ and such that $\mu:=\psi\,\kappa$ defines a
   signed measure in the Kato class $\KK_0(X)$.
  \end{Lemma}
  
  \begin{proof} Choose $p>1$ such that $c:=\inf_x[1+p\psi](x)>0$. Put
    $u:=1+p\psi$ and $\nu:=\frac{\sfL
      u}u=\frac{p}{1+p\psi}\kappa$. Then $(-\sfL+\nu)u=0$ which
    implies
  $$c\,P^\nu_t1\le P^\nu_tu=u\le C\;,$$
  and thus $q-\sup_x P^\nu_t1(x)\le C/c<\infty$ for all $t>0$. Hence,
  $\nu$ is moderate.  In other words, the distribution
  $\kappa_1:=\frac1p \nu$ is $p$-moderate. On the other hand, the
  distribution $\kappa_2:=\kappa-\kappa_1=\frac{p\psi}{1+p\psi}\kappa$
  is a signed measure in the Kato class.  Therefore, in particular it
  is $p'$-moderate for $p'\in(1,\infty)$ being dual to $p$.  Thus
  according to Remark \ref{rem:moderate} (iii), $\kappa=\kappa_1+\kappa_2$ is moderate.
  \end{proof}
  
  The domain $X_\infty$ for the example mentioned above will be
  constructed as the limit of a sequence of domains
  $X_n\subset X_0=\R\times \R_+$, $n\in\N$, The building blocks of our
  construction of $X_n$ will be defined in terms of the functions
  $H_\ell:\R\to\R_+$ for $\ell\in\N$ given by
 $$H_\ell(s)=\frac1{(2\ell+1)^2\,\pi^2}\big( 1+\cos((2\ell+1)\pi\, s)\big)\cdot 1_{[-1,1]}(s).$$
 Define a measure $\tilde\mu_\ell$ on $\R$ in terms of the curvature
 $H_\ell''/(1+{H_\ell'}^2)^{3/2}$ and the arclength
 $(1+{H_\ell'}^2)^{1/2}$ of the curve $\big(s,H_\ell(s)\big)_{s\in\R}$
 by
 $$\tilde\mu_\ell(ds)=\frac{\cos((2\ell+1)\pi\, s)}{1+\big(\frac\pi{2\ell+1}\big)^2\sin^2( (2\ell+1)\pi\, s )}\,ds.$$
 Put $Y_\ell:=\{(x,y)\in\R^2: y\ge H_\ell(x)\}$ and define a measure
 $\mu_\ell$ on $\partial Y_\ell$ as the push forward of the measure
 $\tilde\mu_\ell$ under the map $s\mapsto (s,H_\ell(s))$.  Let
 $\Psi_\ell:=G_0\mu_\ell$ be the logarithmic potential of $\mu_\ell$
 in $\R^2$, that is,
 $$\Psi_\ell\big((x,y)\big):=-\int_{-1}^1 \log\Big(\Big(x-s\Big)^2+\Big(y-\frac{1+\cos((2\ell+1)\pi\, s)}{(2\ell+1)^2\pi^2}\Big)^2
 \Big)\,\frac{\cos((2\ell+1)\pi\, s)}{1+\big(\frac\pi{2\ell+1}\big)^2\sin^2( (2\ell+1)\pi\, s )}
 \,ds.$$
 
 \begin{Lemma}\label{Psi-est} Uniformly in $x,y$ and $\ell\ge2$:
 $$\big|\Psi_\ell\big((x,y)\big)\big|\le C\cdot \frac {\log\ell}{\ell}.$$
 \end{Lemma}
 
 \begin{proof} 
   Since $\Psi_\ell$ vanishes at $\infty$ and since it is harmonic on
   $\R^2\setminus Z$, by the maximum principle $\Psi_\ell$ has to
   attain its extrema on $Z:=\partial Y_\ell\cap ([-1,1]\times\R)$. By
   symmetry, we conclude that $\Psi_\ell$ attains its maximum at the
   points $(\pm1,0)$ and its minimum at the point
   $(0,H_\ell(0))$. Note that for fixed $\ell\ge2$, each of these
   values can be estimated in terms of $C_\ell\frac{\log\ell}\ell$ by
   explicit calculations. Thus in the sequel, we may assume that
   $\ell$ is sufficiently large. To estimate $\Psi_\ell$ uniformly in $\ell$ at the point $(-1,0)$,
   we decompose the interval $[-1,1]$ into $2\ell$ intervals
   $I_j=[-1+(j-\frac14)\frac2{2\ell+1},
   -1+(j+\frac34)\frac2{2\ell+1}]$ of length $\frac2{2\ell+1}$ as well
   as a left end interval $I_0=[-1,-1+\frac34\frac2{2\ell+1}$ of
   length $\frac34\cdot \frac2{2\ell+1}$ and a right end interval
   $I_{2\ell+1}$ of length $\frac14\cdot \frac2{2\ell+1}$. Then in the
   defining integral for $\Psi_\ell$, for each $j=1,\ldots, \ell$ the
   negative contributions of $s\in I_{2j-1}$ dominate in absolute
   value the positive contributions of $s\in I_{2j}$. Thus uniformly
   in $\ell\ge2$
 \begin{eqnarray*}\Psi_\ell\big((-1,0)\big)&\le&-\int_{I_0\cup I_{2\ell+1}}
 \log\Big(\big(1+s\big)^2+
 \Big(\frac{1+\cos((2\ell+1)\pi\, s)}{(2\ell+1)^2\pi^2}\Big)^2
  \Big)\,
  \,\frac{\cos((2\ell+1)\pi\, s)}{1+\big(\frac\pi{2\ell+1}\big)^2\sin^2( \, s )}
 \,ds\\
 &\le&2
 \int_{I_0\cup I_{2\ell+1}}\log\frac1{1+s}
 \,ds\le C\cdot \frac{\log\ell}\ell
 \end{eqnarray*}
 for suitable $C$. To estimate
 $$\Psi_\ell(0,H_\ell(0)) =-\int_{-1}^1
 \log\Big(s^2+\Big(\frac{1-\cos((2\ell+1)\pi\,
   s)}{(2\ell+1)^2\pi^2}\Big)^2 \Big)\,\frac{\cos((2\ell+1)\pi\,
   s)}{1+\big(\frac\pi{2\ell+1}\big)^2\sin^2( (2\ell+1)\pi\, s )}
 \,ds, $$ we argue similar. We now decompose the interval $[-1,1]$
 into $4\ell+1$ intervals
 $I_j=[\frac{2i-1}{2(2\ell+1)}, \frac{2i+1}{2(2\ell+1)}]$ for
 $j=-2\ell,\ldots, 2\ell$ of length $\frac1{2\ell+1}$ and two boundary
 intervals $I_{-(2\ell+1)}=[-1,-1+\frac1{2(2\ell+1)}]$ and
 $I_{(2\ell+1)}=[1-\frac1{2(2\ell+1)},1]$ of length
 $\frac1{2(2\ell+1)}$.  Then in the defining integral for
 $\Psi_\ell(0,H_\ell(0))$, for each $j=1,\ldots, \ell$ the positive
 contributions of $s\in I_{2j-1}$ dominate in absolute value the
 negative contributions of $s\in I_{2j}$. Similarly, the positive
 contributions of $s\in I_{-2j+1}$ dominate in absolute value the
 negative contributions of $s\in I_{-2j}$.
 
 Note that
 $s^2+\Big(\frac{1-\cos((2\ell+1)\pi\, s)}{(2\ell+1)^2\pi^2}\Big)^2\le1$
 for sufficiently large $\ell$ and $|s|\le 1-1/\ell^2$.  Thus
  $$ \log\Big(s^2+\Big(\frac{1-\cos((2\ell+1)\pi\, s)}{(2\ell+1)^2\pi^2}\Big)^2
  \Big)\le0$$ for all $s\in [-1+1/\ell^2,1-1/\ell^2]$.  Therefore, the
  contributions of
  $s\in I_{\pm(2\ell+1)}\cap [-1+1/\ell^2,1-1/\ell^2]$ are positive.
 
  Thus with $J:=I_0\cup [-1,-1+1/\ell^2]\cup [1-1/\ell^2,1]$ and large
  enough $\ell$,
 \begin{eqnarray*}\
 \Psi_\ell(0,H_\ell(0))
& \ge&-\int_{J
 } \log\Big(s^2+\Big(\frac{1-\cos((2\ell+1)\pi\, s)}{(2\ell+1)^2\pi^2}\Big)^2
 \Big)\,\frac{\cos((2\ell+1)\pi\, s)}{1+\big(\frac\pi{2\ell+1}\big)^2\sin^2( (2\ell+1)\pi\, s )}
 \,ds\\
 & \ge&
 -\frac C{\ell^2}-\int_{I_0
 } \log\Big(s^2+\Big(\frac{1-\cos((2\ell+1)\pi\, s)}{(2\ell+1)^2\pi^2}\Big)^2
 \Big) \,ds\\
 & \ge&
 -\frac C{\ell^2}-C'\frac{\log\ell}\ell\;.
 \end{eqnarray*}
 This proves the claim.
 \end{proof}
 
 Let us now define the functions $\psi_n$ for $n\in\N$ by symmetrized,
 rescaled and translated versions of the $\Psi_\ell$:
 $$\psi_n\big((x,y)\big)= \Psi_{\ell_n}\Big(\frac{x-R_n}{r_n},\frac y{r_n}\Big)-\Psi_{\ell_n}\Big(\frac{x+R_n}{r_n},\frac y{r_n}\Big)$$
 with $\ell_n:=4^{n}, r_n=2^{-n}$, $R_n=2^{2-n}$, and we put
 $\kappa_n^*:=\Delta\psi_n$. Similarly, we define functions $h^*_n$ and
 $h_n$ by
 $$h^*_n(s)= H_{\ell_n}\Big(\frac{s-R_n}{r_n}\Big)-H_{\ell_n}\Big(\frac{s+R_n}{r_n}\Big), \qquad h_n:=\sum_{i=1}^n h_i^* $$
 and we put $X^*_n:=\{(x,y)\in\R^2: y\ge h^*_n(x)\}$ as well as
 $X_n:=\{(x,y)\in\R^2: y\ge h_n(x)\}$.  Then for each $n$ the measure
 $\kappa_n^*$ is supported by $\partial X^*_n$. Indeed, it is the
 curvature measure of $X^*_n$. Similarly, the measure
 $\kappa^n=\sum_{i=1}^n\kappa_i^*$ is the curvature measure of $X_n$ and
 supported by $\partial X_n$.  Moreover, define
 \begin{align}\label{eq:psi-kappa-h}
 \psi=\sum_{n=1}^\infty\psi_n,\quad
 \kappa=\sum_{n=1}^\infty\kappa_n^*,\quad h=\sum_{n=1}^\infty h_n^*\;,
 \end{align}
 and $\X_\infty=\{(x,y)\in\R^2: y\ge h(x)\}$.  Then the boundary
 curvature of $\X_\infty$ is given by $\kappa=\Delta\psi$.  Note that
 $\X_\infty$ is not a monotone limit of the $\X_n, n\in\N$. Instead,
 $$ \X_\infty \cap \big(\R_+\times\R_-\big)=\Big(\bigcup_n  \X_n\Big) \cap \big(\R_+\times\R_-\big),\qquad
 \X_\infty \cap \big(\R_+\times\R_+\big)=\Big(\bigcap_n  \X_n\Big) \cap \big(\R_+\times\R_+\big).$$

 \begin{Lemma}\label{psi-lip} There exists $C_1$ such that 
 $$|\psi(x,y)|\le C_1\,|x|$$
 uniformly in $(x,y)\in\R^2$ and $n\in\N$.
 \end{Lemma}
 
 \begin{proof} For each $i$, the function $(x,y)\mapsto \psi_i(x,y)$
   is bounded by $Ci4^{-i}$, it vanishes at $x=0$, and it is harmonic
   for $|x|\le R_i-r_i=3\cdot 2^{-i}$. Thus for
   $|x|\le R_i-r_i=3\cdot 2^{-i}$,
 $$|\psi_i(x,y))|\le  Ci4^{-i}\cdot |x|.$$
 Given $x$, choose $n$ such that $2^{-n}\le |x|<2^{1-n}$. Then
 $$\sum_{i=1}^n|\psi_i(x,y))|\le  C'\cdot |x|$$
 uniformly in $(x,y)\in\R^2$ and $n\in\N$. On the other hand, by
 boundedness of the $\psi_i$,
 $$\sum_{i=n+1}^\infty|\psi_i(x,y))|\le C\sum_{i=n+1}^\infty i4^{-i} \le C''2^{-n}\le C''\cdot |x|.
 $$
 This proves the claim.
 \end{proof}
 
 \begin{Theorem} The distribution $\kappa$ defined in
   \eqref{eq:psi-kappa-h} is moderate. It is not given by a signed
   Radon measure (and in particular, not by a signed measure in the
   Kato class). Moreover, the Neumann heat flow in $\X_\infty$
   satisfies $\BE_1(\kappa,\infty)$.
 \end{Theorem}
 
 \begin{proof}
   {\bf (i)} According to the previous Lemma \ref{Psi-est},
   $\|\psi_n\|_{L^\infty}\le Cn4^{-n}$ for all $n$ and thus for fixed $p>1$,
 $$\Big\|\sum_{n=j}^\infty\psi_n\Big\|_{L^\infty}\le C\sum_{n=j}^\infty n4^{-n}< 1/p$$
 for sufficiently large $j$. Fix such a $j$ and put
 $\psi'=\sum_{n=j}^\infty\psi_n$. Decompose $\kappa$ into
 $\kappa':=\sum_{n=j}^\infty\kappa_n^*$ and
 $\kappa'':=\sum_{n=1}^{j-1}\kappa_n^*$. Obviously, $\kappa''$ is a
 signed measure in the Kato class. By Lemma \ref{osc-est}, the
 distribution $\kappa'$ is $p$-moderate, provided $p\psi'\kappa'$ is
 Kato. In order to prove the latter, note that $\Delta\Psi_\ell$ is a
 signed measure and that
 $$\big\| G_1|\Delta\Psi_\ell|\big\| _{L^\infty}\le C'$$
 uniformly in $\ell\in\N$. Indeed, $|\Delta\Psi_\ell|$ converges
 weakly to the uniform distribution on the interval
 $[-1,1]\times\{0\}$ as $\ell\to\infty$. By rescaling, we obtain
 $$\big\| G_1|\kappa_n^*|\big\| _{L^\infty}\le  C'\, \log \frac1{r_n}=C''\, n$$
 uniformly in $n \in\N$.  Moreover, according to Lemma \ref{psi-lip},
 on the support of $|\kappa_n^*|$, that is, on
 $[R_n-r_n, R_n+r_n]\times\{0\}$, the function $\psi'$ is bounded by
 $C_0(R_n+r_n)\le C_0\, 2^{3-n}$. Thus
 $$\big\| G_1\big(|\psi'\kappa_n|\big)\big\|_{L^\infty}\le C_1n 2^{3-n}$$
 for each $n$ and therefore the 1-potential $G_1\nu$ of the measure
 $\nu:=\sum_{n\ge j}p|\psi'\kappa_n^*|$ is bounded and uniformly
 continuous on $\R^2$. According to Lemma \ref{unif-crit}, this proves
 that $\nu$ lies in the Kato class and thus so does
 $p\psi'\kappa'$. Hence, $\kappa'$ is $p$-moderate and thus $\kappa$ is moderate by Remark \ref{rem:moderate} (iii).

 On the other hand, of course, $\kappa$ will not be a signed Radon
 measure since for each $\delta>0$ and sufficiently large
 $N=N(\delta)$,
$$|\kappa| \big(B_\delta(0)\big)\ge\sum_{n\ge N}|\kappa_n^*|(\R^2)\ge\sum_{n\ge N}1 =\infty.$$
where $B_\delta(0)$ denotes the $\delta$-neighborhood of the origin in $\R^2$.

{\bf (ii)} Note that all "wiggles" on the positive $x$-axis lie below
the curve $y=c|x|^4$ for a suitable constant $c$. Thus we can approximate $\X_\infty$
by an increasing sequence of smooth domains $\tilde \X_n$ as follows.  Choose an
decreasing sequence of functions $\phi_n\in{\mathcal C}^2_c(\R)$ with
support $[0,3\cdot 2^{-n}]$ and with $\phi_n(x)= c|x|^4$
for $x\in[0,5\cdot 2^{-n-1}]$. Then define
 $$\tilde \X_n:=\big\{(x,y)\in\R^2: y\ge h_n(x)+\phi_n(x)\big\}\cap B_n\;,$$
 where $B_n$ is essentially the ball of radius $n$ in $\R^2$. Note
 that $\tilde \X_n$ contains only the first $n$ pairs of
 ``wiggles''. The Neumann heat semigroup on $\tilde \X_n$ satisfies
 $\BE_1(\kappa_n+\tilde\kappa_n,\infty)$, where $\tilde\kappa_n$ is
 the curvature measure associated with the relative boundary $Z_n$ of
 $\tilde \X_n$ in $\X_\infty$. To conclude, it suffices to check that the
 sets $\tilde X_n$ provide a regular exhaustion and apply Theorem
 \ref{thm:stability} (with $\sfM =\X_\infty$ and $k=0$). As in the
 previous example, one checks that the capacity of $Z_n$ vanishes as
 $n\to\infty$ and thus A1) holds. A2) is obvious by construction. To
 check A3) (with $\kappa_n+\tilde\kappa_n$ taking the role of
 $\kappa_n$ there), note first that $\phi_n$ can be chosen such that
 the density of $\tilde \kappa_n$ w.r.t. the Hausdorff measure goes to $0$ as $n\to\infty$ and hence
 $|A^{\tilde \kappa_n/2}_t|\leq ct$ for arbitrary small $c>0$ for $n$
 sufficiently large. Arguing as in step (i) to control
 $\kappa_n=\kappa'_n+\kappa''$ with $\kappa'_n=\sum_{i=j}^n\kappa_i$
 instead of $\kappa$, we see that the resulting bound on
 $\sup_{t\in[0,1]}\sup_x\mathbb E_x\big[e^{-A^{\kappa_n/2}_t}\big]$ is
 independent of $n$. Together with the bound on
 $A^{\tilde\kappa_n/2}_t$, this shows the uniform moderateness of
 $\kappa$. Essentially the same argument for $2\kappa_n$ yields the
 uniform $2$-moderateness.
\end{proof}

\section{Functional Inequalities for Tamed Spaces}\label{sec:funct-ineq}
In this section we derive local (reverse) Poincar\'e and logarithmic
Sobolev inequalities for the heat flow on tamed spaces. We use the notation
\[
\cF_b := \cF \cap L^\infty(\X, \mm).
\] 

\begin{Theorem}[Local (reverse) Poincar\'e inequality]\label{thm:Poin}
Let $(\X,  \mathcal{E}, \mm)$ be a Dirichlet space with a $2$-moderate distribution $\kappa \in \mathcal{F}^{-1}_{\rm qloc}$ satisfying $\GE_1(\kappa, \infty)$. Then for any $f \in \cF$ and any $t > 0$ we have $\mm$-a.e.~on $\X$:
\begin{equation}\label{eq:locPoi}
\underline C_t^\kappa \cdot \Gamma (P_{t} f)\leq \frac1{2t}\Big[P_t(f^2) - (P_t f)^2\Big] \le \overline C_t^\kappa  \cdot P_t\Gamma (f)\;,
\end{equation}
with $\overline C_t^\kappa:=\frac1t\int_0^t C_{s}^{\kappa}\,\d s$ and $\underline C_t^\kappa:=\frac1t\int_0^t \big(C_{s}^{\kappa}\big)^{-1}\,\d s$
where $C_{t}^{\kappa}$ is the time-depending constant defined in
\eqref{eq:Holder-const}. Note that $\big(\underline C_t^\kappa\big)^{-1}\le \overline C_t^\kappa$ for all $t>0$ and $\limsup_{t\to0} \overline C_t^\kappa<\infty$.

The first inequality in \eqref{eq:locPoi} is valid for any
$f\in L^2(\X,\mm)$.
\end{Theorem}

\begin{proof}
Let $f,g \in \cF_b, g \ge 0$ be given. For any $t > 0$, we set $f_t := P_t f$, $g_t := P_t g$, and
\[
{\Theta}(s) := \int_\X (f_{t-s})^2 g_s \, \d \mm \quad \text{for} \quad  s \in [0, t].
\]
A direct computation gives
\[
\begin{split}
\dfrac{\d}{\d s} \Theta(s) & =  \int_\X \big( - 2 g_s f_{t-s} \partial_s f_{t-s} + f_{t-s}^2 \partial_s g_s \big) \, \d \mm\\
&  =  \int_\X \Big( 2 \Gamma(g_s f_{t-s}, f_{t-s}) - \Gamma(f^2_{t-s}, g_s) \Big) \,\d\mm =   2 \int_\X g_s  \Gamma(f_{t-s}) \,\d\mm,
\end{split}
\]
which in turn provides
\[
\int_\X P_t g f^2 \, \d \mm - \int_\X g (P_t f)^2 \, \d \mm = \int_0^t \bigg( \dfrac{\d}{\d s} \Theta(s) \bigg) \, \d s = 2\int_0^t \int_\X g_s  \Gamma(f_{t-s}) \,\d\mm \, \d s.
\]
At this point, applying first \eqref{eq:L1gradient} and then \eqref{eq:Holder1}, we obtain
\[
\int_\X g P_s\big( \sqrt{\Gamma(P_{t-s} f)} \big)^2 \, \d \mm \overset{\eqref{eq:L1gradient}}{\le} \int_\X g P_s\big( P_{t-s}^{\kappa/2} \sqrt{\Gamma(f)}\big)^2 \, \d \mm \overset{\eqref{eq:Holder1}}{\le} C^{\kappa}_{t-s} \int_\X g P_s\big(  P_{t-s} (\Gamma(f)) \big) \, \d \mm\;,
\]
which leads to
\begin{equation}\label{eq:Poi1bd}
\int_\X  g \big( P_t(f^2)  - (P_t f)^2 \big)\, \d \mm  \le 2 \bigg(\int_0^t C^{\kappa}_{t-s} \, \d s \bigg) \int_\X g P_{t} (\Gamma(f))   \, \d \mm\;.
\end{equation}
On the other hand, applying first \eqref{eq:Holder1} and then \eqref{eq:L1gradient}, we get
\[
\int_\X g P_s\Big( \sqrt{\Gamma(P_{t-s} f)}\, \Big)^2 \, \d \mm \overset{\eqref{eq:Holder1}}{\ge}  \dfrac{1}{C^{\kappa}_{s}} \int_\X g  \Big(  P^{\kappa/2}_s \sqrt{\Gamma(P_{t-s}f)}  \Big)^2 \, \d \mm \overset{\eqref{eq:L1gradient}}{\ge}  \dfrac{1}{C^{\kappa}_{s}} \int_\X g    \Gamma(P_s(P_{t-s}f))   \, \d \mm\;,
\]
which in particular means
\begin{equation}\label{eq:Poi2bd}
\int_\X  g \big( P_t(f^2)  - (P_t f)^2 \big)\, \d \mm  \ge 2 \bigg(\int_0^t \dfrac{1}{C^{\kappa}_{s}} \, \d s \bigg) \int_\X g  \Gamma (P_{t} f)   \, \d \mm\;.
\end{equation}
By a standard approximation argument, we extend the validity of
\eqref{eq:Poi1bd}, \eqref{eq:Poi2bd} to $f\in L^2(\X,\mm)$
and $f\in\cF$ respectively and we conclude by the arbitrariness of $g$.
\end{proof}

\begin{Theorem}[Local (reverse) log-Sobolev inequality]\label{thm:logSob}
Let $(\X,  \mathcal{E}, \mm)$ be a Dirichlet space with a $2$-moderate distribution $\kappa \in \mathcal{F}^{-1}_{\rm qloc}$ satisfying $\GE_1(\kappa, \infty)$. Then for any $t > 0$ and for any $f \ge 0$ with the property that $\sqrt{f} \in \mathcal{F}$ and $f \log(f) \in L^1(\X,\mm)$, it holds $\mm$-a.e.~on $\X$:
\begin{equation}\label{eq:loclogSob}
 \int_0^t \dfrac{\Gamma(P_t f)}{P^{\kappa/2}_sP_{t-s} f}\,\d s \leq P_t(f \log f) - P_t f \log (P_t f) \le  \int_0^t P_sP^\kappa_{t-s}\bigg(\dfrac{\,\Gamma(f)\,}{f}\bigg)\d s\;,
\end{equation}
The first inequality holds more generally for all non-negative
$f \in L^1(\X,\mm)$ with $f \log(f) \in L^1(\X,\mm)$.
\end{Theorem}

\begin{proof} Let $\epsilon > 0$ be fixed and  $\psi_\epsilon \colon [0, \infty) \to \R$ be defined by $\psi_\epsilon (z) := (z + \epsilon) \log(z+\epsilon) - \epsilon \log(\epsilon)$. For $t > 0$, $g \in L^1\cap L^\infty(\X,\mm)$, $g \ge 0$, and $f \in L^\infty(\X,\mm)$ such that $f \ge 0$, $\sqrt{f} \in \mathcal{F}$, and $f \log (f) \in L^1(\X,\mm)$, we define
\[
\Psi_\epsilon(s) := \int_\X g_s  \, \psi_\epsilon (f_{t-s}) \, \d \mm, \quad \text{for any} \,\,  0 < s < t,
\]
where $g_s := P_s g$ and $f_{t-s} := P_{t-s} f$. Notice that the continuity of $s \mapsto g_s$ and $s \mapsto f_{t-s}$ in $L^2(\X,\mm)$ ensures that the map $s \mapsto \Psi(s)$ is continuous. Hence, a direct computation gives:
\[
\begin{split}
\dfrac{\d}{\d s} \Psi_\epsilon (s) &= \dfrac{\d}{\d s} \int_\X g_s  \, \psi_\epsilon (f_{t-s}) \, \d \mm = \int_\X \sfL g_s \psi_\epsilon (f_{t-s}) \, \d \mm - \int_\X g_s \psi_\epsilon'(f_{t-s}) \sfL f_{t-s} \, \d \mm\\
&= - \int_\X \Gamma(g_s, f_{t-s}) \psi'_\epsilon (f_{t-s}) \, \d \mm + \int_\X \Gamma (g_s \psi'_{\epsilon}(f_{t-s}), f_{t-s}) \, \d \mm\\ &= \int_\X g_s \psi''_{\epsilon}(f_{t-s}) \Gamma (f_{t-s}) \, \d \mm = \int_\X  g P_s \bigg( \dfrac{\Gamma(f_{t-s})}{f_{t-s} + \epsilon}  \bigg) \, \d \mm.
\end{split}
\]
By Jensen's inquality we have
  \begin{align*}
    \dfrac{\big(P^{\kappa/2}_{r}\sqrt{\Gamma(f)}\big)^2}{P_{r}f + \epsilon}(x)
    = \dfrac{\Big(\mathbb E_x\big[e^{-A^{\kappa/2}_r}\sqrt{\Gamma(f)(B_r)}\big]\Big)^2}{\mathbb E_x\big[f(B_{r})+\epsilon\big]}
    \leq \mathbb E_x\bigg[\frac{e^{-2A^{\kappa/2}_r}\Gamma(f)(B_r)}{f(B_r)+\epsilon}\bigg]
    = P^{\kappa}_r\bigg( \frac{\Gamma(f)}{f+\epsilon}\bigg)(x)\;.
  \end{align*}
  This, together with the gradient estimate \eqref{eq:L1gradient}, ensures that
  \[
    P_s \bigg( \dfrac{\Gamma(f_{t-s})}{f_{t-s} + \epsilon}  \bigg) \le P_s \bigg( \dfrac{\big(P^{\kappa/2}_{t-s}\sqrt{\Gamma(f)}\big)^2}{f_{t-s} + \epsilon} \bigg) \le P_sP^{\kappa}_{t-s}\bigg( \frac{\Gamma(f)}{f+\epsilon}\bigg)\;.
  \]
Integrating over $[0,t]$, we get
\begin{equation}\label{eq:logSobeps}
\int_\X g_t \psi_\epsilon (f) \, \d \mm - \int_\X g \, \psi_\epsilon (f_t) \, \d \mm = \int_0^t \dfrac{\d}{\d s} \Psi_\epsilon (s) \, \d s \le \int_0^t \int_\X g P_sP^{\kappa}_{t-s} \bigg( \dfrac{\Gamma(f)}{f + \epsilon}  \bigg) \, \d \mm \d s\;.
\end{equation}

At this point we notice that $f \log (f) \in L^1(\X,\mm)$ implies $f_t \log (f_t) \in L^1(\X,\mm)$ for any $t > 0$,
and so we can pass to the limit as $\epsilon \to 0$ in the left-hand side
of \eqref{eq:logSobeps}. By monotone convergence, we can also pass to
the limit in the right-hand side, obtaining
\begin{equation*}
\int_\X g P_t (f \log (f)) \, \d \mm - \int_\X g(P_t f \log(P_t f)) \, \d \mm \le \int_0^t \int_\X g P_sP^\kappa_{t-s} \bigg( \dfrac{\Gamma(f)}{f}  \bigg) \, \d \mm\d s\;.
\end{equation*}
In order to extend the result to general $f \ge 0$ with the property that $\sqrt{f} \in \mathcal{F}$ and $f \log(f) \in L^1(\X,\mm)$, we approximate it by taking $f_n := f \wedge n$ and then we let $n \to \infty$, using  the fact that $f^n \to f$ and $P_t(f^n) \to P_t f$ in $L^1$, while $\Gamma(f^n) = \Gamma(f) \mathds{1}_{\{ f < n \}}$ $\mm$-a.e. Hence, the arbitrariness of $g$ allows to conclude the second inequality in \eqref{eq:loclogSob}.

The first bound in \eqref{eq:loclogSob} can be obtained noting that by Jensen's inequality
and the gradient estimate
  \begin{align*}
    P_s \bigg( \dfrac{\Gamma(f_{t-s})}{f_{t-s}} \bigg)(x)
    &= \mathbb E_x\Bigg[\dfrac{e^{-A^{\kappa/2}_s}\big(\sqrt{\Gamma(f_{t-s})}\big)^2(B_s)}{e^{-A^{\kappa/2}_s}f_{t-s}(B_s)}\Bigg]
      \geq \dfrac{\Big(\mathbb E_x\big[e^{-A^{\kappa/2}_s}\sqrt{\Gamma(f_{t-s})}(B_s)\big]\Big)^2}{\mathbb E_x\big[e^{-A^{\kappa/2}_s}f_{t-s}(B_s)\big]}\\
    &=\dfrac{\big(P^{\kappa/2}_s\sqrt{\Gamma(P_{t-s}f)}\big)^2}{P_s^{\kappa/2}P_{t-s}f}(x)
      \geq \dfrac{\Gamma(P_{t}f)}{P_s^{\kappa/2}P_{t-s}f}(x)\;.
  \end{align*}
Thus, arguing as above, we get the desired estimate.
\end{proof}

\section{Self-Improvement of the Taming Condition}
\label{sec:self-improve}

In this section we discuss the self-improvement of the taming
condition, namely whether the $L^2$-Bochner inequality
$\BE_2(\kappa,N)$ already implies the stronger $L^1$ version
$\BE_1(\kappa,N)$. We will give an affirmative answer in a slightly restricted setting assuming throughout this section, that $\kappa$ is a signed measure in the Kato class $\KK_0(\X)$ (or extended Kato class $\KK_{1-}(\X)$).

We adapt here ideas developed in \cite{Savare13} in the case of constant
lower Ricci bounds.

\subsection{Measure-Valued Taming Operator and Bochner Inequality}
\label{sec:measureL}
The first step is to extend the definition of the taming operator
$\sfL^\kappa$ and the iterated carr\'e du champ $\Gamma_2^\kappa$ to possibly taking values in the space of measures.
\medskip

\textbf{Measure-valued taming operator.}
Recall that under the above assumption, Proposition \ref{prop:Kato-moderate-mu} ensures that  $\kappa$ is a moderated distribution, while from Corollary \ref{cor:domDsm0} it holds
\[
 \fD(\mathcal E^\kappa) = \cF.
\]

We recall the following approximation procedure: given a non-negative kernel $\eta \in \mathcal{C}_c^\infty(0, \infty)$ with $\int_0^\infty \eta(r) \,\d r = 1$, for any $f \in L^2(\X, \mm)$ and $\epsilon > 0$ we set
\begin{equation}\label{def:Pfrak}
\mathfrak{P}^\kappa_\epsilon f := \dfrac{1}{\epsilon}\int_0^\infty P_r^\kappa f \eta(r/\epsilon) \, \d r = \int_0^\infty P^\kappa_{\epsilon s} f \eta(s) \, \d s.
\end{equation}
Notice that $\mathfrak{P}^\kappa_\epsilon$ is positivity preserving, and that $\mathfrak{P}^\kappa_\epsilon f \in  \fD (\sfL^\kappa)$ for any
$\epsilon > 0$. Moreover, for $f\in L^\infty(\X,\mm)$ we have $\sfL^\kappa \mathfrak{P}^\kappa_\epsilon f\in L^\infty(\X,\mm)$, since
\begin{align*}
\sfL^\kappa \mathfrak{P}_\epsilon f = \int _0^\infty\frac{1}{\epsilon}\eta'(r/\epsilon)P^\kappa_{r} f\d r\;.
\end{align*}
\begin{Lemma}\label{lem:linfunc} Let
    $l \in \cF^{-1}_{qloc}$ be a linear functional on $\cF$ such that
    $\langle l, v\rangle\geq 0$ for any non-negative
    $v\in\bigcup_n\mathcal F_{G_n}$, with $(G_n)_{n \in \N}$ a quasi-open nest on
    which $\kappa$ is defined. Then there exists a unique non-negative
    $\sigma$-finite regular Borel measure $ \mu$ on $\X$ such that
    $\mu$ does not charge $\mathcal E$-polar sets, the
    quasi-continuous representative of any $f\in\mathcal F_{G_n}$ is
    integrable w.r.t.~$\mu$,
    i.e.~$\bigcup_n\tilde{\mathcal F}_{G_n}\subset L^1(\X, \mu)$, and
    \begin{align*}
    \langle l,v\rangle = \int_\X\tilde v \,\d \mu,\qquad \forall v\in \bigcup_n \mathcal F_{G_n}\;.
    \end{align*}
  \end{Lemma}
  \begin{proof}
    By the Lax-Milgram theorem, for each $n$, there exists a unique
    $v_n\in\mathcal F_{G_n}$ such that
    \begin{align*}
    \langle l,v\rangle = \mathcal E_{G_n}(v,v_n) + \int_\X vv_n\d\mm, \qquad \forall v\in \mathcal F_{G_n}\;. 
    \end{align*}
    The function $v_n$ is $1$-excessive for $\mathcal E_{G_n}$ and in particular
    non-negative. By \cite[Lemma~3.4]{Kuwae98} the
    restricted forms $(\mathcal E_{G_n},\mathcal F_{G_n})$ are again
    quasi-regular Dirichlet forms. Then, by
    \cite[Proposition~2.1]{MaRockner92} there exists a unique
    $\sigma$-finite positive Borel measure $\mu_n$ on $G_n$ not
    charging $\mathcal E_{G_n}$-polar sets such that
    \begin{align*}
      \langle l,v\rangle = \mathcal E_{G_n}(v,v_n) + \int_\X vv_n \,\d\mm =\int_\X\tilde v \, \d \mu_n, \qquad \forall v\in \mathcal F_{G_n}\;. 
    \end{align*}
    Note that by uniqueness and the inclusion
    $\mathcal F_{G_n}\subset \mathcal F_{G_{n+1}}$, we have
    $\mu_{n+1}(A\cap G_{n+1}) = \mu_n(A\cap G_n)$ for $A\subset \X$.
    Since further $\mathcal E$-polar subsets of $G_n$ are
    $\mathcal E_{G_n}$-polar \cite[Lemma~3.5]{Kuwae98},
    it is readily checked that $\mu(A):=\nearrow\lim \mu_n(A\cap G_n)$
    yields the measure with the desired properties.
   \end{proof}
 \begin{Lemma}\label{lem:Lstar}
    Let $u \in L^1\cap L^\infty (\X, \mm)$ be non-negative, and let
    $g \in L^1 \cap L^2(\X, \mm)$ be such that for any non-negative
    $\varphi \in \fD(\sfL^\kappa) \cap L^\infty(\X, \mm)$ with
    $\sfL^\kappa\varphi \in L^\infty(\X, \mm)$ it holds
  \begin{equation}\label{eq:hypg}
  \int_\X u \sfL^\kappa \varphi \, \d \mm \ge - \int_\X g \varphi \, \d \mm\;.
  \end{equation}
Then
\begin{equation}\label{eq:boundEk}
u \in \fD(\cE^\kappa) = \cF\;, \quad \mathcal{E}^\kappa(u) \le \int_\X u \, g \, \d \mm\;. 
\end{equation}
Moreover, there exists a unique $\sigma$-finite regular Borel measure $\mu := \mu_+ - g\mm$, with $\mu_+ \ge 0$, such that every $\mathcal{E}$-polar set is $| \mu |$-negligible, the quasi-continuous representative of any function in $\cF$ is in $L^1(\X, |\mu|)$, and
\begin{equation}\label{eq:meaEk}
-\mathcal{E}^\kappa (u, \varphi) = \int_\X \tilde{\varphi} \, \d \mu, \qquad \forall \varphi \in \cF\;.
\end{equation}
\end{Lemma}
\begin{proof}
  First, let us consider $u_\epsilon := \mathfrak{P}^\kappa_\epsilon (u)=\int_0^\infty P^\kappa_tu \frac{1}{\epsilon}\eta(t/\epsilon)\d t$. The regularizing properties of $(\mathfrak{P}^\kappa_\epsilon)_{\epsilon > 0}$ ensure that $u_\epsilon \in \fD(\sfL^\kappa)$ with $\sfL^\kappa u_\epsilon \in L^1 \cap L^\infty (\X, \mm)$. Inequality \eqref{eq:hypg} ensures that for every non-negative $\varphi \in L^2 \cap L^\infty(\X, \mm)$ it holds
\begin{equation}\label{eq:ineqg}
\int_\X \sfL^\kappa u_\epsilon \varphi \, \d \mm = \int_\X u \sfL^\kappa \mathfrak{P}^\kappa_\epsilon \varphi \, \d \mm \ge - \int_\X g \mathfrak{P}^\kappa_\epsilon \varphi \, \d \mm,
\end{equation}
which implies that $\sfL^\kappa u_\epsilon + \mathfrak{P}^\kappa_\epsilon g \ge 0$. By choosing $\varphi := u_\epsilon$ in \eqref{eq:ineqg}, we get
\[  \mathcal{E}^\kappa (u_\epsilon) = - \int_\X u_\epsilon \sfL^\kappa u_\epsilon  \, \d \mm \le \int_\X u_\epsilon \mathfrak{P}^\kappa_\epsilon g \, \d \mm.  \]
Thus \eqref{eq:boundEk} follows by passing to the limit as $\epsilon \downarrow 0$. Similarly, we obtain:
\begin{align}\label{eq:non-neg-pre}
  -\mathcal E^\kappa(u,\varphi)+ \int_\X g\varphi\, \d\mm \geq 0, \qquad\forall \varphi\in \cF\;,~\varphi\geq0\;. 
\end{align}
\\
Thus applying Lemma \ref{lem:linfunc} to the linear functional $l\in \cF^{-1}_{qloc}$ given by
\begin{align*}
  \langle l, v\rangle := -\mathcal E^\kappa(u,v) +\int_\X vg\,\d\mm,\quad v\in \mathcal F\;,
\end{align*}
yields the representation via a suitable measure $\mu$.
\end{proof}
In the following we will denote by $\mathbb{M}^\kappa$ the space of $u \in \fD(\cE^\kappa) = \cF$ such that there exists a
$\sigma$-finite Borel measure $\mu = \mu_+ - \mu_-$ with $\mu_{\pm}$ non-negative and charging no $\mathcal E$-polar sets such that \eqref{eq:meaEk} holds. We will write $\sfL^\kappa_\star u := \mu$. Moreover, we set
\[\mathbb M^\kappa_\infty:=\mathbb M^\kappa\cap \mathcal F_b\quad \text{and} \quad \cF_{\rm bg} := \{ f \in \cF : f, \Gamma(f) \in L^\infty(\X, \mm)   \}.\]
We have the following calculus rule.
\begin{Corollary}\label{cor:prodLK}
For every $u \in \mathbb{M}^\kappa_\infty$ and $f \in \fD(\sfL) \cap \mathcal F_{\rm bg}$ we have $f u \in \mathbb{M}^\kappa_\infty$ with
\begin{equation}\label{eq:prodLK}
\sfL^\kappa_\star (f u) = \tilde{f} \sfL^\kappa_\star u + u \sfL f \mm + 2 \Gamma(u, f) \mm\;.
\end{equation}
\end{Corollary}
\begin{proof}
  Observe that $\tilde{f} \in L^\infty(\X, |\mu|)$, with
  $\mu=\sfL^\kappa_\star u$ and that $\tilde f$ coincides with $f$
  $|\mu|$-a.e. 
            Let $f_n$ be a sequence in $\cup_k\mathcal F_{G_k}$ for
  some admissible sequence of quasi-open sets $G_k$ which approximates
  $f$ w.r.t.~$\mathcal E_1$. Further let
  $\psi\in\cup_k\mathcal F_{G_k}$ be bounded. Note that also $uf_n$
  and $\psi f_n$ belong to $\cup_k\mathcal F_{G_k}$, and hence to
  $\mathcal F$. Thus, we have
\begin{align}\nonumber
- \mathcal{E}^\kappa(f_n u, \psi) &= - \mathcal E (f_nu, \psi) - \langle \kappa, f_n u \psi \rangle\\\nonumber
&= - \mathcal E(u, f_n \psi) -\mathcal E(f_n, u \psi) + \int_\X 2 \psi \Gamma(f_n, u)  \, \d \mm - \langle \kappa, f_n u \psi \rangle\\\nonumber
&= - \mathcal E^\kappa(u, f_n \psi) -\mathcal E(f_n, u \psi)+ \int_\X  2 \psi\Gamma(f_n, u) \, \d \mm\\\label{eq:prodLK1}
&= \int_\X \tilde{\psi} \tilde{f_n}\, \d \sfL^\kappa_\star u  -\mathcal E(f_n, u \psi)+ \int_\X    2\psi \Gamma(f_n, u) \, \d \mm\;.
\end{align}
Choosing in particular $\psi=f_nu$ yields
\begin{align*}
  - \mathcal{E}^\kappa(f_n u) &= \int_\X \tilde{f_n}^2\tilde u\, \d \sfL^\kappa_\star u  -\mathcal E(f_n, u^2f_n)+ \int_\X    2 f_nu\Gamma(f_n, u) \, \d \mm\;,
\end{align*}
and since $f$ is essentially bounded and $\tilde{u} \in L^1(\X, | \mu |)$, passing to the limit $n\to\infty$ shows that $fu\in \mathcal F$. Since $\mathcal E^\kappa$ is a closed form, this also shows that $\mathcal E^\kappa(f_nu,\psi)\to\mathcal E^\kappa(fu,\psi)$. Similarly, we then deduce that \eqref{eq:prodLK1} holds for $f$ in place of $f_n$ and a further integration by parts yields
\begin{align*}
  - \mathcal{E}^\kappa(f u, \psi) = \int_\X  \tilde{\psi} \tilde{f}\, \d \sfL^\kappa_\star u  +\int_\X \psi u\sfL f\d \mm + \int_\X    2\psi \Gamma(f, u) \, \d \mm\;.
\end{align*}
Finally, one readily extends the previous identity to arbitrary
$\psi \in \mathcal F$
and Lemma \ref{lem:Lstar} yields the claim.
\end{proof}
 \medskip
  \textbf{Measure-valued Bochner inequality.} 
Let us now extend the definition of the perturbed iterated carr\'e du
champ using the measure-valued taming operator.

Let us introduce the class of so-called \emph{test-functions}
\begin{align*}
   \mathbb D_\infty&:= \mathcal F_{\rm bg}\cap \fD_{\mathcal E}(\sfL)\;,
\end{align*}
We recall from \cite[Lemma~3.2]{Savare13} that $\D_\infty$ is an
algebra (i.e., closed w.r.t. pointwise multiplication) and, for every
$\Phi \in \mathcal{C}^\infty (\R^n)$ with $\Phi(0) = 0$ and
$\bold{f} = (f_i)_{i = 1}^n \in (\D_\infty)^n$ we have
$\Phi(\bold{f}) \in \D_\infty$.

Let us introduce the multilinear form $\mathbf{\Gamma}_2^{\kappa}$, defined by
\begin{equation}\label{def:Gamma2}
\mathbf{\Gamma}^{\kappa}_2 [f, g; \varphi] := \dfrac{1}{2} \int_\X \bigg(  \Gamma (f, g) \sfL^{\kappa} \varphi - (\Gamma (f, \sfL g) + \Gamma (g, \sfL f)) \varphi \bigg) \, \d \mm, \qquad \text{for} \, (f, g, \varphi) \in \fD(\mathbf{\Gamma}^{\kappa}_2),
\end{equation}
where $\fD(\mathbf{\Gamma}_2^{\kappa}) :=
  \fD_{\mathcal E}(\sfL) \times \fD_{\mathcal E}(\sfL) \times
  \fD_{L^\infty}(\sfL^{\kappa})$. If $f = g$  we write for short
\[ \mathbf{\Gamma}_2^{\kappa} [f; \varphi] := \mathbf{\Gamma}^{\kappa}_2[f, f; \varphi]\;.\]

In this notation, the Bochner inequality \eqref{eq:L2Bochner} takes the form
\[ \mathbf{\Gamma}_2^{\kappa} [f; \varphi] \ge \frac{2}{N}\int_\X\varphi (\sfL f)^2\,\d\mm, \quad \text{for every} \,\, (f, \varphi) \in \fD(\mathbf{\Gamma}^{\kappa}_2), \, \varphi \ge 0\;.    \]

\begin{Lemma}\label{lemma:Gamma2star}
If $\BE_2(\kappa, N)$ holds, then for every $f \in \D_\infty$ we have $\Gamma(f) \in \M_\infty^{\kappa}$ with
\begin{equation}\label{eq:boundEGamma}
\mathcal{E}^{\kappa}(\Gamma(f))\le -2 \int_\X \bigg( \Gamma(f) \Gamma(f, \sfL f) + \dfrac{2}{N} \Gamma(f) (\sfL f)^2 \bigg) \, \d \mm
\end{equation}
and
\begin{equation}\label{eq:boundLstar}
\dfrac{1}{2}\sfL^{\kappa}_\star \, \Gamma(f)  - \Gamma(f, \sfL f) \mm \ge \dfrac{2}{N} (\sfL f)^2 \mm.
\end{equation}
\end{Lemma}

\begin{proof}
  First of all we recall that for every $f \in \mathcal F_{\rm bg}$ we
  have
  $\Gamma(f) \in L^1(\X, \mm) \cap L^\infty(\X, \mm) \subset L^p(\X,
  \mm)$ for any $p \in [1, \infty]$. For $f \in \D_\infty$ we set
  \[ g := -2 \bigg( \Gamma(f, \sfL f) +\dfrac{2}{N}(\sfL f)^2 \bigg)
    \quad \text{and} \quad u := \Gamma(f)\;.\] Thanks to
  $\BE_2(\kappa, N)$, the hypothesis \eqref{eq:hypg} is satisfied with
  the so-defined $g$ and $u$. Therefore Lemma \ref{lem:Lstar} ensures
  that $\Gamma(f) \in \mathcal F$,
  $\Gamma(f) \in \M^{\kappa}_\infty$ and the validity of
  \eqref{eq:boundEGamma} and \eqref{eq:boundLstar}.
\end{proof}
For every $f \in {\mathbb{D}_\infty}$, we define the Borel
measure $\Gamma_{2,\star}^{\kappa}(f)$ by setting
\begin{equation}\label{def:Gammastar}
  \Gamma^{\kappa}_{2,\star}(f):= \dfrac{1}{2} \sfL^{\kappa}_\star \, \Gamma(f) - \Gamma(f, \sfL f) \mm.
\end{equation}
Observe that by Lemma \ref{lemma:Gamma2star} we have that
\begin{equation*}
  \Gamma^{\kappa}_{2,\star}(f) = \dfrac{2}{N} (\sfL f)^2 \mm + \mu_+, \quad \text{with} \,\, \mu_+ \ge 0.
\end{equation*}
Denoting by $\gamma_2^{\kappa}(u) \in L^1(\X, \mm)$ the density of its
absolutely continuous part w.r.t.~$\mm$, it holds
\begin{equation}\label{eq:decGammastar}
  \begin{split}
    \Gamma^{\kappa}_{2,\star}(f)  &= \gamma_2^{\kappa}(f) \mm+ \Gamma^{\kappa}_{2,\perp}(f),\\
    & \text{with } \Gamma^{\kappa}_{2,\perp}(f) \perp \mm, \,\,
    \gamma_2^{\kappa}(f) \ge \dfrac{2}{N} (\sfL f)^2 \,\,\,
    \mm\text{-a.e. in }\X, \text{ and } \Gamma^{\kappa}_{2,\perp}(f)
    \ge 0.
  \end{split}
\end{equation}
Finally, as in \eqref{def:Gamma2}, we define for $f, g \in \D_\infty$
\begin{equation*}
  \Gamma^{\kappa}_{2,\star}(f, g) := \dfrac{1}{4}\Gamma^{\kappa}_{2,\star}(f + g) - \dfrac{1}{4}\Gamma^{\kappa}_{2,\star}(f - g) = \dfrac{1}{2}\bigg( \sfL^{\kappa}_\star \Gamma(f, g) - \Gamma(f, \sfL g) \mm - \Gamma(g, \sfL f) \mm  \bigg),
\end{equation*}
and, similarly,
\begin{equation*}
  \gamma_2^{\kappa}(f, g) := \dfrac{1}{4}\gamma_2^{\kappa}(f + g) - \dfrac{1}{4}\gamma_2^{\kappa}(f - g), \qquad \Gamma^{\kappa}_{2,\star}(f, g) = \gamma_2^{\kappa}(f, g) \mm + \Gamma^{\kappa}_{2,\perp}(f, g).
\end{equation*}

In the next lemma we note a chain rule for
$\Gamma^{\kappa}_{2,\star}$.

\begin{Lemma}\label{lemma:fundid}
Let  $\bold{f} = (f_i)_{i = 1}^n \in (\D_\infty)^n$  and let $\Phi \in \mathcal{C}^\infty (\R^n)$ with $\Phi(0) = 0$. Then 
\begin{equation}\label{eq:fundid}
\begin{split}
\displaystyle \Gamma^{\kappa}_{2,\star}&(\Phi(\bold{f})) = \sum_{i, j} \Phi_i(\tilde{\bold{f}}) \Phi_j(\tilde{\bold{f}}) \Gamma^{\kappa}_{2,\star}(f^i, f^j)\\
&+\bigg( 2 \sum_{i, j, k} \Phi_i(\bold{f}) \Phi_{j k}(\bold{f}) H[f^i](f^j, f^k) + \sum_{i, j, k ,h} \Phi_{i j}(\bold{f}) \Phi_{k h}(\bold{f}) \Gamma(f^i, f^j) \Gamma(f^k, f^h) \bigg) \mm,
\end{split}
\end{equation}
where, for every $f, g, h \in \D_{\infty}$, $H[f](g, h)$ is defined by 
\begin{equation}\label{def:Hf}
H[f](g, h) := \dfrac{1}{2} \bigg(  \Gamma(g, \Gamma (f, h)) + \Gamma (h, \Gamma(f, g)) - \Gamma(f, \Gamma(g, h)) \bigg).
\end{equation}
In the same way,
\begin{equation}\label{eq:fundidgamma}
\begin{split}
\displaystyle \gamma_2^{\kappa}&(\Phi(\bold{f})) = \sum_{i, j} \Phi_i(\tilde{\bold{f}}) \Phi_j(\tilde{\bold{f}}) \gamma_2^{\kappa}(f^i, f^j)\\
&+\bigg( 2 \sum_{i, j, k} \Phi_i(\bold{f}) \Phi_{j k}(\bold{f}) H[f^i](f^j, f^k) + \sum_{i, j, k ,h} \Phi_{i j}(\bold{f}) \Phi_{k h}(\bold{f}) \Gamma(f^i, f^j) \Gamma(f^k, f^h) \bigg) \mm.
\end{split}
\end{equation}
\end{Lemma}

\begin{proof}
  Recall that $\Phi(\bold{f})\in\D_\infty$.
We set for $i, j \in \{ 1, \dots, n \}$ 
\[
\begin{split}
g^{i j} := \Gamma (f^i, f^j) \in {\M^{\kappa}_\infty}, \quad \ell^i &:= \sfL f^i \in \cF, \\ \phi_i := \Phi_i (\tilde{\bold{f}}), \quad \phi_{i j} := \Phi_{i j} (\tilde{\bold{f}}), \quad &\phi_{i j k} := \Phi_{i j k} (\tilde{\bold{f}}) \in \mathbb{D}_\infty\;.
\end{split}
\]
We have $\Gamma(\Phi(\bold{f})) = g^{i j} \phi_i \phi_j$, while Lemma \ref{lemma:Gamma2star} ensures that $\Gamma(\Phi(\bold{f})) \in \M^{\kappa}_\infty$.  

Since $\phi_i \phi_j \in \D_\infty$, the identity in \eqref{eq:prodLK} yields (with Einstein summation convention)
\[
\begin{split}
\dfrac{1}{2}\sfL^{\kappa}_\star(\Gamma(\Phi(\bold{f}))) = \dfrac{1}{2}\sfL^{\kappa}_\star (g^{i j} \phi_i \phi_j) = \dfrac{1}{2}\tilde{\phi}_i \tilde{\phi}_j \sfL^{\kappa}_\star g^{ij} + \bigg(\dfrac{1}{2} g^{ij} \sfL (\phi_i \phi_j) +  \Gamma(\phi_i\phi_j, g^{ij}) \bigg)\mm.
\end{split}
\]
From here one can proceed the calculation exactly as in the proof of \cite[Lemma~3.3]{Savare13} to obtain \eqref{eq:fundid}, and \eqref{eq:fundidgamma}.
\end{proof}
\subsection{Self-Improvement of the $L^2$-Taming Condition}

The following pointwise estimate for the $\Gamma$ operator will be
crucial to obtain self-improvement.

\begin{Theorem}
If $\BE_2(\kappa,N)$ holds with $\kappa$ a signed measure in the extended Kato class $\KK_{1-}(\X)$. Then for any $f, g, h \in \mathbb{D}_\infty$ we have 
\begin{align}
\big| H[f](g, h) \big|^2 &\le \bigg( \gamma^\kappa_2(f) - \dfrac{2}{N} (\sfL f)^2\bigg) \Gamma(g) \Gamma(h), \label{eq:imp1}\\
\sqrt{\Gamma(\Gamma(f, g))} &\le \sqrt{\gamma^\kappa_2(f) - \dfrac{2}{N} (\sfL f)^2 } \sqrt{\Gamma(g)} + \sqrt{\gamma^\kappa_2(g) - \dfrac{2}{N} (\sfL g)^2 } \sqrt{\Gamma(f)}, \label{eq:imp2}\\
\Gamma(\Gamma(f)) &\le 4 \bigg( \gamma^\kappa_2(f) - \dfrac{2}{N} (\sfL f)^2  \bigg)\Gamma(f), \label{eq:imp3}
\end{align}
where all the inequalities are to be intended in the $\mm$-a.e. sense on $\X$.
\end{Theorem}

\begin{proof}
First of all we observe that $\Gamma(f), \Gamma(g), \Gamma(h) \in \cF_b$, thanks to Lemma \ref{lem:Lstar}. Then we take the polynomial $\Phi \colon \R^3 \to \R$ defined by
\begin{equation*}
\Phi(\bold{f}) := \lambda f^1 + (f^2 - a)(f^3 - b) -ab, \quad \lambda, a, b \in \R.
\end{equation*}
In particular we have
\[ 
\begin{split}
\Phi_1(\bold{f}) &= \lambda,  \qquad \Phi_2(\bold{f}) = f^3 - b, \qquad \Phi_3(\bold{f}) = f^2 - a, \\
\Phi_{2, 3}(\bold{f}) &= \Phi_{3, 2}(\bold{f}) = 1, \qquad \Phi_{i, j}(\bold{f}) = 0  \text{ otherwise}.
\end{split}
\]
If $\bold{f} \in \D_\infty$, then Lemma \ref{lemma:fundid} yields $\Phi(\bold{f}) \in \D_\infty$, while inequality \eqref{eq:decGammastar} provides
\begin{equation}\label{ineq:gamma2}
\gamma_2^\kappa(\Phi(\bold{f})) \ge \dfrac{2}{N} (\sfL \Phi(\bold{f}))^2, \quad \mm\text{-a.e. in } \X.
\end{equation}
Now we recall the identity in \eqref{eq:fundidgamma}, and we observe that both sides of the inequality depend on $\lambda, a, b \in \R$: we choose a dense subset $Q \subset \R^3$ of parameters $(\lambda, a, b)$ such that inequality \eqref{ineq:gamma2} holds for every $(\lambda, a, b) \in Q$ and $\mm$-a.e. in $\X$. The continuous dependence of the left- and the right-hand side of the inequality w.r.t. $\lambda, a, b$ allows to conclude that actually  \eqref{ineq:gamma2} holds for every $(\lambda, a, b) \in \R^3$ and for $\mm$-a.e. $x \in \X$. Therefore, up to a negligible set, for every $x \in \X$ we choose $a := f^2(x)$ and $b := f^3(x)$ in such a way that $\Phi_2(\bold{f})(x) = \Phi_3(\bold{f})(x) = 0$ and
\[
\lambda^2 \gamma_2^\kappa(f^1) + 4 \lambda H[f^1](f^2, f^3) + 2 \big( \Gamma(f^2) \Gamma(f^3) + \Gamma(f^2, f^3)^2 \big) \ge \dfrac{2}{N} \lambda^2 \big(\sfL f^1\big)^2.
\]
The arbitrariness of $\lambda$, together with the fact that
$\Gamma(f^2, f^3)^2 \le \Gamma(f^2)\Gamma(f^3)$ gives the following
inequality
\begin{equation}
\bigg( H[f^1](f^2, f^3) \bigg)^2 \le \bigg( \gamma^\kappa_2(f^1) - \dfrac{2}{N} (\sfL f^1)^2\bigg) \Gamma(f^2) \Gamma(f^3),
\end{equation}
which proves \eqref{eq:imp1}. As for \eqref{eq:imp2}, we start noticing that 
\begin{equation}
H[f](g, h) + H[g](f, h) = \Gamma(\Gamma(f, g), h).
\end{equation}
Hence, a direct computation yields
\begin{equation}\label{eq:estGammaGamma}
\big| \Gamma(\Gamma(f, g), h)  \big| \le \bigg[ \sqrt{\gamma^\kappa_2(f) - \dfrac{2}{N} (\sfL f)^2 } \sqrt{\Gamma(g)} + \sqrt{\gamma^\kappa_2(g) - \dfrac{2}{N} (\sfL g)^2 } \sqrt{\Gamma(f)}\bigg] \sqrt{\Gamma(h)}.
\end{equation}
Inequality \eqref{eq:estGammaGamma} can be extended to arbitrary $h \in \cF_b$ via approximation based on \eqref{def:Pfrak}.
Choosing $h = \Gamma(f, g)$ yields \eqref{eq:imp2}. Inequality \eqref{eq:imp3} then
follows by taking $g = f$.
\end{proof}

As another preparation, we show that the class of test functions
$\D_\infty$ is dense in $\cF$. This will follow from a variant of the
reverse Poincar\'e inequality.

\begin{Proposition}\label{prop:revPoin2}
  Let $(\X, \mathcal{E}, \mm)$ satisfy $\BE_2(\kappa,\infty)$ with
  $\kappa\in \KK_{1-}(\X)$ a signed measure with decomposition
  $\kappa=\kappa^+-\kappa^-$ for non-negative measures $\kappa^+,\kappa^-$. Then
  $(\X,\mathcal{E},\mm)$ also satisfies
  $\BE_2(-\kappa^-,\infty)$. Moreover, for every
  $f \in L^2(\X,\mm)\cap L^\infty(\X,\mm)$ and every $t > 0$ it holds
\begin{equation}\label{eq:bounded-gradient}
\Gamma(P_t f) \leq \frac{1}{2t} \|P^{-\kappa^-}_t\|_{L^\infty,L^\infty}\cdot \|f\|_{L^\infty}\;. 
\end{equation}
\end{Proposition}

\begin{proof} {\bf (i)}: To see that $(\X,\mathcal{E},\mm)$ satisfies
  $\BE_2(-\kappa^-,\infty)$, we note that
  $A^{\kappa}=A^{\kappa^+}+A^{-\kappa^-}$. Hence, for any non-negative
  $h$ and $t>0$, we have $P^{\kappa}_t h\leq P^{-\kappa^-}_t h$, so that
  $\GE_2(\kappa,\infty)$ implies $\GE_2(-\kappa^-,\infty)$ and we
  conclude by the equivalence of $\BE_2$ and $\GE_2$, Theorem
  \ref{thm:BE2GE2}. \smallskip

  {\bf (ii)}: To show \eqref{eq:bounded-gradient}, let
  $f,g \in L^\infty(\X,\mm)$ with $g \ge 0$. For any $t>0$, we set
  $f_t := P_t f$, $g_t := P^{-\kappa^-}_t g$, and define for
  $s\in[0,t]$
\[
{\Upsilon}(s) := \int_\X (f_{t-s})^2 g_s \, \d \mm\;.
\]
Then we have for all $s\in(0,t)$:
\[
\begin{split}
\dfrac{\d}{\d s} \Upsilon(s) & =  \int_\X \big(-2 f_{t-s} \sfL f_{t-s} g_s + f_{t-s}^2 \sfL^{-\kappa^-}g_s \big)\, \d \mm =  2 \cE (f_{t-s} g_s, f_{t-s}) - \cE^{-\kappa^-}(f_{t-s}^2, g_s)\\
&\geq 2 \cE (f_{t-s} g_s, f_{t-s}) - \cE(f_{t-s}^2, g_s) =2\int_\X g_s \Gamma(f_{t-s}) \, \d \mm  \;.
\end{split}
\]
Here we have used that
$\cE^{-\kappa^-}(f_{t-s}^2, g_s)=\cE(f_{t-s}^2, g_s)-\langle \kappa^-,
f^2_{t-s}g_s\rangle\leq \cE(f^2_{t-s},g_s)$.  The $L^2$-gradient
estimate $\GE_2(-\kappa^-,\infty)$ then yields
\begin{align*}
  \int_\X g\Big[P^{-\kappa^-}_t f^2  -  (P_t f)^2 \Big]\, \d \mm &= 2\int_0^t \int_\X  P_s^{-\kappa^-}g \Gamma(P_{t-s}f) \,\d\mm \, \d s \geq 2t \int_\X g\Gamma(P_t f)\d \mm\;,                                                             
\end{align*}
and we conclude by the arbitrariness of $g$.
\end{proof}

\begin{Corollary}\label{cor:densDinf}
  Let $(\X, \mathcal{E}, \mm)$ satisfy
  $\BE_2(\kappa,\infty)$ with $\kappa\in \cF^{-1}_{\rm qloc}$ a signed measure in the extended Kato class $\KK_{1-}(\X)$. Then the set $\mathbb{D}_\infty$ is dense in $\mathcal{F}$.
\end{Corollary}
\begin{proof}
As a direct consequence of \eqref{eq:bounded-gradient} we have that
\[  f \in L^2 \cap L^\infty(\X,\mm) \qquad \Rightarrow \qquad P_t f \in \mathbb{D}_\infty, \quad \forall t > 0\;. \]
This in particular provides the density of $\mathbb{D}_\infty$ in $\mathcal{F}$.
\end{proof}

Now, we can prove the main result of this section.

\begin{Theorem}[$\BE_2(\kappa, N)$ implies $\BE_1(\kappa, N)$]\label{thm:selfimp}
  Let $(\X, \mathcal{E}, \mm)$ be a Dirichlet space satisfying
  $\BE_2(\kappa,N)$ for $\kappa\in \cF^{-1}_{\rm qloc}$  a signed measure in the extended Kato class $\KK_{1-}(\X)$. Then the condition
  $\BE_1(\kappa, N)$ holds. Moreover, for any $f\in \cF$ and
  $\alpha\in [1/2,1]$ it holds 
  \begin{equation}\label{dis:corsi1}
\Gamma\big( P_t f \big)^\alpha \le P_t^{\alpha \kappa}\big( \Gamma(f)^\alpha \big)\quad \mm\text{-a.e.}\;,
\end{equation}
and, if $N<\infty$, we have 
\begin{equation}\label{dis:corsi2}
\Gamma\big( P_t f \big)^\alpha + \dfrac{4 \alpha}{N} \int_0^t P_s^{\alpha \kappa}\bigg(\dfrac{(\Gamma (P_{t-s} f))^\alpha}{\Gamma (P_{t-s} f)} (\sfL P_{t-s} f)^2\bigg) \, \d s \le P_t^{\alpha \kappa}\big( \Gamma(f)^\alpha \big)\;.
\end{equation}
\end{Theorem}
\begin{proof}
  Recall that $\BE_2$ is equivalent to the gradient estimate $\GE_2$. We will prove \eqref{dis:corsi2}, which gives in particular $\GE_1(\kappa,N)$ for $\alpha=1/2$, and recall that also $\BE_1$ and $\GE_1$ are equivalent.
  
Fix $\alpha \in [1/2, 1]$ and define the concave and smooth function $\eta_\epsilon (r) := (\epsilon + r)^{\alpha} - \epsilon^\alpha$, for $\epsilon > 0$ and $r \ge 0$. In particular, $\eta_\epsilon$ is Lipschitz with
\begin{equation}\label{prop:eta}
\eta_\epsilon(r) \le r^\alpha, \quad (r + \epsilon) \eta_\epsilon'(r) = \alpha \eta_\epsilon(r) + \alpha \epsilon^\alpha, \quad r \eta'_\epsilon \ge \alpha \eta_\epsilon, \quad 2 \eta'_{\epsilon} + 4 r \eta_\epsilon'' \ge 0.
\end{equation}
Furthermore, for $t > 0$,  $\tau, s \in [0, t]$, we define the following curves
\begin{equation}
f_\tau := P_\tau f,  \quad u_\tau:= \Gamma(f_\tau), \quad \zeta_s := P_s^{\alpha \kappa} \zeta, \quad G_\epsilon(s) := \int_\X \eta_\epsilon (u_{t-s}) \zeta_s \, \d \mm,
\end{equation}
where  $\zeta \in \cF_b$ is a non-negative function, and $f \in \D_\infty$. Let us point out that for every $s$ we have that $f_{s} \in \D_\infty$ thanks to the gradient estimate $\GE_2(\kappa,\infty)$ and the fact that $P^{\kappa}_s$ is bounded on $L^\infty$. So, thanks to Lemma \ref{lem:Lstar}, it follows that $u_{t-s} \in \M^{\kappa}_\infty$, and, in particular, that $u_{t-s} \in \cF \cap L^1 \cap L^\infty(\X, \mm)$. Hence, a direct computation gives
\[
\dfrac{\d}{\d s} u_{t-s} = -2 \Gamma\big(f_{t-s}, \sfL f_{t-s}\big), \quad 
\dfrac{\d}{\d s} \eta_{\epsilon}(u_{t-s}) = -2 \eta'_{\epsilon} \Gamma\big(f_{t-s}, \sfL f_{t-s}\big) \,\, \text{ in } L^1 \cap L^2(\X, \mm).
\]
We are going to use these identities while differentiating $G_\epsilon(\,\cdot\,)$ with respect to $s \in (0, t)$:
\begin{equation}\label{eq:G'}
G'_\epsilon(s) = \int_\X \Big(\eta_\epsilon(u_{t-s}) \sfL^{\alpha \kappa}\zeta_s - 2 \eta'_\epsilon(u_{t-s})  \Gamma\big(f_{t-s}, \sfL f_{t-s}\big) \zeta_s \Big) \, \d \mm
\end{equation}
According to the definition of $\cE^{\alpha\kappa}$, we have
\[
\int_\X \eta_\epsilon(u_{t-s}) \sfL^{\alpha \kappa}\zeta_s \, \d \mm = -\cE\big(\eta_\epsilon(u_{t-s}), \zeta_s\big) - \alpha \langle \kappa, \eta_\epsilon(u_{t-s}) \zeta_s \rangle.
\]
The chain rule for $\Gamma$ yields
\[
\begin{split}
-\cE\big(\eta_\epsilon(u_{t-s}),\zeta_s\big) &= -\int_\X \eta'_\epsilon(u_{t-s}) \Gamma(u_{t-s}, \zeta_s)\, \d \mm\\
&= -\int_\X \bigg( \Gamma( u_{t-s}, \eta'_\epsilon(u_{t-s}) \zeta_s) - \Gamma(u_{t-s})\eta''_\epsilon(u_{t-s}) \zeta_s \bigg) \, \d \mm\\
&= -\mathcal{E}^{\kappa}\big(u_{t-s}, \eta'_\epsilon(u_{t-s}) \zeta_s\big) + \langle \kappa,  \eta'_\epsilon(u_{t-s}) u_{t-s} \zeta_s \rangle + \int_\X \Gamma(u_{t-s})\eta''_\epsilon(u_{t-s}) \zeta_s \, \d \mm\\
&\overset{\eqref{eq:meaEk}}{=} \int_\X \tilde{\eta}'_\epsilon(u_{t-s}) \tilde{\zeta}_s \,\d \sfL^{\kappa}_\star(u_{t-s}) + \langle \kappa,  \eta'_\epsilon(u_{t-s}) u_{t-s} \zeta_s \rangle + \int_\X \Gamma(u_{t-s})\eta''_\epsilon(u_{t-s}) \zeta_s \, \d \mm
\end{split}
\]
Inserting these identities in \eqref{eq:G'}, and recalling the definition of $\Gamma^{\kappa}_{2,\star}(f)$ in \eqref{def:Gammastar}, we find
\begin{equation*}
\begin{split}
G_\epsilon'(s) = 2 \int_\X \eta'_\epsilon(\tilde{u}_{t-s})  \tilde{\zeta}_s \, \d \Gamma^{\kappa}_{2, \star}&(f_{t-s})  + \int_\X \Gamma(u_{t-s})\eta''_\epsilon(u_{t-s}) \zeta_s \,\d \mm \bigskip\\ &+ \big\langle \kappa, \big(\eta'_\epsilon(u_{t-s}) u_{t-s} - \alpha \eta_\epsilon(u_{t-s})\big) \zeta_s\big \rangle.
\end{split}
\end{equation*}
Keeping in mind \eqref{eq:decGammastar}, we have
\[
\int_\X \eta'_\epsilon(\tilde{u}_{t-s})  \tilde{\zeta}_s \, \d \Gamma^{\kappa}_{2, \star}(f_{t-s}) \ge \int_\X \eta'_\epsilon({u}_{t-s}){\zeta}_s \, \gamma^{\kappa}_2 (f_{t-s}) \, \d \mm
\]
while inequality \eqref{eq:imp3}, together with the fact that $\eta_\epsilon'' \le 0$, ensures that
\[
\int_\X \Gamma(u_{t-s})\eta''_\epsilon(u_{t-s}) \zeta_s \,\d \mm \ge 4 \int_\X \eta''_{\epsilon}(u_{t-s})\Big( \gamma_2^\kappa(f_{t-s}) - \dfrac{2}{N} (\sfL f_{t-s})^2\Big) u_{t-s} \zeta_s \, \d \mm.
\]
Summing up this chain of inequalities, we obtain
\begin{equation*}
\begin{split}
G'_\epsilon (s )& \ge  2 \int_\X \eta'_\epsilon({u}_{t-s}){\zeta}_s \, \gamma^\kappa_2 (f_{t-s}) \, \d \mm + 4 \int_\X \eta''_{\epsilon}(u_{t-s})\Big( \gamma_2^\kappa(f_{t-s}) - \dfrac{2}{N} (\sfL f_{t-s})^2\Big) u_{t-s} \zeta_s \, \d \mm \bigskip\\
& \qquad + \big\langle \kappa, \big(\eta'_\epsilon(u_{t-s}) u_{t-s} - \alpha \eta_\epsilon(u_{t-s})\big) \zeta_s\big \rangle\\
& = \int_\X \Big(  2 \eta'_\epsilon(u_{t-s}) + 4 \eta_\epsilon''(u_{t-s})u_{t-s}  \Big) \Big( \gamma_2^\kappa(f_{t-s}) - \dfrac{2}{N} (\sfL f_{t-s})^2  \Big) \zeta_s \, \d \mm\\
& \qquad + \dfrac{4}{N} \int_\X \eta'_\epsilon(u_{t-s}) (\sfL f_{t-s})^2 \zeta_s \, \d \mm + \big\langle \kappa, \big(\eta'_\epsilon(u_{t-s}) u_{t-s} - \alpha \eta_\epsilon(u_{t-s})\big) \zeta_s\big \rangle\;.
\end{split}
\end{equation*}
Now we observe that the first term is non-negative, while, since $\eta_\epsilon' \geq 0$, for the second term it holds
\[\begin{split}
\dfrac{4}{N} \int_\X \eta'_\epsilon(u_{t-s}) (\sfL f_{t-s})^2 \zeta_s \, \d \mm 
\overset{\eqref{prop:eta}}{\geq} &\dfrac{4 \alpha}{N} \int \dfrac{\eta_\epsilon(u_{t-s})}{u_{t-s}} (\sfL f_{t-s})^2 \zeta_s \, \d \mm\;.
\end{split}\]
As for the last term, note that by \eqref{prop:eta}, 
\[
  \big| \eta'_\epsilon(u_{t-s})u_{t-s}- \alpha \eta_\epsilon(u_{t-s})\big|
  = \big|\epsilon\eta'_\epsilon(u_{t-s})+\alpha\epsilon^\alpha \big|
  \leq 2\alpha \epsilon^\alpha\;,
\]
and hence as $\epsilon\to0$:
\[
  \int_0^t\big\langle \kappa, \big(\eta'_\epsilon(u_{t-s}) u_{t-s} - \alpha \eta_\epsilon(u_{t-s})\big) \zeta_s\big \rangle\d s
  \leq 2\alpha\epsilon^\alpha\int_0^t\int\zeta_s\d|\kappa|\d s\rightarrow 0\;.
\]
By continuity of $G_\epsilon$ and the monotone convergence of $\eta_\epsilon(r)\to r^\alpha$ we can pass to  the limit as $\epsilon \downarrow 0$, obtaining
\begin{equation}\label{dis:lastsi1}
\int_\X \Gamma(f)^\alpha P_t^{\alpha \kappa} \zeta \, \d \mm \geq \int_\X \big( \Gamma(P_t f) \big)^\alpha \zeta \, \d \mm
\end{equation}
or, taking care of the dimension term,
\begin{equation}\label{dis:lastsi2}
\int_\X \Gamma(f) P_t^{\alpha \kappa} \zeta \, \d \mm \geq \int_\X \big( \Gamma(P_t f) \big)^\alpha \zeta \, \d \mm + \dfrac{4 \alpha}{N} \int_0^t \int \dfrac{(\Gamma (P_{t-s} f))^\alpha}{\Gamma (P_{t-s} f)} (\sfL f_{t-s})^2 \zeta_s \, \d \mm \, \d s.
\end{equation}
Then we use the density of $\D_\infty$ in $\cF$ (Corollary \ref{cor:densDinf}) in order to extend \eqref{dis:lastsi1} and \eqref{dis:lastsi2} to an arbitrary $f \in \cF$, and obtain \eqref{dis:corsi1} and \eqref{dis:corsi2}, since $\zeta$ is arbitrary.
\end{proof}
\begin{Proposition} Assume that the Dirichlet space $(\X,\cE,\mm)$ is tamed by a signed measure $\kappa\in \cF^{-1}_{\rm qloc}$ which is in the extended Kato class $\KK_{1-}(\X)$. Then for any $f\in \fD_\cE(\sfL)$ it holds $\Gamma(f)^{1/2}\in\cF$.
\end{Proposition}

\begin{proof}
Let us first consider $f \in \D_\infty$. Directly from \eqref{eq:imp3} we have
\[
\Gamma\big(\Gamma(f)^{{1}/{2}}\big) = \dfrac{\Gamma(\Gamma(f))}{\Gamma(f)} \le 4 \bigg(\gamma^\kappa_2(f) - \dfrac{2}{N} \big( \sfL f   \big)^2\bigg), \quad \mm\text{-a.e. on } \X,
\]
and, integrating it, we get
\[
\begin{split}
 \dfrac{1}{4} \int_\X \Gamma\big(\Gamma(f)^{{1}/{2}}\big) \, \d \mm &\le  \int_\X \bigg( \gamma^\kappa_2(f) - \dfrac{2}{N} ( \sfL f  )^2 \bigg) \, \d \mm \overset{\eqref{eq:decGammastar}}{\le} \int_\X \d \Gamma^{\kappa}_{2,\star}(f) - \int_\X \dfrac{2}{N} ( \sfL f  )^2  \, \d \mm \\
&= \dfrac{1}{2}  \int_\X \d \sfL^{\kappa}_\star \, \Gamma(f) - \int_\X \bigg( \Gamma(f, \sfL f) + \dfrac{2}{N} ( \sfL f  )^2 \bigg) \, \d \mm\\
&= \dfrac{1}{2}  \int_\X \d \sfL^{\kappa}_\star \, \Gamma(f) + \bigg( 1 - \dfrac{2}{N} \bigg) \int_\X ( \sfL f )^2 \, \d \mm
\end{split}
\]
which is finite for $f \in \D_\infty \subset \fD_\cE(\sfL)$, since Lemma \ref{lemma:Gamma2star} ensures that $\Gamma(f) \in \mathbb M^\kappa_\infty$.

As for the general case, let $f \in \fD_\cE(\sfL)$, and, for any $n \in \N$ and $t > 0$, let us consider $P_t(f_n)$, where $f_n := \min \{ \max\{ f, -n  \}, n  \}$. Proposition \ref{prop:revPoin2} guarantees that  $P_t(f_n) \in \D_\infty$, hence the previous argument ensures that $\Gamma(P_t(f_n))^{1/2} \in \cF$. Now, recalling that $\sfL P_t(f_n) \to \sfL P_t(f) $ in $L^2(\X,\mm)$ as $n \to \infty$ and that $\sfL P_t(f) \to \sfL f$ in $L^2(\X,\mm)$ as $t \downarrow 0$, and similarly that $P_t(f_n)$ converges to $P_t(f)$ in $\cE_1$ as $n \to \infty$, while $P_t (f) \to f$ and $\sfL P_t(f) \to \sfL f$ in $\cE_1$ and $L^2(\X,\mm)$, respectively, as $t \downarrow 0$, the conclusion follows by letting first $n \to \infty$, and then $t \downarrow 0$, being $\cE$ a closed form.
\end{proof}

\section{Sub-tamed Spaces}\label{sec:sub-tamed}
  
As we have seen before, the concept of (distribution-valued) synthetic Ricci bounds has powerful  applications to semigroups with Neumann boundary conditions. In its standard form, however, it will not apply to semigroups with Dirichlet boundary conditions. 

\begin{Example}
Let $(\X,\cE,\mm)$ be the canonical Dirichlet space with Dirichlet boundary conditions on a bounded, connected, non-empty open subset $\X \subset \sfM$ with Lipschitz boundary of a smooth Riemannian manifold $(\sfM,\sfg)$, i.e.\ $\cE(f)=\frac12\int_\X|\nabla f|^2\,\d\mm$ with $\fD(\cE)=W^{1,2}_0(\X)$ and $\mm=\text{vol}\big|_\X$. Then there will exist no 2-moderate $\kappa$ such that this Dirichlet space satisfies {\sf BE}$_1(\kappa, \infty)$.
\end{Example}
 
\begin{proof} 
Assume that $(\X,\cE,\mm)$ is tamed with 2-moderate $\kappa$. Consider the first eigenfunction $\varphi>0$ for the Dirichlet Laplacian such that $-\frac12\Delta\varphi=\lambda\varphi$ for some $\lambda>0$. Then
\[
e^{-\lambda t}|\nabla\varphi|=|\nabla P_t\varphi|\le P^{\kappa/2}_t|\nabla\varphi|\le C_t\, \big(P_t|\nabla\varphi|^2\big)^{1/2}.
\]
with $C_t:=\big\|P^{\kappa}_t1\big\|_{L^\infty}^{1/2}<\infty$. By local regularity, $|\nabla\varphi|^2\in L^2(\X,\mm)$ and by the regularizing property of $P_t$ this in turn implies that $P_t|\nabla \varphi|^2 \in W^{1,2}_0(\X)$.
 Hence $P_t|\nabla \varphi|^2$ vanishes on the boundary, and so does $|\nabla \varphi|$ by the previous estimate. This means that $\varphi$ satisfies Neumann boundary conditions too, and therefore, by \cite[Proposition 6.4]{Sturm2019}, it belongs to the ``metric-measure'' Sobolev space $W^{1,2}(\bar{\X}) = \fD({\sf Ch})$ built over the closure of $\X$ endowed with the distance induced by $\sfM$ and the restricted measure (see Example \ref{dir-cheeger}); as for functions in $W^{1,2}(\bar{\X})$ integration by parts formula holds and $\mathds{1} \in W^{1,2}(\bar{\X})$, we get
\[
\int_\X \varphi\,\d\mm = -\frac{1}{2\lambda}\int_\X \Delta\varphi\,\d\mm = \frac{1}{2\lambda}\int_\X \langle\nabla\varphi,\nabla\mathds{1}\rangle\,\d\mm = 0
\]
and this is in contradiction with $\varphi > 0$.
\end{proof}

\subsection{Reflected Dirichlet Spaces and Sub-taming}\label{reflDirspace}
  
In order to apply it to semigroups with Dirichlet boundary conditions, we will extend the concept of (distribution-valued) synthetic Ricci bounds and introduce the notion of ``sub-tamed spaces''. Given a Dirichlet space $(\X,\cE,\mm)$ we will construct the ``sub-taming energy'' in terms of the reflected Dirichlet space $(\bar{\X},\bar\cE,\bar\mm)$.

To introduce the latter, let a strongly local, quasi-regular Dirichlet space $(\X,\cE,\mm)$ which admits a carr\'e du champ $\Gamma$ be given. In particular, $\cE(f)=\frac12\int_\X \Gamma(f)\,\d\mm$ for all $f\in\cF := \fD(\cE)$. Locality of $\cE$ allows to extend the definition of $\Gamma$ to $\cF_{\rm qloc}$ and thus to define the \emph{reflected Dirichlet form}
\[
\bar\cE(f):=\frac12\int_\X \Gamma(f)\,\d\mm
\]
with
\[
\bar\cF := \fD(\bar\cE) := \Big\{f\in L^2(\X,\mm)\cap \cF_{\rm qloc}: \int_\X \Gamma(f)\,\d\mm<\infty\Big\}.
\] 
Regarded as a Dirichlet form on a suitable extension $\bar{\X}$ of the space $\X$, this indeed is again a
strongly local, quasi-regular Dirichlet form, \cite[Remark 6.6.11]{Chen-Fuku}. The initial set $\X$ will be an $\bar\cE$-quasi-open subset of $\bar{\X}$ (up to an $\cE$-polar set) and the measure $\bar\mm$ extends $\mm$ in such a way that $\bar\mm(\bar{\X} \setminus \X)=0$.
 
In the sequel, $(\mathbb{P}_x, B_t)_{x\in \bar\X, t\ge0}$ will denote a fixed $\bar\mm$-reversible, continuous, strong Markov process (with life time $\bar\zeta$) properly associated with $(\bar\cE,\bar\cF)$. Killing this process at the first exit from $\X$ will yield a process with life time $\zeta:=\bar\zeta\wedge \bar\tau_\X$ which is properly associated with $(\cE,\cF)$. Generator, resolvent and semigroup associated to $(\cE,\cF)$ henceforth will be denoted by $(\sfL,\fD(\sfL))$, $(G_\alpha)_{\alpha>0}$ and
$(T_t)_{t\ge0}$, resp.  The corresponding quantities associated to $(\bar\cE,\bar\cF)$ will be denoted by $(\bar \sfL,\fD(\bar \sfL))$, $(\bar G_\alpha)_{\alpha>0}$ and $(\bar T_t)_{t\ge0}$.  In terms of the reflected Dirichlet form $(\bar\cE,\bar\cF)$ on $L^2(\bar{\X},\bar\mm)$, we will define the spaces of distributions $\bar\cF^{-1}$ and $\bar\cF^{-1}_{\rm qloc}$ as well as the Feynman-Kac semigroups $(\bar P^\kappa_t)_{t\ge0}$ for $\kappa\in \bar\cF^{-1}_{\rm qloc}$ and the energy functionals $\bar\cE^\kappa$.
  
\begin{definition}\label{def:sub-tamed} 
We say that the Dirichlet space $(\X,\cE,\mm)$ is {\bf sub-tamed} if there exists a moderate $\kappa\in\bar\cF^{-1}_{\rm qloc}$ (called distribution-valued Ricci bound) such that the following Bochner inequality $\overline \BE_1(\kappa,\infty)$ holds: for all $f \in \fD_{\cF}(\sfL)$ and all non-negative $\varphi \in \fD(\bar\sfL^{\kappa/2})$
\begin{equation}\label{eq:sub-L1Bochner}
\int \bar\sfL^{\kappa/2} \varphi \Gamma(f)^{1/2} \d\mm - \int \varphi \frac{\Gamma(f,\sfL f)}{\Gamma(f)^{1/2}}\d\mm \geq 0\;.
\end{equation}
If moreover, this $\kappa$ is also $p$-moderate for some $p\in[1,\infty)$, then the space is called \emph{$p$-sub-tamed}. $(\bar P^{p\kappa/2}_t)_{t\ge0}$ will be called \emph{$p$-sub-taming semigroup} and $\cE^{p\kappa/2}$ will be called \emph{$p$-sub-taming energy form} for $(\X,\cE,\mm)$.
\end{definition}

Recall the remark after Definition \ref{def:L1Bochner} concerning the interpretation of the second integral in \eqref{eq:sub-L1Bochner}.

Note that the energy form in \eqref{eq:sub-L1Bochner} is defined as a perturbation of the reflected Dirichlet form $\bar\cE$ whereas the second term in \eqref{eq:sub-L1Bochner} is defined in terms of the generator $\sfL$ of the original Dirichlet form $\cE$.

More generally, we define $p$-versions of the Bochner inequality and the gradient estimate for $p\in[1,\infty)$ and also with additional dimension parameter.

\begin{definition}[$L^p$-Bochner inequality and gradient estimate]
Let $p\in[1,\infty)$, $N\in[1,\infty]$ and let $\kappa$ be a $p$-moderate distribution in $\bar\cF^{-1}_{\rm qloc}$.
\begin{itemize}
\item We say that the Bochner inequality $\overline\BE_p(\kappa,N)$ holds if for all $f \in \fD_{\cF}(\sfL)$ and all non-negative $\varphi \in \fD(\bar\sfL^{p\kappa/2})$
\begin{equation}\label{eq:LpBochner}
\int \bar\sfL^{p\kappa/2} \varphi \Gamma(f)^{p/2} \d\mm - p\int \varphi \Gamma(f,\sfL f)\Gamma(f)^{p/2-1}\d\mm \geq \frac{2p}{N}\int\varphi(\sfL f)^2\Gamma(f)^{p/2-1}\,\d\mm\;,
\end{equation}
where the right-hand side is read as $0$ if $N=\infty$.
\item We say that the gradient estimate $\overline\GE_p(\kappa,N)$ is satisfied if for any $f\in \cF$ and $t>0$
\begin{equation}\label{eq:Lpgradient}
\Gamma(P_t f)^{p/2} + \frac{2p}{N}\int_0^t \bar P_s^{p\kappa/2}\Big(\frac{(\sfL P_{t-s} f)^2}{\Gamma(P_{t-s} f)^{1-p/2}}\Big) \, \d s \leq \bar P_t^{p\kappa/2} \Gamma(f)^{1/2}\;.
\end{equation}
\end{itemize}
\end{definition}

Note that the gradient along the original heat flow $P_t$ is controlled via the taming semigroup $\bar P_t^{p\kappa/2}$ constructed from the reflected heat flow $\bar{P}_t$.

As before, we have equivalence of the Bochner inequality and the corresponding gradient estimate, and the $L^1$-version is the strongest one in this scale of estimates.

\begin{proposition}
Let $p\in[1,\infty)$, $N\in[1,\infty]$ and let $\kappa$ be a $p$-moderate distribution in $\bar\cF^{-1}_{\rm qloc}$. Then $\overline{\sf BE}_p(\kappa,N)$ and $\overline{\sf GE}_p(\kappa,N)$ are equivalent. Moreover, $\overline{\sf GE}_1(\kappa,N)$ implies $\overline{\sf GE}_p(\kappa,N)$.
\end{proposition}

\begin{proof}
The result is obtained following the argument for the proof of Theorem \ref{thm:BE1GE1} and Proposition \ref{GE1toGE2} with the obvious modifications.
\end{proof}

Similar as in Section \ref{sec:funct-ineq}, the sub-taming condition implies functional inequalities for the 

\begin{proposition}[Local (reverse) Poincar\'e inequality]\label{prop:sub-Poin}
Let $(\X,\cE,\mm)$ be a Dirichlet space sub-tamed by a $2$-moderate distribution $\kappa \in \cF^{-1}_{\rm qloc}$, i.e.~$\overline\GE_1(\kappa, \infty)$ holds. Then for any $f \in \cF$ and any $t > 0$ we have $\mm$-a.e.~on $\X$:
\begin{equation}\label{eq:sub-locPoi}
\underline C_t^\kappa \cdot \Gamma (P_t f)\leq \frac1{2t}\Big(\bar P_t(f^2) - (P_t f)^2\Big) \le \overline C_t^\kappa  \cdot \bar P_t\Gamma (f)\;,
\end{equation}
with $\overline C_t^\kappa:=\frac1t\int_0^t C_{s}^{\kappa}\,\d s$ and $\underline C_t^\kappa:=\frac1t\int_0^t \big(C_{s}^{\kappa}\big)^{-1}\,\d s$ where $C_t^{\kappa}$ is the time-depending constant defined in Section \ref{sec:Jensen}. Note that $\big(\underline C_t^\kappa\big)^{-1}\le \overline C_t^\kappa$ for all $t>0$ and $\limsup_{t\to0} \overline C_t^\kappa<\infty$. The first inequality in \eqref{eq:sub-locPoi} is valid for any $f\in L^2(\X,\mm)$.
\end{proposition}

\begin{proof}
Given $f,g\in\cF_b$ with $g\geq0$, we first observe that
    \begin{align*}
      \int \bar P_t g f^2 \, \d \mm - \int g (P_t f)^2 \, \d \mm = 2\int_0^t\int g \bar P_s\big(\sqrt{\Gamma(P_{t-s}f)}\big)^2\d\mm\d s \;.
    \end{align*}
From here, arguing as in the proof of Theorem \ref{thm:Poin} with appropriate modifications yields the desired result.
\end{proof}

\begin{proposition}[Local (reverse) log-Sobolev inequality]\label{prop:sub-logSob}
Let $(\X,\cE,\mm)$ be a Dirichlet space with a $2$-moderate distribution $\kappa \in \cF^{-1}_{\rm qloc}$ satisfying $\overline\GE_1(\kappa, \infty)$. Then for any $t > 0$ and for any $f \ge 0$ with the property that $\sqrt{f} \in \cF$ and $f \log(f) \in L^1(\mm)$, it holds $\mm$-a.e.~on $\X$:
\begin{equation}\label{eq:sub-loclogSob}
\int_0^t \dfrac{\Gamma(P_t f)}{\bar P^{\kappa/2}_sP_{t-s} f}\d s \leq \bar P_t(f \log f) - P_t f \log (P_t f) \le  \int_0^t \bar P_s\bar P^\kappa_{t-s}\bigg(\dfrac{\,\Gamma(f)\,}{f}\bigg)\d s\;,
\end{equation}
The first inequality holds more generally for all non-negative $f \in L^1(\X,\mm)$ with $f \log(f) \in L^1(\X,\mm)$.
\end{proposition}

\begin{proof}
The proof follows similar to the one of Theorem \ref{thm:logSob} with appropriate modifications, starting from the interpolation
\[
\Psi_\varepsilon(s) := \int \bar P_sg  \, \psi_\varepsilon (P_{t-s}f) \, \d \mm,
\]
with $\psi_\varepsilon (z) := (z + \varepsilon) \log(z+\varepsilon) - \varepsilon \log(\varepsilon)$ and for $g \in L^1\cap L^\infty(\mm)$, $g \ge 0$, and $f \in L^\infty(\X,\mm)$ such that $f \ge 0$, $\sqrt{f} \in \cF$, and $f \log (f) \in L^1(\X,\mm)$.
\end{proof}

\subsection{Doubling of Dirichlet Spaces and Sub-taming}
  
As before, let a strongly local, quasi-regular Dirichlet space $(\X,\cE,\mm)$ which admits a carr\'e du champ $\Gamma$ be given. In the following we will write $\cE_1(f) := \cE(f) + \| f \|_{L^2(\X,\mm)}$ for any $f \in \cF$.
  
Let $(\bar{\X}, \bar\cE, \bar\mm)$ denote the reflected Dirichlet space defined on some extension of $\X$ such that the latter is a quasi-open subset of $\bar{\X}$ with $\bar\mm (Z)=0$, where $Z:=\bar{\X} \setminus \X$.
  
Define the {\bf doubled space} $\hat{\X}$ by gluing two copies of $\bar{\X}$ along their common ``boundary'':
\[
\hat{\X} := \bar{\X} \times \{+,-\}\Big/\sim
\]
with $(x,\sigma)\sim(y,\tau)$ if and only if $(x,\sigma)=(y,\tau)$ or if $x=y\in Z$. Putting $\X^+ := \X\times\{+\}$ and $\X^- := \X\times\{-\}$, allows for a representation as a disjoint union
\[
\hat{\X} = \X^+\ \dot\cup \ \X^-\ \dot\cup\  Z
\]
in terms of the ``boundary'' $Z$ and two copies of $\X$. \\
  
We endow $\hat{\X}$ with the quotient topology: let us denote by 
\[
q \colon \bar{\X} \times \{+, -\} \, \to \, \hat{\X} = \bar{\X} \times\{+,-\}\Big/\sim
\]
the quotient map, then $\hat A \subset \hat{\X}$ is open if and only if $q^{-1}(\hat A)$ is an open subset of $\bar\X \times \{+, -\}$. It is worth to notice that for any $A \subset \bar\X$ it holds
\begin{equation}\label{eq:imq}
q^{-1}\big(q(A \times \{ +, - \})\big) = A \times \{ +, - \}.
\end{equation}
Let us also define a measure $\hat\mm$ on $\hat{\X}$ which coincides with $\mm/2$ on each of the copies $\X^\pm$ and which gives no mass to $Z$, namely
\[
\hat\mm(A) := \frac{1}{2}\mm(q^{-1}(A) \cap \X^+) + \frac{1}{2}\mm(q^{-1}(A) \cap \X^-)
\] 
for all Borel set $A \subset \hat\X$, where with a little abuse of notation $\mm$ is defined on $\X^\pm$ in the obvious way. Finally, given a function $f: \hat\X \to \R$,
   define functions $f^\pm: \bar\X \to \R$ by $f^+(x):=f(q(x,+))$ and $f^-(x):=f(q(x,-))$. (Note that these are not the positive and negative parts!) Then 
\[
f\in L^2(\hat\X,\hat \mm) \quad \Longleftrightarrow \quad f^+, f^-\in L^2(\bar\X,\bar \mm) \quad \Longleftrightarrow \quad f^+\big|_\X, f^-\big|_\X \in L^2(\X, \mm).
\]
   
\begin{definition} 
The doubled Dirichlet space $(\hat\X,\hat \cE, \hat\mm)$ is defined as a Dirichlet form on $L^2(\hat\X,\hat \mm)$ by $\hat \cF = \fD(\hat \cE) := \big\{f\in L^2(\hat\X,\hat \mm) \,:\, f^+ + f^-\in \bar\cF, \ f^+-f^-\in \cF \big\}$ and
\[
\hat \cE(f):=\bar\cE\Big(\frac{f^++f^-}2\Big)+\cE\Big(\frac{f^+-f^-}2\Big).
\]
\end{definition}

\begin{Proposition}\label{prop:hatF}
(i) $\hat \cF = \big\{f\in L^2(\hat\X,\hat \mm) \,:\, f^\pm\in \bar\cF, \ \tilde f^+=\tilde f^- \,\,\bar\cE\text{-q.e.~on }Z \big\}$ and
\begin{equation}\label{eq:gluedE} 
\hat \cE(f)=\frac12\bar\cE(f^+)+\frac12\bar\cE(f^-).
\end{equation} 
Here $\tilde f^\pm$ denote the quasi-continuous versions of the functions $f^\pm \in \bar\cF$. By polarization,
\begin{equation}\label{glued-df}
\hat \cE(f,g) = \frac12\bar\cE(f^+,g^+)+\frac12\bar\cE(f^-,g^-).
\end{equation}
   
(ii) Let $\{ \bar F_n \}_{n \in \N}$ be a sequence of subsets in $\bar\X$, and let $F_n = q (\bar F_n \times \{ +, - \}) \subset \hat\X$, for any $n \in \N$. Then $\{ \bar F_n \}_{n \in \N}$ is an $\bar \cE$-nest in $\bar\X$ if and only if $\{ F_n \}_{n \in \N}$ is an $\hat \cE$-nest in $\hat\X$.
   
(iii) A function $f \colon \hat\X \to \R$ is $\hat\cE$-quasi-continuous if and only if both the functions $f^+,f^-:\bar\X\to\R$ are $\bar\cE$-quasi-continuous and $f^+=f^-$ $\bar\cE$-q.e.\ on $Z$.
\end{Proposition}
   
\begin{proof} 
(i) Obviously, $\hat \cF \subset \big\{f\in L^2(\hat\X,\hat \mm)\,:\, f^\pm\in \bar\cF, \ \tilde f^+=\tilde f^-\,\bar\cE \text{-q.e.~on }Z \big\}$. To see the reverse inclusion, note that $f^+=f^-$ $\bar\cE$-q.e.\ on $Z$ for $f\in L^2(\hat\X,\hat \mm)$. Hence, $f^+-f^-\in\cF$ provided $f^\pm\in \bar\cF$. Moreover it holds $\cE (f^+ - f^-) = \bar \cE(f^+ - f^-)$, which in turns implies \eqref{eq:gluedE}. For more details we refer to \cite[Theorem 3.3.8]{Chen-Fuku}. 

\smallskip
     
(ii) First of all, let $\{ \bar F_n \}_{n \in \N}$ be an $\bar \cE$-nest in $\bar\X$. Directly from the definition of the quotient map and \eqref{eq:imq}, we get that $\{ F_n \}_{n \in \N}$ is an increasing sequence in $\hat\X$ made of closed set. Let us see that $\bigcup_{n \in \N} \hat \cF_{F_n}$ is $\hat \cE_1$-dense in $\hat \cF$. Since $f^{\pm} \in \bar \cF$ for any $f \in \hat \cF$, we can find two sequences $\{ \bar f_n^{+} \}_{n \in \N}, \{ \bar f_n^{-} \}_{n \in \N} \subset \bigcup_{n \in \N} \bar \cF_{\bar F_n}$ which are $\bar \cE_1$-converging to $f^+, f^-$, respectively. Hence, we consider the sequence in $\hat \cF$ given by $f_n (x, \sigma) := \bar f_n^\sigma (x)$, for $x \in \X$, $\sigma \in \{+, - \}$, and $f_n (z, \pm) = \mathds{1}_{Z \cap \bar F_n} f (z, \pm)$ for $z \in Z$. This sequence $\{ f_n \}_{n \in \N} \subset \bigcup_{n \in \N} \hat \cF_{F_n}$ is actually $\hat \cE_1$-converging to $f$: indeed, recalling that $\hat\X = \X^+ \ \dot\cup \ \X^- \ \dot\cup \ Z$ with $\hat \mm(Z) = 0$, by \eqref{eq:gluedE} it holds
\[
\begin{split}
2\hat \cE_1( f_n -  f) & = \bar \cE (\bar f^+_n - f^+) + \bar \cE (\bar f^-_n -  f^-) + \int_{\bar\X} \big( | \bar f^+_n - f^+|^2 + |\bar f^-_n - f^-|^2\big) \, \d \bar \mm\\
& = \bar \cE_1(\bar f^+_n -  f^+) + \bar \cE_1(\bar f^-_n - f^-).
\end{split}
\]
Viceversa, if $\{ F_n \}_{n \in \N}$ is an $\hat \cE$-nest in $\hat\X$, then also $\{ \bar F_n \}_{n \in \N}$ is an increasing sequence in $\bar\X$, while the definition of quotient topology guarantees that each $\bar F_n$ is actually closed. Now, starting from a function $\bar f \in \bar \cF$, we can define $f \in \hat \cF$ simply by posing $f(x, \pm) := \bar f(x)$ for any $x \in \bar\X$. Then the $\hat \cE_1$-density of $\bigcup_{n \in \N} \hat \cF_{F_n}$ in $\hat \cF$ allows to find a sequence $\{ f_n \}_{n \in  \N} \subset \bigcup_{n \in \N} \hat \cF_{F_n}$ $\hat \cE_1$-converging to $f$, and $\{f_n^{+}\} \subset \bigcup_{n \in \N} \bar \cF_{\bar F_n}$ provides a sequence $\bar \cE_1$-converging to $\bar f$.

\smallskip
     
(iii) Let $f$ be an $\hat\cE$-quasi-continuous function on $\hat\X$ and let $\{ F_n \}_{n \in \N}$ be an $\hat \cE$-nest such that $f|_{F_n}$ is finite and continuous on $F_n$ for each $n \in \N$. Since the quotient map $q$ is surjective, there exists a sequence $\{ \bar F_n  \}_{n \in \N} \subset \bar\X$ such that $F_n = q(\bar F_n \times \{ +, - \})$ for any $n \in \N$, and $\{ \bar F_n \}_{n \in \N}$ is an $\bar \cE$-nest in $\bar\X$. Thus, from the fact that $f \in \hat \cF$, it follows that $f^{\pm} \in \bar \cF$ with $f^+ = f^-$ $\bar\cE$-q.e.\ on $Z$, and that $f^{\pm}|_{\bar F_n}$ are finite and continuous on $\bar F_n$ for each $n \in \N$.
     
Conversely, if a function $f$ on $\hat\X$ is such that $f^{\pm}$ are $\bar \cE$-quasi-continuous, then there exist two $\bar \cE$-nests $\{\bar F^+_n\}_{n \in \N}$ and $\{\bar F^-_n\}_{n \in \N}$ such that $f^+|_{\bar F^+_n}$ and $f^-|_{\bar F^-_n}$ are finite and continuous on $\bar F^+_n$ and $\bar F^-_n$, respectively. Now, using the fact that the refined sequence $\{ \bar F^+_j \cap \bar F^-_k \}_{j, k \in \N}$ is still a nest on $\bar\X$, we have that the sequence $\{ F_{j, k} \}_{j, k \in \N} $, where $F_{j, k} := q\big( (\bar F^+_j \cap \bar F^-_k) \times\{ +, -  \} \big)$, is an $\hat \cE$-nest on $\hat\X$ such that $f|_{F_{j, k}}$ is finite and continuous on each $F_{j, k}$, by the very definition of $f^\pm$ and the fact that $f^+ = f^-$ $\bar\cE$-q.e.\ on $Z$.
\end{proof}

\begin{Lemma} 
$(\hat\X,\hat \cE, \hat\mm)$ is a strongly local, quasi-regular Dirichlet space and it admits a carr\'e du champ which will also be denoted by $\Gamma$.
\end{Lemma}

\begin{proof} 
The strong locality of $\hat\cE$ follows from \eqref{glued-df} and the strong locality of $\bar\cE$. Also the existence of a carr\'e du champ can be concluded from \eqref{glued-df}. In the following we show the quasi-regularity of $\hat\cE$, giving a detailed proof of properties (i)-(iii) in \cite[Definition 1.3.8]{Chen-Fuku}.

\smallskip
  
(i) Let $\big\{ \bar F_n  \big\}_{n \in \N}$ be an $\bar \cE$-nest in $\bar\X$ made of compact sets and put $F_n := q\big(\bar F_n \times \{+, - \}\big) \subset \hat\X$. Hence, (ii) in Proposition \ref{prop:hatF} ensures that $\{ F_n\}_{n \in \N}$ is an $\hat \cE$-nest in $\hat\X$. Moreover, each $\hat F_n$ is compact, being the image through a quotient map of a compact set. 

\smallskip

(ii) Denote by $\bar {\mathcal D}$ the $\bar \cE_1$-dense subset of $\bar \cF$, whose elements have $\bar \cE$-quasi-continuous $\bar \mm$-version, and define $\hat{\mathcal D} := \{  f \in L^2(\hat\X, \hat \mm) \,:\, f^{\pm} \in \bar {\mathcal{D}}, \, \tilde f^+ =  \tilde f^-\, \bar\cE\text{-q.e. on }\, Z \} \subset \hat \cF$. Directly from (iii) in Proposition \ref{prop:hatF} we know that every element in $\hat {\mathcal D}$ has an $\hat \cE$-quasi-continuous $\hat \mm$-version. Now, let $ f \in \hat \cF$ be fixed. Since $f^{\pm} \in \bar \cF$, the quasi-regularity of $\bar \cE$ guarantees the existence of two sequences $\{ \bar f_n^+ \}_{n \in \N}, \{\bar f_n^- \}_{n \in \N} \subset \bar {\mathcal D} \,$  $\bar \cE_1$-converging to $ f^+$ and $ f^-$, respectively. Therefore, for any $n \in \N$, we define $f_n \colon \hat\X \to \R$ by setting $f_n(x,\sigma) := \bar f_n^{\sigma}(x)$ for $x \in \X$, $\sigma \in \{+, - \}$, and $ f_n(x,\pm) := f(x,\pm)$, for $x \in Z$. Hence it holds $f_n^{+} = \bar f_n^{+}, f_n^{-} = \bar f_n^{-} \in \bar {\mathcal D}$ and $f_n^+ = f_n^-$ $\bar\cE$-q.e.\ on $Z$, showing that $ f_n \in \hat {\mathcal D}$. Arguing as in the proof of (ii) in Proposition \ref{prop:hatF}, we can conclude that $\hat \cE_1(\hat f_n - \hat f) \to 0$ as $n \to \infty$, and this implies the $\hat \cE_1$-density of $\hat {\mathcal D}$ in $\hat \cF$.

\smallskip

(iii) Let $\{ \bar f_n  \}_{n \in \N} \subset \bar \cF$ be a sequence whose elements have an $\bar \cE$-quasi-continuous $\bar \mm$-version, $\bar f_n^{\sim}$, and let $\bar N \subset \bar\X$ such that $\{ {\bar f}_n^{\sim} \}_{n \in \N}$ separates the points of $\bar\X \setminus \bar N$. The fact that the quotient map is surjective together with Proposition \ref{prop:hatF} ensures that $q(\bar N \times \{+, - \}) \subset \hat\X $ is an $\hat \cE$-polar set. Thus we define the sequence $f_n \colon \hat\X \to \R$ by setting $f_n(x, \sigma) := \sigma f_n(x)$ for $x \in \X$, $\sigma \in \{+,-\}$ and $f_n(x,\pm) := \bar f_n (x)$ for $x \in Z$. In particular, $f_n^{\pm} = \pm f_n \in \bar \cF$ and $\tilde f^+ = \tilde f^-$ $\bar\cE$-q.e.\ on $Z$, ensuring  that $\{ f_n \}_{n \in \N} \subset \hat \cF$. At this point, Proposition \ref{prop:hatF} grants that each $f_n$ has an $\hat \cE$-quasi-continuous $\hat \mm$-representative, $\tilde f_n$. The only thing left to prove is the fact that $\{ \tilde f_n  \}_{n \in \N} \subset \hat \cF$ separates points in $\hat\X \setminus q(\bar N \times \{+, - \})$. Let $(x, \sigma)$, $(y, \tau) \in \hat\X \setminus q(\bar N \times \{+, - \})$ be any couple of distinct points: in the case in which $x \neq y$, the existence of $f_n$ such that $\tilde f_n (x, \sigma) \neq \tilde f_n (y, \tau)$ is ensured by the fact that $\{ \bar f_n  \}_{n \in \N}$ separates the points of $\bar\X \setminus \bar N$, while if $x = y$ we have $f_n (x, +) = - f_n (x, -)$ for any $ f_n$ such that $f_n (x, \sigma) \neq 0$.
\end{proof}

\begin{Lemma} 
The strongly continuous semigroup $(\hat T_t)_{t\ge0}$ for the doubled Dirichlet space $(\hat\X,\hat \cE, \hat\mm)$ is given in terms of the reflected semigroup $(\bar T_t)_{t\ge0}$ and the original (``absorbed'') semigroup $(\bar T_t)_{t\ge0}$ as
\begin{equation}\label{eq:semigroups coincide}
\hat T_tf(q(x,\pm)) = \bar T_t\Big(\frac{f^++f^-}2\Big)(x) + T_t\Big(\pm\frac{f^+-f^-}2\Big)(x)
\end{equation}
for $f\in  L^2(\hat\X,\hat \mm)$, where $q$ denotes the quotient map. 
  
Conversely, for $h \in L^2(\bar\X,\bar \mm)$ and $g \in L^2(\X,\mm)$, 
\begin{equation}\label{heat-proj}
\bar T_th(x) = \hat T_t\hat h(q(x,\pm)), \quad x \in \bar\X; \qquad  T_tg(x) = \pm\hat T_t\check g(q(x,\pm)), \quad x \in \X
\end{equation}
with (symmetric and anti-symmetric, resp.) extensions $\hat h, \check g \in L^2(\hat\X,\hat \mm)$ given by $\hat h(x,\pm) := h(x)$ for $x\in\bar\X$ and $\check g(x,\pm) := \pm g(x)$ for $x\in \X$, $\check g(x,\pm):= 0$ for $x\in Z$.
\end{Lemma}

\begin{proof}
For gluing of metric measure spaces, this was proven in \cite[Theorem 3.10]{ProfetaSturm18}. Our setting here is slightly more general but the same arguments apply. Indeed, let us first note that $T_t(f^+-f^-) = 0$ $\bar\cE$-q.e.\ on $\bar\X \setminus \X$, hence \eqref{eq:semigroups coincide} is meaningful. Secondly, for \eqref{eq:semigroups coincide} to hold it is sufficient to check that the Dirichlet forms associated to $\hat T_t$ and $T_t'$ defined on $L^2(\hat\X,\hat\mm)$ by
\[
T_t' f(q(x,\pm)) = \bar T_t\Big(\frac{f^++f^-}{2}\Big)(x) + T_t\Big(\pm\frac{f^+-f^-}2\Big)(x)
\]
(which is well defined thanks to what previously said) coincide, i.e.\ $(\hat\cE,\hat\cF) = (\cE',\fD(\cE'))$, where
\[
\begin{split}
\fD(\cE') & := \Big\{f \in L^2(\hat\X,\hat\mm) \,:\, \exists \lim_{t \downarrow 0} -\frac{1}{t}\int_{\hat\X} f \big(T_t'f - f\big)\,\d\hat\mm < \infty \Big\}, \\
\cE'(f) & := \lim_{t \downarrow 0} -\frac{1}{t}\int_{\hat\X} f \big(T_t'f - f\big)\,\d\hat\mm.
\end{split}
\]
To this aim we shall use the following identity, which can be readily verified following the algebraic manipulations in \cite[Lemma 3.8]{ProfetaSturm18}: 
\begin{equation}\label{eq:approximated forms}
\begin{split}
-\frac{1}{t}\int_{\hat\X} f \big(T_t'f - f\big)\,\d\hat\mm & = -\frac{1}{t}\int_{\bar\X}\frac{f^++f^-}{2}\Big(\bar{T}_t\Big(\frac{f^++f^-}{2}\Big) - \frac{f^++f^-}{2}\Big)\,\d\bar\mm \\
& \quad - \frac{1}{t}\int_\X \frac{f^+-f^-}{2}\Big(T_t\Big(\frac{f^+-f^-}{2}\Big) - \frac{f^+-f^-}{2}\Big)\,\d\mm.
\end{split}
\end{equation}
If $f \in \fD(\cE')$, then by taking the limit as $t \downarrow 0$ in the identity above and interchanging limit and sum on the right-hand side (this is possible as $t \mapsto -\frac{1}{t}\int_{\hat\X} f \big(T_t'f - f\big)\,\d\hat\mm$ is non-increasing and non-negative and the same is true with $T_t$, $\bar{T}_t$ in place of $T_t'$) we obtain
\[
\bar \cE\Big(\frac{f^++f^-}2\Big) + \cE\Big(\frac{f^+-f^-}2\Big) = \cE'(f) < \infty,
\]
whence $f^++f^- \in \bar\cF$ and $f^+-f^- \in \cF$, namely $f \in \hat\cF$. Moreover $\cE'(f) = \hat\cE(f)$ by the identity above and the very definition of $\hat\cE$. On the other hand, if $f \in \hat\cF$ then $f^++f^- \in \bar\cF$ and $f^+-f^- \in \cF$ by construction and this means that the limits as $t \downarrow 0$ of both summands on the right-hand side of \eqref{eq:approximated forms} exist, thus the limit of the left-hand side too, which means $f \in \fD(\cE')$ and again $\cE'(f) = \hat\cE(f)$.

Finally, the validity of \eqref{heat-proj} is straightforward by construction.
\end{proof}

\begin{Corollary}\label{mp-proj}
Let $(\hat{\mathbb P}_x, \hat B_t)_{x\in \hat\X, t\ge0}$ denote the unique (in the sense of \cite{Chen-Fuku}) $\hat\mm$-reversible, continuous, strong Markov process (with life time $\hat\zeta$) properly associated with the doubled Dirichlet space $(\hat\X, \hat\cE, \hat\mm)$ and define $\pi: \hat\X \to \bar\X$ as $\pi(q(x,\pm)) := x$. Then the process $(\bar{\mathbb{P}}_x, \bar{B}_t)_{x \in \bar\X, t\ge0}$ given by
\[
\bar{\mathbb{P}}_x := \frac{1}{2}\hat{\mathbb P}_{q(x,+)} + \frac{1}{2}\hat{\mathbb P}_{q(x,-)}, \qquad \bar B_t := \pi(\hat B_t)
\]
is $\bar\mm$-reversible, continuous, strong Markov and properly associated with the reflected Dirichlet space $(\bar\X, \bar\cE, \bar\mm)$.
\end{Corollary}

\begin{proof}
By construction, for all Borel function $h \in L^2(\bar\X,\bar\mm)$ and for all $x \in \bar\X$ it holds
\[
\begin{split}
\bar{\mathbb E}_x\big[h(\bar B_t)\big] & = \frac{1}{2}\hat{\mathbb E}_{q(x,+)}\big[h \circ \pi (\hat{B}_t)\big] + \frac{1}{2}\hat{\mathbb E}_{q(x,-)}\big[h \circ \pi (\hat{B}_t)\big] \\
& = \frac{1}{2}\hat P_t(h\circ\pi)(q(x,+)) + \frac{1}{2}\hat P_t(h\circ\pi)(q(x,-)),
\end{split}
\]
where $\hat P_t$ denotes the semigroup induced by $(\hat{\mathbb P}_x, \hat B_t)_{x\in \hat\X, t\ge0}$. This is related to the semigroup $\hat T_t$ associated with $\hat\cE$ via $\hat P_t (h\circ\pi) = (\hat T_t (h\circ\pi))^\sim$ $\hat\mm$-a.e.\ where $(\hat T_t (h\circ\pi))^\sim$ is an $\hat\cE$-quasi-continuous $\hat\mm$-version of $\hat{T}_t(h \circ \pi)$ (recall \eqref{eq:properly associated semigroup}). Hence
\[
\bar{\mathbb E}_x\big[h(\bar B_t)\big] = \frac{1}{2}(\hat T_t (h\circ\pi))^\sim(q(x,+)) + \frac{1}{2}(\hat T_t (h\circ\pi))^\sim(q(x,-)), \qquad \textrm{for }\bar\mm\textrm{-a.e. } x \in \bar\X.
\]
Now observe that since $(\hat T_t (h\circ\pi))^\sim$ is $\hat\cE$-quasi-continuous, by Proposition \ref{prop:hatF}-(iii) we know that $((\hat T_t (h\circ\pi))^\sim)^\pm = (\hat T_t (h\circ\pi))^\sim(q(\cdot,\pm))$ are $\bar\cE$-quasi-continuous and by \eqref{heat-proj} this yields
\begin{equation}\label{eq:highway star}
\bar{\mathbb E}_x\big[h(\bar B_t)\big] = (\bar{T}_t h)^\sim(x) = \bar P_t h(x), \qquad \textrm{for }\bar\mm\textrm{-a.e. } x \in \bar\X,
\end{equation}
where $\bar P_t$ denotes the semigroup induced by the Markov process properly associated with $(\bar\X, \bar\cE, \bar\mm)$. As such a process is $\bar\mm$-reversible, continuous and strong Markov, the same holds for $(\bar{\mathbb{P}}_x, \bar{B}_t)_{x \in \bar\X, t\ge0}$.

In particular, the Markov property is inherited for the following reason. Let $(\mathscr{F}_t)_{t \geq 0}$ be a filtration w.r.t.\ which the strong Markov property holds for $(\hat{\mathbb P}_x, \hat B_t)_{x\in \hat\X, t\ge0}$. Then $(\mathscr{F}_t)_{t \geq 0}$ is admissible for $(\bar B_t)_{x\in \bar\X, t\ge0}$, since for any Borel set $A \subset \bar\X$
\[
\bar{B}_t^{-1}(A) = \hat{B}_t^{-1}(\pi^{-1}(A)) = \hat{B}_t^{-1}(q(A \times \{+,-\})) \in \mathscr{F}_t,
\]
and it is easy to see that
\[
\bar{\mathbb{P}}_x(\bar{B}_t \in \cdot \,|\, \mathscr{F}_t) = \frac{1}{2}\hat{\mathbb{P}}_{q(x,+)}(\hat{B}_t \in \pi^{-1}(\,\cdot\,) \,|\, \mathscr{F}_t) + \frac{1}{2}\hat{\mathbb{P}}_{q(x,-)}(\hat{B}_t \in \pi^{-1}(\,\cdot\,) \,|\, \mathscr{F}_t)
\]
is the conditional distribution for $\bar{B}_t$ given $\mathscr{F}_t$. For any $(\mathscr{F}_t)_{t \geq 0}$-stopping time $\sigma$, the strong Markov property applies to the right-hand side above, whence $\hat{\mathbb{P}}_{q(x,\pm)}(\hat{B}_{t+\sigma} \in \pi^{-1}(\,\cdot\,) \,|\, \mathscr{F}_\sigma) = \hat{\mathbb{P}}_{\hat{B}_\sigma}(\hat{B}_t \in \pi^{-1}(\,\cdot\,))$ $\hat{\mathbb{P}}_{q(x,\pm)}$-a.s. As an involution $\iota : \hat\X \to \hat\X$, $\iota(q(x,\pm)) := q(x,\mp)$, is naturally associated with $\hat\X$ and $\pi \circ \iota = \pi$, note that
\[
\begin{split}
& \hat{\mathbb{P}}_{q(x,+)}(\hat{B}_t \in \pi^{-1}(\,\cdot\,) \,|\, \mathscr{F}_t) = \hat{\mathbb{P}}_{q(x,-)}(\hat{B}_t \in \pi^{-1}(\,\cdot\,) \,|\, \mathscr{F}_t) = \hat{\mathbb{P}}_{q(x,+)}(\iota(\hat{B}_t) \in \pi^{-1}(\,\cdot\,) \,|\, \mathscr{F}_t), \\
& \hat{\mathbb{P}}_{\hat{B}_\sigma}(\hat{B}_t \in \pi^{-1}(\,\cdot\,)) = \hat{\mathbb{P}}_{\iota(\hat{B}_\sigma)}(\hat{B}_t \in \pi^{-1}(\,\cdot\,)) = \hat{\mathbb{P}}_{\hat{B}_\sigma}(\iota(\hat{B}_t) \in \pi^{-1}(\,\cdot\,)).
\end{split}
\]
This means that $\hat{\mathbb{P}}_{q(x,+)}(\hat{B}_{t+\sigma} \in \pi^{-1}(\,\cdot\,) \,|\, \mathscr{F}_\sigma) = \hat{\mathbb{P}}_{\hat{B}_\sigma}(\hat{B}_t \in \pi^{-1}(\,\cdot\,))$ holds both $\hat{\mathbb{P}}_{q(x,+)}$-a.s.\ and $\hat{\mathbb{P}}_{q(x,-)}$-a.s.\ and thus a fortiori $\bar{\mathbb{P}}_x$-a.s. For the same reason and using the second identity on both lines above we also have $\hat{\mathbb{P}}_{q(x,-)}(\hat{B}_{t+\sigma} \in \pi^{-1}(\,\cdot\,) \,|\, \mathscr{F}_\sigma) = \hat{\mathbb{P}}_{\iota(\hat{B}_\sigma)}(\hat{B}_t \in \pi^{-1}(\,\cdot\,))$ $\bar{\mathbb{P}}_x$-a.s. Hence we deduce that
\[
\bar{\mathbb{P}}_x(\bar{B}_{t+\sigma} \in \cdot \,|\, \mathscr{F}_\sigma) = \frac{1}{2}\hat{\mathbb{P}}_{\hat{B}_\sigma}(\hat{B}_t \in \pi^{-1}(\,\cdot\,)) + \frac{1}{2}\hat{\mathbb{P}}_{\iota(\hat{B}_\sigma)}(\hat{B}_t \in \pi^{-1}(\,\cdot\,)) = \bar{\mathbb{P}}_{\bar{B}_\sigma}(\bar{B}_t \in \cdot\,)
\]
holds true $\bar{\mathbb{P}}_x$-a.s.

Finally, the fact that $(\bar{\mathbb{P}}_x, \bar{B}_t)_{x \in \bar\X, t\ge0}$ is properly associated with $(\bar\X, \bar\cE, \bar\mm)$, i.e.\ that $\bar{\mathbb E}_x\big[h(\bar B_t)\big]$ is an $\bar\cE$-quasi-continuous $\bar\mm$-version of $\bar{T}_t h$, is a consequence of the first identity in \eqref{eq:highway star}.
\end{proof}

\begin{Example} 
Let a metric measure space $(\X,\sfd,\mm)$ be given and a dense open subset $\Y \subset \X$ with $\mm(Z)=0$ where $Z := \X \setminus \Y$. Define the doubled space $\hat\X = \Y^+\ \dot\cup \ \Y^-\ \dot\cup\ Z$ as before (now with $\X$ and $\Y$ in the place of $\bar\X$ and $\X$, resp.) by gluing two copies of $\X$ along their common boundary. Define a metric on $\hat\X$ by
\[
\hat\sfd(x,y):=\inf_{z\in Z}\Big[\sfd(x,z)+ \sfd(z,y)\Big]
\]
if $(x,y)\in \big(\Y^+\times \Y^-\big)\cup \big(\Y^-\times \Y^+\big)$ and $\hat\sfd(x,y):=\sfd(x,y)$ if $(x,y) \in \big(\X^+\times \X^+\big)\cup \big(\X^-\times \X^-\big)$. Moreover, define as before a measure $\hat\mm$ on $\hat\X$ which coincides with $\mm$ on each of the copies $\Y^\pm$ and which gives no mass to $Z$.
  
If the mm-space $(\X,\sfd,\mm)$ gives rise to the Dirichlet space $(\X,\cE,\mm)$, then the doubled mm-space $(\hat\X,\hat\sfd,\hat\mm)$ gives rise to the doubled Dirichlet space $(\hat\X,\hat\cE, \hat\mm)$, as shown in \cite[Lemma 3.3]{ProfetaSturm18}.
\end{Example}

\bigskip

Now let us have a closer look on distributions on the doubled space.

\begin{lemma}\label{ext-kappa}
(i) Each $\kappa\in \bar\cF^{-1}$ defines in a canonical way a distribution $\hat\kappa\in \hat\cF^{-1}$ by
\[
\hat\kappa=(-\hat \sfL+1)\hat\psi \quad \text{where}\quad \hat\psi(x,\pm):=\psi(x), \ \psi:=(-\bar \sfL+1)^{-1}\kappa.
\]

(ii) Each quasi-open nest $(G_n)_n$ in $\bar\X$ defines in a canonical way a quasi-open nest $(\hat G_n)_n$ in $\hat\X$ by
\[
\hat G_n=G_n^+\cup G_n^-, \quad G_n^\pm:=G_n\times\{\pm\}.
\]

(iii) Each $\kappa\in \bar\cF^{-1}_{\rm qloc}$ defines in a canonical way a $\hat\kappa\in \hat\cF^{-1}_{\rm qloc}$. Given a quasi-open nest $(G_n)_n$ in $\bar\X$ and a distribution $\kappa\in\bigcap_n\bar\cF^{-1}_{G_n}$, define $\hat\kappa\in\bigcap_n\hat\cF^{-1}_{\hat G_n}$ by
\[
\hat\kappa = (-\hat \sfL_{\hat G_n}+1)\hat\psi_n\quad\text{where}\quad \hat\psi_n(x,\pm):=\psi_n(x), \ \psi_n:=(-\bar \sfL_{G_n}+1)^{-1}\kappa.
\]
\end{lemma}

\begin{proof}
(i) For $f\in\hat\cF$ and $\kappa\in \bar\cF^{-1}$ with $\kappa$ as above and with $\bar f:=\frac12(f^++f^-)$,
\[
\langle \hat\kappa,f\rangle=\hat\cE_1\big(\hat\psi,f\big)=\bar\cE_1\big(\psi,\bar f\big)=\langle \kappa,\bar f\rangle.
\]
(ii), (iii) straightforward.
\end{proof}

\begin{Lemma}\label{caf-proj}
Let $\kappa\in \bar\cF^{-1}_{\rm qloc}$ be given and let $(\bar A_t^\kappa)_t$ denote the local continuous additive functional associated with it for the Markov process $(\bar{\mathbb P}_x, \bar B_t)_{x\in \bar\X, t\ge0}$ on $\bar X$ obtained by projection of the Markov process $(\hat{\mathbb P}_x, \hat B_t)_{x\in \hat\X, t\ge0}$ on $\hat X$ as in Corollary \ref{mp-proj}.
  
Then $(\bar A_t^\kappa)_t$ coincides with the local continuous additive functional associated with $\hat\kappa\in \hat\cF^{-1}_{\rm qloc}$, the canonical extension of $\kappa$ onto $\hat\X$ as considered in the previous Lemma \ref{ext-kappa}.
\end{Lemma}

\begin{proof} 
To simplify the presentation, we restrict ourselves to the case $\kappa\in \bar\cF^{-1}$. The extension to the general case will follow by straightforward approximation arguments. Then the associated CAF $(\bar A_t^\kappa)_t$ is characterized by the identity
\[
\langle \kappa, \bar G_1\varphi\rangle=\hat{\mathbb E}_{\varphi\,\bar\mm}\Big[\int_0^\infty e^{-t}\d\bar A_t \Big]\qquad \big(\forall \varphi\in L^2(\bar\X,\bar\mm)\big).
\]
An analogous characterization holds for the CAF $(\hat A_t)_t$ associated with $\hat\kappa\in \hat\cF^{-1}$. Thus for all $\Phi \in L^2(\hat\X,\hat\mm)$ and with $\bar\Phi:=\frac12(\Phi^++\Phi^-) \in L^2(\bar\X,\bar\mm)$,
\[
\hat{\mathbb E}_{\Phi\,\hat\mm}\Big[\int_0^\infty e^{-t}\d\hat A_t \Big] = \langle \hat\kappa, \hat G_1\Phi\rangle=\langle \kappa, \bar G_1\bar\Phi\rangle = \hat{\mathbb E}_{\bar\Phi\,\bar\mm}\Big[\int_0^\infty e^{-t}\d\bar A_t \Big] = \hat{\mathbb E}_{\Phi\,\hat\mm}\Big[\int_0^\infty e^{-t}\d\bar A_t\Big]
\]
since $\langle \hat\kappa, \varphi\rangle=\langle \kappa, \bar\varphi\rangle$ and $\overline{\hat G_1\Phi}=\bar G_1\bar\Phi$ by \eqref{heat-proj}. This proves that $\bar A=\hat A$ (up to equivalence of CAFs).
   \end{proof}

\begin{Lemma} 
For $h\in  L^2(\bar\X,\bar \mm)$, \begin{equation}\label{schrod-proj}
\bar P^\kappa_th=\hat P^{\hat \kappa}_t (h\circ \pi)
\end{equation}
and
\begin{equation*}
\bar \cE^\kappa(h)=\hat\cE^{\hat \kappa} (h\circ \pi).
\end{equation*}
\end{Lemma}

\begin{proof}
The first assertion follows from Corollary \ref{mp-proj} and Lemma \ref{caf-proj}. The second assertion is a direct consequence of the first one since both of the quadratic forms are generated by the respective semigroups.
\end{proof}

\begin{Theorem} 
Let the Dirichlet space $(\X,\cE,\mm)$ and the moderate distribution $\kappa\in\bar\cF^{-1}_{\rm qloc}$ be given. Extend the latter to $\hat\kappa\in\hat\cF^{-1}_{\rm qloc}$, and assume that the doubled Dirichlet space $(\hat\X,\hat \cE, \hat\mm)$ is tamed with synthetic Ricci bound $\hat\kappa$. Then the original Dirichlet space $(\X,\cE,\mm)$ is {sub-tamed} with synthetic Ricci bound $\kappa$.
  
In other words,
\[
{\sf BE}_1(\hat\kappa,\infty) \text{ for }(\hat\X,\hat\cE,\hat\mm) 
\qquad\Longrightarrow\qquad
\overline{\sf BE}_1(\kappa,\infty) \text{ for }(\X,\cE,\mm).
\]
\end{Theorem}
  
\begin{proof} 
For each $f\in \fD(\cE)$,
  \begin{eqnarray*}
  \Gamma\big(P_tf\big)^{1/2}= \Gamma\big(\hat P_t\check f\big)^{1/2}\le \hat P_t^{\hat\kappa}\big(\Gamma(\check f)^{1/2}\big)=\hat P_t^{\hat\kappa}\big({\Gamma(f)^{1/2}}\circ\pi\big)=
  \bar P_t^{\kappa}\big(\Gamma(f)^{1/2}\big)
  \end{eqnarray*}
with the first identity due to \eqref{heat-proj}, the last one due to \eqref{schrod-proj}, and the inequality due to the taming property of the doubled space. Moreover, the equality $\Gamma(\check f)^{1/2}= {\Gamma(f)^{1/2}}\circ\pi$ follows from the locality of $\hat\cE$.
\end{proof}
  
  \subsection{Doubling of Riemannian Surfaces}
  Let $(\sfM,\sfg)$ be a compact 2-dimensional Riemannian manifold
  with boundary and denote by $\hat\sfM$ the doubling of $\sfM$
  (i.e.~the gluing of two copies of $\sfM$ along their common
  boundary). That is,
  $$\hat\sfM=\sfM_0^+\ \dot\cup \ \sfM_0^-\ \dot\cup \ \partial\sfM$$
  where $\sfM_0^+$ and $\sfM_0^-$ denote two copies of   the interior of $\sfM$.  Let $\mm$ and $\hat\mm$ denote the volume
  measures on $\sfM$ and $\hat\sfM$, resp., and let $\sigma$ denote
  the surface measure on $\partial\sfM$ (regarded both as a subset of
  $\sfM$ and as subset of $\hat\sfM$).
 
  Let $\big(\sfM,\cE_\sfM,\mm\big)$ be the canonical Dirichlet space
  on $\sfM$ (with ``Neumann boundary conditions'') with
  $\cE_\sfM(f):=\frac12\int_{\sfM_0}\nabla f|^2\,d\mm$ and
  $\fD(\cE_\sfM):=W^{1,2}(\sfM_0)$. Moreover, let
  $\big(\sfM_0,\cE^0_\sfM,\mm\big)$ denote its restriction (with
  ``Dirichlet boundary conditions'') onto $\sfM_0$,
  i.e.~$\cE_\sfM^0:=\cE_\sfM$ on $\fD(\cE^0_\sfM):=W_0^{1,2}(\sfM_0)$.
   
  The doubled Dirichlet space
  $\big(\hat\sfM,\hat\cE_\sfM,\hat\mm\big)$ on $\hat\sfM$ is given by
   \begin{equation*}
  \hat \cE_\sfM(f):=\cE_\sfM\Big(\frac{f^++f^-}2\Big)+\cE^0_\sfM\Big(\frac{f^+-f^-}2\Big)
   \end{equation*}
   with $\fD(\hat \cE_\sfM):=\big\{f\in L^2(\hat X,\hat \mm): \ f^++f^-\in W^{1,2}(\sfM_0), \  f^+-f^-\in W^{1,2}_0(\sfM_0)\big\}$.
   
   \begin{Theorem} Assume that $k\in{\mathcal C}(\sfM)$ is a pointwise
     lower bound for the Ricci curvature on $\sfM_0$ and that
     $\ell\in{\mathcal C}(\partial\sfM)$ is a pointwise lower bound
     for the second fundamental form on $\partial\sfM$.  Then the
     Dirichlet space $\big(\hat\sfM,\hat\cE_\sfM,\hat\mm\big)$
     satisfies ${\sf BE}_1(\kappa,\infty)$ with
  $$\kappa:=\hat k\,\hat\mm+ \ell\,\sigma$$
where $\hat k:=k\circ\pi$  and $\pi:\hat\sfM\to\sfM$ denotes the projection.
  \end{Theorem}
  
  Here as usual --- if not explicitly specified otherwise --- the
  manifold $\sfM$ and its boundary $\partial\sfM$ are assumed to be
  smooth (at least $\mathcal C^2$).
  
  \begin{proof} The first parts of our argumentation apply to manifols
    of arbitray dimension $n$. Only in the last step $n=2$ is
    requested.
  
    (i) Given $\epsilon>0$, choose
    $\ell_\epsilon\in{\mathcal C}^2(\sfM)$ and
    $V\in{\mathcal C}^2(\sfM)$ with
    $\ell\ge\ell_\epsilon\ge\ell-\epsilon$ on $\partial\sfM$ and
    $V=-\sfd(.,\partial\sfM)$ on
    $B_\epsilon(\partial\sfM):=\{x\in\sfM \,:\, 
    \sfd(x,\partial\sfM)<\epsilon\}$. (The existence of such
    $\ell_\epsilon$ is obvious; the existence of such $V$ follows from
    the fact that $d(.,\partial\sfM)$ itself is smooth in
    $B_\epsilon(\partial\sfM)$ for sufficiently small $\epsilon>0$
    where ``smallness'' is in terms of bounds for the second
    fundamental form on $\partial\sfM$ and for the sectional curvature
    on $\sfM$, cf. e.g. \cite{Wang14}, (A3.2.1) and related
    construction.)
 
    (ii) Now consider the conformally transformed Riemannian manifold
    $(\sfM, \sfg')$ with $\sfg'=e^{-2\psi}\sfg$ where
    $\psi=\psi_\epsilon=(\epsilon-\ell_\epsilon)\, V$. (To improve
    readability, we will suppress the dependency (of $\psi$, $\sfg'$,
    and $k'$) on $\epsilon$ here and below.) Again, this is a
    Riemannian manifold with boundary but now the boundary is convex
    (according to general abuse of notations; the precise meaning is
    that $\sfM_0$ is convex and/or that $\sfM$ can be regarded as a
    convex subset of an ambient space).  Indeed, this follows from
    Theorem 5.16 in \cite{Sturm2019} where instead of a conformal
    transformation a time change was considered which leads to the
    same transformed distance function and hence to the same convexity
    notion. The Ricci curvature of the transformed manifold
    $(\sfM, \sfg')$ is bounded from below (see \cite{Sturm14}) by
\begin{equation*}k':=e^{-2\psi}\Big[k-\Delta\psi-(n-2)\nabla\psi|^2+(n-2)\inf_{u\in {\mathcal C}^1(\sfM)}\frac1{|\nabla u|^2}\Big( -\nabla^2\psi(\nabla u,\nabla u)+\langle\nabla\psi,\nabla u\rangle^2\Big)
\Big].
\end{equation*}

(iii) Next let us consider $(\hat\sfM, \hat\sfg')$, the doubling of $(\sfM, \sfg')$, and the
associated Dirichlet space $\big(\hat\sfM,\hat\cE'_\sfM,\hat\mm'\big)$
with
$$\hat\cE'_\sfM(f)=\frac12\int_{\hat\sfM}\hat\Gamma(f)\,e^{(n-2)\hat\psi}\,d\hat\mm=
\frac12\int_{\hat\sfM}\hat\Gamma'(f)\,d\hat\mm'$$ and
$\hat\Gamma(\,\cdot\,)=e^{-2\hat\psi}\,\hat\Gamma(\,\cdot\,)$,
$\hat\mm'=e^{n\hat\psi}\hat\mm$ where $\hat\psi:=\psi\circ\pi$.
According to \cite{ProfetaSturm18}, based on a detailed approximation property
derived in \cite{Schlichting12}, this space satisfies
${\sf BE}_1(\hat k'\,\hat\mm',\infty)$ with $\hat
k':=k'\circ\pi$. (Actually, this is proven in \cite{ProfetaSturm18} only for
constant $k'$. However, in view of the equivalence of Eulerian and
Lagrangian formulations of variable synthetic Ricci bounds
\cite{Braun-Habermann-Sturm19} and in view of the stability of the
latter \cite{Ketterer2015b}, this easily extends to uniformly bounded,
continuous functions $k'$.)

(iv) Finally, we will conformally re-transform $(\hat\sfM, \hat\sfg')$
with the weight $e^{+2\hat\psi}$ such that
$$\hat\sfg=e^{+2\hat\psi}\,\hat\sfg'.$$
On $\hat\sfM\setminus\partial\sfM$, this of course leads to a smooth
Riemannian structure which (on each of the two copies) coincides with
the original one and with Ricci curvature locally bounded from below
by $\hat k:=k\circ\pi$.  To provide a global estimate, valid also on
$\partial\sfM$, from now on we will restrict ourselves to the
2-dimensional case.

In this case, the initial conformal transformation is just a time
change and the conformal re-transformation is a time-re-change as
considered in \cite{Sturm2019}. Following the argumentation there ---
now with the doubled Dirichlet space in place of the reflected
Dirichlet space --- we conclude from \cite[Theorem 6.7]{Sturm2019} that
the Dirichlet space $\big(\hat\sfM,\hat\cE_\sfM,\hat\mm\big)$
satisfies ${\sf BE}_1(\hat\kappa,\infty)$ with
\begin{equation}\label{glu-ric-psi}
\hat\kappa:=\hat k\, \hat\mm+ \hat{\underline\Delta}\hat\psi\big|_{\partial\sfM}
\end{equation}
with
$\hat{\underline\Delta}\hat\psi\big|_{\partial\sfM}:=\hat{\underline\Delta}\hat\psi-
\big(\hat\Delta\hat\psi\big)\,\hat\mm\big|_{\hat\sfM\setminus\partial\sfM}$. Here
$\hat{\underline\Delta}$ denotes the distributional Laplacian
acting on $\hat\psi\in\fD(\hat\cE_\sfM)$ whereas
$\hat\Delta\hat\psi=(\Delta\psi)\circ\pi$ denotes the
continuous function on $\hat\sfM\setminus\partial\sfM$ obtained by applying the
operator $\hat\Delta$ locally to $\hat\psi$ (or $\Delta$ locally on
$\sfM_0$ to $\psi$).

Note that for all $f\in \fD(\hat\cE_\sfM)$ with $f^\pm\in\fD(\cE_\sfM)$, $f^\pm$ being defined as in the previous Section 7.2,
\begin{eqnarray*}
\langle\hat{\underline\Delta}\hat\psi\big|_{\partial\sfM}, f\rangle&=&
-\int_{\hat\sfM\setminus\partial\sfM}\big[\Gamma(\hat\psi,f)+\hat\Delta\hat\psi\cdot f\big]d\,\hat\mm\\
&=&
-\int_{\sfM_0}\big[\Gamma(\psi,f^++f^-)+\Delta\psi\cdot (f^++f^-)\big]d\,\mm\\
&=&\langle\bar{\underline\Delta}\psi\big|_{\partial\sfM},f^++f^-\rangle.
\end{eqnarray*}
Hence, the distribution
$\hat{\underline\Delta}\hat\psi\big|_{\partial\sfM}$ can be identified
with
$\bar{\underline\Delta}\psi\big|_{\partial\sfM}:=\bar{\underline\Delta}\psi-
\big(\Delta\psi\big)\,\mm\big|_{\sfM_0}$ where
$\bar{\underline\Delta}$ denotes the distributional Neumann
Laplacian acting on $\psi\in\fD(\cE_\sfM)$.

(v) So far, the estimate in \eqref{glu-ric-psi} depends on $\epsilon$
(via the $\epsilon$-dependence of $\psi$). However, we can get rid of
this (and other) dependencies and ambiguities.  Again following the
argumentation in \cite{Sturm2019} with the doubled Dirichlet space in
place of the reflected Dirichlet space, we conclude from Theorem 6.14
there (and its proof) that $\big(\hat\sfM,\hat\cE_\sfM,\hat\mm\big)$
indeed satisfies ${\sf BE}_1(\hat\kappa,\infty)$ with
\begin{equation*}
\hat\kappa:=\hat k\, \hat\mm+ \ell\,\sigma\end{equation*}
where $\sigma$ denotes the surface measure of ${\partial\sfM}$
(\cite[Example 6.12]{Sturm2019}).
  \end{proof}

\bibliographystyle{siam}
\bibliography{biblio}

\end{document}